\documentclass[11pt]{article}
\usepackage {amssymb} %, adjustbox, enumerate, amsbsy}
\usepackage {amsmath}
\usepackage[mathscr]{eucal}
\usepackage{amsthm}
\usepackage{graphicx}
\usepackage {amscd}
\usepackage{subfigure}
\usepackage{stmaryrd}
\usepackage{pifont}
\usepackage{color}
\usepackage{multirow}
\usepackage{tabu}
\usepackage[titletoc, title]{appendix}
\usepackage{diagbox}
\usepackage[all,color]{xy}
\usepackage[colorlinks, citecolor=red]{hyperref}

\usepackage{geometry}  

\newcommand{\vol}{{\rm vol}}
\newcommand{\ord}{{\rm ord}}
\newcommand{\fm}{\mathfrak{m}}
\newcommand{\fa}{\mathfrak{a}}
\newcommand{\cO}{\mathcal{O}}
\newcommand{\bR}{\mathbb{R}}
\newcommand{\bC}{\mathbb{C}}
\newcommand{\bZ}{\mathbb{Z}}
\newcommand{\lct}{{\rm lct}}
\newcommand{\wt}{{\rm wt}}

\newcommand{\Val}{{\rm Val}}

\newcommand{\hvol}{{\widehat{\rm vol}}}

\newcommand{\cF}{{\mathcal{F}}}

\newcommand{\fb}{{\mathfrak{b}}}

\newcommand{\bQ}{{\mathbb{Q}}}

\newcommand{\cX}{{\mathcal{X}}}

\newcommand{\cL}{{\mathcal{L}}}

\newcommand{\cR}{{\mathcal{R}}}
\newcommand{\cY}{{\mathcal{Y}}}

\newcommand{\cW}{{\mathcal{W}}}

\newcommand{\cG}{{\mathcal{G}}}

\newcommand{\bP}{{\mathbb{P}}}

\newcommand{\bin}{{\bf in}}

\newcommand{\cS}{{\mathcal{S}}}

\newcommand{\cB}{{\mathcal{B}}}

\newcommand{\sddb}{{\sqrt{-1}\partial\bar{\partial}}}
\newcommand{\cD}{{\mathcal{D}}}
\newcommand{\vphi}{{\varphi}}

\newcommand{\KE}{{\rm KE}}

\newcommand{\cE}{{\mathcal{E}}}

\newcommand{\mult}{{\rm mult}}
\newcommand{\bA}{{\mathbb{A}}}

\newcommand{\Fut}{{\rm Fut}}

\newcommand{\til}{\tilde}
\newcommand{\reg}{{\rm reg}}
\newcommand{\NA}{{\rm NA}}
\newcommand{\im}{{\rm Im}}

\newcommand{\Hom}{{\rm Hom}}
\newcommand{\la}{\langle}
\newcommand{\ra}{\rangle}
\newcommand{\bN}{\mathbb{N}}
\newcommand{\bD}{\mathbb{D}}

\newcommand{\Red}{\textcolor{red}}
\newcommand{\Blue}{}

\newcommand{\Cyan}{\textcolor{cyan}}

\newcommand{\HH}{\mathbb{H}}

\newcommand{\GG}{\mathbb{G}}
\newcommand{\PP}{\mathbb{P}}

\newtheorem{thm}{Theorem}[section]

\newtheorem{lem}[thm]{Lemma}
\newtheorem{def-prop}[thm]{Definition-Proposition}
\newtheorem{cor}[thm]{Corollary}
\newtheorem{claim}[thm]{Claim}
\newtheorem{defn}[thm]{Definition}

\newtheorem{prop}[thm]{Proposition}

\newtheorem{rem}[thm]{Remark}
\newtheorem{exmp}[thm]{Example}

\begin{document}

\title{Algebraicity of the Metric Tangent Cones and Equivariant K-stability}
\author{Chi Li, Xiaowei Wang and Chenyang Xu}
\date{}

\maketitle

\abstract{We prove two new results on the K-polystability of $\bQ$-Fano varieties based on purely algebro-geometric arguments. The first one says that any K-semistable log Fano cone has a special degeneration to a uniquely determined K-polystable log Fano cone. As a corollary, we combine it with the differential-geometric results to complete the proof of Donaldson-Sun's Conjecture which says that the metric tangent cone of any  point appearing on a Gromov-Hausdorff limit of K\"ahler-Einstein Fano manifolds depends only on the algebraic structure of the singularity.
The second result says that for any log Fano variety with a torus action, K-polystability is equivalent to equivariant K-polystability, that is, to check K-polystability, it is sufficient to check special test configurations which are equivariant under the torus action.
 % (see \cite[Conjecture 3.22]{DS17}). Built on the results in \cite{LX17}, this is achieved by giving a complete description of the degenerations of K-semistable objects with a torus action to K-polystable ones, using only algebro-geometric method.   
}

%\vskip -12mm
\setcounter{tocdepth}{2}
\tableofcontents

\section{Introduction}
We work over the field $\mathbb{C}$ of complex numbers. This paper is a sequel to the works in \cite{Li17, LX16, LX17}. Together with the previous works, we complete the proof of Donaldson-Sun's Conjecture \cite[Conjecture 3.22]{DS17} (see Theorem \ref{t-DSconj}), which says that as an affine variety, the metric tangent cone $C:=C_o(M_\infty, d_\infty)$ of any point $o$ on a Gromov-Hausdorff (GH) limit $(M_{\infty}, d_\infty)$ of a sequence of K\"ahler-Einstein Fano manifolds only depends on the algebraic structure of the singularity and is independent of the metric structure. Previously in \cite{LX17} we proved that the intermediate semistable cone $W$ in Donaldson-Sun's work (see \cite{DS17}) only depends on the algebraic structure. 

Our strategy %, as  first proposed in \cite{Li18}, 
is to systematically use minimizers of the normalized volume functional (defined in \cite{Li18}) to characterize valuations associated to metric tangent cones. Aiming at a vast generalization of the original differential geometric approach, we try to algebraize the construction of  \cite{DS17} by giving a completely local definition of a two-step degeneration process for an arbitrary klt singularity. This has been done under suitable assumptions about the minimizer of the normalized volume. In fact, these assumptions yield the first step of the degeneration and our current note draws a complete picture of the second step in the degeneration. In particular, with the help of the metric structures, we now have a rather satisfactory understanding of this process for those singularities appearing on the GH-limit $M_{\infty}$.  We will give more details in the following discussion.

\subsection{Main results}

For the first step of the degeneration, in  \cite{LX17}, we showed that the valuation considered in \cite{DS17}, whose original definition depends on the metric, is  a minimizer of the normalized volume (see \ref{def-hvol}) and such a minimizer is {\em uniquely} determined by the underlying algebraic structure. In fact, we proved in \cite{LX17} that for any klt singularity $(X, x)$, the real valuation $v\in \Val_{X,x}$ that minimizes the normalized volume functional and satisfies the following two conditions is unique up to rescaling:  {(a)} $v$ is quasi-monomial; {(b)} $v$ has a finitely generated associated graded ring. These two conditions are satisfied by the valuation constructed via K\"{a}hler-Einstein 
metric structure in \cite{DS17}.
So this result allows us to recover the semistable cone $W$, which is defined by the associated graded ring of $v$, by using the minimizing valuation and hence verifies the first part of  \cite[Conjecture 3.22]{DS17}. 
Note that it was conjectured in \cite{Li17} that minimizing valuations always satisfy these two conditions. \footnote{After the submission of this paper, the quasi-monomial property has been proved in \cite{Xu20}. More recently, the uniqueness of minimizers (up to rescaling) is proved in \cite{XZ20} unconditionally. }

We know that the semistable cone $W$ degenerates to the metric tangent cone $C$ and is K-semistable (see \cite{DS17} and \cite[Theorem 5.5]{LX17}). In the current paper, we complete the picture by showing that the metric tangent cone $C$  %, which 
is the unique K-polystable degeneration of $W$.  In particular,  this implies that $C$  depends only on the algebraic structure of $W$, which itself only depends on the algebraic structure of $o\in M_{\infty}$. %We refer \cite{DS17} and \cite[Part 2]{LX17} for more background.

\begin{thm}[{\cite[Conjecture 3.22]{DS17}}]\label{t-DSconj}
The metric tangent cone $C$ of $o\in M_{\infty}$ on a GH-limit of K\"ahler-Einstein Fano manifolds depends only on the algebraic  %(but not the metric) 
structure of $o\in M_{\infty}$.
\end{thm}
%Theorem \ref{t-Kpoly} extends the result in \cite{Ber15} to irregular case and improves \cite[Theorem 1.5]{LX16}.

As in \cite{DS17}, the assumption on $M_{\infty}$ can be  weakened, e.g. $M_\infty$ is  a GH-limit of a sequence of projective manifolds $X$ with fixed volumes, bounded Ricci curvature and diameter. All arguments  extend verbatim. One can expect Theorem \ref{t-DSconj} will significantly simplify the determination of metric tangent cones in examples (see e.g. \cite{HS16}).

Since  a Fano cone singularity $(C,\xi)$ with a Ricci-flat K\"ahler cone metric is aways K-polystable (see \cite[Theorem 7.1]{CS15} and also Corollary \ref{cor-RF2K}), once knowing that $W$  depends only on the algebraic structure of $o\in M_{\infty}$, Theorem \ref{t-DSconj} is just a consequence of the following more general result by letting $(X,D, \xi_0)=(W, \emptyset, \xi_0)$: 

\begin{thm}[{Existence and uniqueness of K-polystable degenerations: log Fano cones}]\label{t-uniquecone}
Given a K-semistable log Fano cone singularity $(X,D,\xi_0)$, there always exists a special test configuration $(\mathcal{X},\mathcal{D}, \xi_0; \eta)$ that degenerates $(X,D,\xi_0)$ to a K-polystable log Fano cone singularity $(X_0,D_0,\xi_0)$. Furthermore, such $(X_0,D_0,\xi_0)$ is uniquely determined by $(X, D, \xi_0)$ up to isomorphism.
\end{thm}

If we restrict ourselves to the quasi-regular case of log Fano cones, then we obtain the following result for log Fano varieties.

\begin{thm}[{Existence and uniqueness of K-polystable degenerations: log Fano varieties}]\label{t-uniqueness}

Given a K-semistable log Fano variety $(S,B)$, there always exists a special test configuration $(\mathcal{S},\mathcal{B})$ that degenerates $(S, B)$ to a K-polystable pair $(S_0,B_0)$. Furthermore, such log Fano pair $(S_0,B_0)$ is uniquely determined by $(S, B)$ up to isomorphism.
\end{thm}

We note that for the special case of $\bQ$-Gorenstein smoothable Fano varieties, this was proved in \cite[7.1]{LWX14} based on an analytic results on the existence and uniqueness of Gromov-Hausdorff limit for a flat family of Fano K\"ahler-Einstein manifolds (see also \cite{SSY16}).
We emphasize here that our proof of Theorem \ref{t-uniquecone} is new and  uses only algebro-geometric arguments. Moreover, our techniques also give rise to an equivariant criterion for testing K-polystability.

\begin{thm}[{$T$-equivariant K-stability=K-stability}]\label{t-equivariant}
Let $(S,B)$ be a log Fano variety with an action by a torus group $T\cong (\mathbb{C}^*)^d$. Then $(S,B)$ is K-polystable if and only if it is $T$-equivariantly K-polystable, that is for all $T$-equivariant special test configuration $(\mathcal{S},\mathcal{B})$, the generalized Futaki invariant ${\rm Fut}(\mathcal{S},\mathcal{B})\ge 0$, and the equality holds only when the test configuration is a product, i.e. $(\mathcal{S},\mathcal{B})\cong (S,B)\times \mathbb{A}^1$.
\end{thm}

Note that Theorem \ref{t-equivariant} is proved for smooth Fano manifolds  in \cite{DaSz16} for general reductive group actions using analytic approach. 
Our result works for any singular $\bQ$-Fano varieties. This combined the work \cite{IS17} allows one to effectively check the K-stability of $\bQ$-Fano $T$-varieties of complexity one. \footnote{After the submission of the paper, the equivalence between equivariant K-stability and K-stability has been completely solved in \cite{Zhu20}.}

%In the below, we will briefly explain the idea of the proof.  As we mentioned, our argument will be completely algebraic.  

%\subsection{Idea of the proof}
%Now we give some idea in the proof of Theorem \ref{t-uniqueness}, which concerns the case of log Fano varieties corresponding to the case of quasi-regular case of log Fano cones. More explanation will appear in section \ref{sec-common1}. % and our result is new even in this case. 
%In \cite{LX16,LX17}, we have developed techniques to study K-semistability type question for equivariant degenerations of objects with an action by a torus group $T$. In this paper, we will use similar techniques to prove statements which say that among all degenerations of a  $T$-equivariant object, the most stable one, if exists, should automatically admit a $T$-action. 
%We give a sketch of the main ingredients. 
We will first deal with the quasi-regular case i.e. Theorem \ref{t-uniqueness}. 
The key technical result in its proof is Theorem \ref{t-doubledeg}, which says that if 
$(\cS^{(i)}, \mathcal{B}^{(i)}) (i=1,2)$ are two special test configurations of the log Fano pair $(S, B)$ with central fibres $(S^{(i)}_0, B^{(i)}_0)$ and vanishing Futaki invariants, then
there exist special test configurations $(\cS'^{(i)}, B'^{(i)})$ of $(S^{(i)}_0, B^{(i)}_0)$ such that $(\cS'^{(i)}, B'^{(i)})$ have isomorphic central fibers.  In practice, we will work on the cones: we first take the cone over $(S, B)$ to get a log Fano cone $(X, D)$, and then take cones over $(\cS^{(i)}, \cB^{(i)})$ to get test configurations $(\cX^{(i)}, \cD^{(i)})$ of $(X, D)$. Then we just need to find a common degeneration of the central fibre of $(X^{(i)}_0, D^{(i)}_0)$. To construct such test configurations, we aim at constructing a $(\mathbb{C}^*)^2$-equivariant family $(\mathfrak X, \mathfrak D)$ of log Fano cones over $\bC^2$, such that if taking the base change over $\bC\times \{1\}$ (resp $\{1\}\times \bC$), we get back the test configuration $(\cX^{(1)},\cD^{(1)})$ (resp. $(\cX^{(2)},\cD^{2})$). Then the special fiber $(\mathfrak X, \mathfrak D)\times_{\mathbb{C}^2}\{(0,0)\}$ gives a common degeneration of $(X^{(i)},D^{(i)})$. 
This family $(\mathfrak X, \mathfrak D)$ is obtained by using a divisor $\cE_k^{(2)}$ over $(\cX^{(2)}, \cD^{(2)})$ to degenerate $(\cX^{(2)}, \cD^{(2)})$, where the divisor $\cE_k^{(2)}$ is a birational transform of $E_k\times\bC$ and $E_k$ belongs to a sequence of special divisors $\{E_k\}$ $(k\ge 1)$ over $X$. To see $\cE_k^{(2)}$ induces a degeneration, one needs to show that there is a birational model $\cY_k^{(2)}\to (\cX_k^{(2)},\cD_k^{(2)})$ such that the only exceptional divisor of $\cY_k^{(2)}/\cX_k^{(2)}$ is $\cE_k$ for $k\gg 1$. This  is a subtle property, and we prove it by combining the tools of Minimal Model Program (MMP) with a careful analysis of normalized volumes functional (see Section \ref{sec-common1} for a more detailed explanation for this step). 
%\Cyan{
%To achieve this, we first observe that $(\cX^{(1)}, \cD^{(1)})$ is induced by a sequence of special divisorial valuations over the vertex of $X$. Then we can reduce the problem to the extraction of the single divisor $E_k\times \bC$ over $\cX^{(2)}$. 
%To extract this divisor, by the celebrated works of \cite{BCHM}, it suffices to construct a graded sequence of ideal sheaves on the total space $\cX^{(2)}$ that satisfies appropriate conditions. By carefully analyzing the ``almost minimization" property of $\ord_{E_k}$ It turns out the such ideal sheaves are naturally given by the $\bC^*$-equivariant degeneration of \blue{valuation ideals} of $\ord_{E_k}$. 

The following commutative diagram shows some relation between different objects in this proof. The symbols `$\rightsquigarrow$' means the degeneration under a special test configuration and `$\dashrightarrow$' means taking a $\bC^*$-quotient. See \eqref{eq-CD2} for the more detailed diagram.
\begin{equation}
\xymatrix@1 @R=1.1pc @C=0.9pc
{
(X^{(2)}_0, D^{(2)}_0) \ar@{~>}_{(\cX'^{(2)}, \cD'^{(2)})}[dddd]
 \ar@{-->}[dr] &   & &  &(X, D) \ar_{ (\cX^{(2)}, \cD^{(2)})}@{~>}[llll] \ar^{(\cX^{(1)}, \cD^{(1)})}@{~>}[dddd]  \ar@{-->}[ld] & 
 \\
& (S^{(2)}_0,B^{(2)}_0) \ar@{~>}_{(\cS'^{(2)},\mathcal{B}'^{(2)})}[dd] &  &  (S, B) \ar_<<<<<<<{(\cS^{(2)}, \mathcal{B}^{(2)})}@{~>}[ll] \ar^{(\cS^{(1)}, \mathcal{B}^{(1)})}@{~>}[dd] & & \\
&   &   &\\
&                              (S'_0, B'_0)     &               &(S^{(1)}_0, B^{(1)}_0) \ar@{~>}^{\stackrel{}{(\cS'^{(1)},\mathcal{B}'^{(1)})}}[ll]& & \\
(X'_{0}, D'_{0}) \ar@{-->}[ru] & & & & (X^{(1)}_0, D^{(1)}_0)\ar@{~>}^{(\cX'^{(1)}, \cD'^{(1)})}[llll] \ar@{-->}[lu] & .
}
\end{equation}

\medskip

To confirm Donaldson-Sun's Conjecture (see \cite[Conjecture 3.22]{DS17}), we then use some approximation approach to treat the case of a general log Fano cone, i.e. including the irregular case.
%that the metric tangent cone $C$ of a klt singularity $(o\in M_{\infty})$ appearing on a GH limit $M_{\infty}$ of K\"ahler-Einstein Fano manfiolds is uniquely determined by the algebraic germ structure of $(o\in M_{\infty})$, 
%first establish the condition \eqref{e-extract} using approximation approach to get common degenerations as in the quasi-regular case. 
However, the common degenerations are a priori only {\it weakly special} (Definition \ref{d-wstc}). So we extend \cite[Theorem 4]{LX14}  to (possibly irregular) log Fano cones, proving that for K-polystable log Fano cones, weakly special test configurations with vanishing Futaki invariants must already be special. %In other words, we reproduce the last step of \cite{LX14}  for a log Fano cone (see Section \ref{ss-special}). %We also apply \cite{CS15} to get the K-polystability of a Ricci-flat K\"{a}hler cone singularity. 
To apply this to metric tangent cones which are Ricci-flat K\"{a}hler cones, we use the result of Colins-Sz\'{e}kelyhidi \cite{CS15} about the K-polystability of Ricci-flat K\"{a}hler cones. Since we need to allow more general test configurations than just special test configurations, we provide a proof of this fact (see Remark \ref{rem-Ding}) in Appendix \ref{s-analytic} by adapting the argument of Berman in \cite{Ber15} to our setting. 
%For that purpose, we will present two approaches.  One is algebraic, in which we reproduce the last step of \cite{LX14}  for a log Fano cone (see Section \ref{ss-special}); and the other is analytic, in which we extend the result in \cite{Ber15, CS15} and treat the $\bQ$-Gorenstein test configurations of Fano cones (see Section \ref{s-analytic}). The argument is  completed based on these results. 

%we put together these two ideas, for the proof of :  Since ${\rm Fut}(\mathcal{X},\mathcal{D})=0$, it implies the derivative of the volume function is 0, then we can confirm that the decreasing of the log canonical  thresholds is sufficiently small.

\medskip

%Indeed one of our motivations of this note is to completely solve In fact, $C$ is constructed by a two-step process of degeneration: first it degenerates to a K-semistable cone $W$, and then $W$ equivariantly degenerates to $C$. 
%We prove the analogue of Theorem \ref{t-uniqueness} in this case. 

We now sketch the organization of the paper. More details will be given at the beginning of each section. In 2.1, we recall basic tools in our arguments including normalized volumes, normalized multiplies and Koll\'{a}r components. In section 2.2, we recall the notions of log Fano cones, their test configurations and K-stability. We also discuss how to get test configurations using models over log Fano cones. In the quasi-regular case, we are reduced to the K-stability of log Fano pairs. In section 3, we prove our main results in the case of log Fano pairs. In section 3.1, we prove a lemma about special degenerations of K-semistable log Fano pairs with zero Futaki invariants. In section 3.2, we prove the main technical result (Theorem \ref{t-doubledeg}) on common special degenerations of special degenerations with zero Futaki invariants. In section 3.3, we finish the proof of main results for log Fano pairs. In section 4, we deal with the general case of log Fano cones. In 4.1, we obtain common weakly special degenerations for log Fano cones with vanishing generalized Futaki invariants. In 4.2, we show that these weakly special test configurations are indeed special test configurations. We generalize the last step of results in \cite{LX16} to the case of log Fano cones. We complete the proof of Theorem \ref{t-uniquecone} and Donaldson-Sun's conjecuture in section 4.3. In the appendix, we prove the analytic result that Ricci-flat K\"{a}hler cones are Ding-polystable among $\bQ$-Gorenstein test configurations. This result could substitute results in section 4.2 to complete the proof of Theorem \ref{t-DSconj}.

\bigskip

\noindent {\bf Acknowledgement:} 
We would like thank Yuchen Liu for pointing out an inaccuracy in the previous version of the preprint and the anonymous referees for many helpful comments. 
CL is partially supported by NSF (DMS-1405936 and DMS-1810867) and an Alfred P. Sloan research fellowship.  XW is partially supported by a Collaboration Grants for Mathematicians from
Simons Foundation:281299 and NSF:DMS-1609335. 
%X. Wang is partially supported by a Collaboration Grants for Mathematicians from Simons Foundation:281299 and NSF:DMS-1609335. Part of the work was done when X. Wang was visiting IHES, he wants to thank the great research environment they provided.
CX is partially supported by `Chinese National Science Fund for Distinguished Young Scholars (11425101)'. Part of the work was done when  XW was
visiting IHES and CX is visiting Institut Henri Poincar\'e (partially sponsored by `the Poincar\'e Chair'), to which they want to thank the inspiring research environment.

\bigskip

\noindent {\bf Notation and Conventiones: } We follow the standard notation in \cite{KM98, Kol13}. In this paper, a {\it variety} is a reduced, separated and finite type scheme over $\mathbb{C}$, that is allowed to be reducible.  We call a pair $(S,B)$ a {\it log Fano} variety if $(S,B)$ has klt singularities, and $-(K_S+B)$ is ample. %We will interchangeably use the notations $\bC$ with $\bA^1$, and $\bC^*$ with $\bG_m$.

%We will use the following terminology introduced in \cite{Kol18}. We consider a $f\colon(X,D)\to C$ where $X$ is normal flat over a smooth projective curve $C$, $D$ is an effective $\bQ$-divisor on $X$.  We say $(X,D)$ is {\it locally stable over} $C$ if $(X,D+X_t)$ is log canonical for any $t\in C$ (see \cite[Definition 2.2]{Kol18}).

\section{Preliminaries}

\subsection{Normalized volumes}\label{ss-normalised}

In this section, we recall the definition of the normalized volume of valuations centered at a klt singularity $x\in (X,D)$. This is introduced in \cite{Li18}. For readers' convenience, we discuss some basic properties which will be needed later. 
\begin{defn}
Let $X={\rm Spec}_{\bC}(R)$ be an irreducible affine variety and $x\in X$ a closed point.
We denote by $\Val_{X,x}$ the space of real valuations $v: R\rightarrow \bR_{\ge 0}\cup\{+\infty\}$ that satisfy the following conditions: for any $f, g\in R$:\\
(1) $v(fg)= v(f)+v(g)$;\hskip 2mm
(2) $v(f+g)\ge \min\{v(f), v(g)\}$; \hskip 2mm
(3) $v(0)=+\infty, v(\bC^*)=0$; \hskip 2mm
(4) $v(f)>0$ if $f(x)=0$.

For any $v\in \Val_{X,x}$ and $m\in \bR$, its $m$-th valuation ideal is defined as $\fa_m(v):=\fa_m(v, X)=\{f\in R; \  v(f)\ge m\}$. 
\end{defn}
 We remark that $\Val_{X,x}$ is also called the `non-archimedean link' around $x\in X$ in some literatures. 

For any $m>0$, $\fa_m(v)$ is a primary ideal associated to the maximal ideal $\fm_x$. We will denote its Hilbert-Samuel multiplicity by $\mult(\fa_m)$. If $\Lambda=v(R)\subset \mathbb{R}_{\ge 0}$ denotes the valuative semigroup of $v$, then $\{\fa_m(v);\  m\in \Lambda\}$ is a $\Lambda$-graded sequence of ideals. In other words, they satisfy, for any $m, m'\in \Gamma$, $(i): \fa_{m'}(v)\subseteq \fa_{m}(v)$ if $m'\ge m$ and $(ii): \fa_m(v)\cdot \fa_{m'}(v)\subseteq \fa_{m+m'}(v)$. Note that $\{\fa_m(v); m\in \bZ\}$ is also a $\bZ$-graded sequence of ideals.
\begin{def-prop}[\cite{ELS03, LM09}]
Let $X$ be an irreducible variety of dimension $n$. 
For any $v\in \Val_{X,x}$, the volume of $v$ is the following well-defined quantity:
\begin{equation}
\vol(v)=\lim_{m\rightarrow+\infty}\frac{\dim_\bC(R/\fa_m(v))}{m^n/n!}=\lim_{m\rightarrow+\infty}\frac{\mult(\fa_m)}{m^n}=:\mult(\fa_\bullet).
\end{equation}
\end{def-prop}
Now we assume $(X,D)$ is a log pair such that $K_X+D$ is $\bQ$-Cartier. For any divisorial valuation $v=\ord_S$ where $S$ is a prime divisor on a normal variety $Y$ with a proper birational morphism $\mu: Y\rightarrow X$, the {\em log discrepancy} of $\ord_S$ is defined as $A_{(X,D)}(\ord_S)=\ord_S(K_Y-\mu^*(K_X+D))+1$. By \cite{JM12} and \cite{BFFU15}, there is a canonical way to extend the log discrepancy to become a lower semicontinuous function $A_{(X,D)}: \Val_{X,x}\rightarrow \bR\cup \{+\infty\}$. 

\begin{def-prop}[{see \cite[Theorem 1.1]{Li18}}]\label{def-hvol}
Assume $x\in (X,D)$ is a klt singularity. 
For any $v\in \Val_{X,x}$, its normalized volume $\hvol_{(X,D,x)}(v)$ is defined as:
\begin{equation}
\hvol_{(X,D,x)}(v)=\left\{
\begin{array}{lcl}
A_{(X,D)}(v)^n\cdot \vol(v),& \text{ if } & A_{(X,D)}(v)<+\infty;\\
+\infty, & \text{ if } & A_{(X,D)}(v)=+\infty.
\end{array}
\right.
\end{equation}
For simplicity, we will just write $\hvol(v)$ if the singularity $x\in(X,D)$ is clear. This quantity is a rescaling invariant: $\hvol(\lambda v)=\hvol(v)$ for any $\lambda>0$.

The volume of a klt singularity $x\in(X,D)$ is defined to be the following positive number
\begin{equation}
\vol(x,X,D)=\inf_{v\in \Val_{X,x}}\hvol_{(X,D,x)}(v).
\end{equation}

\end{def-prop}
It has been shown that there always exists a minimizer $v$ of $\hvol_{(X,D,x)}$ among all $v\in {\rm Val}_{X,x}$ in \cite{Blu16}. The expected properties of the minimizers are formulated in the {\it Stable Degeneration Conjecture} (\cite[Conjecture 6.1]{Li18}, \cite[Conjecture 1.2]{LX17}).  
The case  of cone singularities over Fano varieties was studied in \cite{Li17,LL16}. The general case was systematically studied in \cite{LX16} under the assumption that the minimizer is a divisorial valuation and in \cite{LX17} under the assumption that the minimizer is a higher rank quasi-monomial valuation.

We will need a relation between the normalized volume and the normalized multiplicity of a graded sequence of ideals.
\begin{prop}[\cite{Liu16}]\label{p-liu}
If $x\in (X,D)$ is an $n$-dimensional klt singularity, then we have
$$\vol(x,X,D)=\inf_{\fb_{\bullet}}\  \mult(\fb_{\bullet})\cdot \lct^n(X,D;\fb_{\bullet}),$$
where $\fb_{\bullet}$ runs over all graded sequence of primary ideals cosupported at $x$.  
\end{prop}
We now state some central results from our previous works and refer to the next section for the notations of log Fano cones (see Definition \ref{d-logfanocone}) and their K-stability (see e.g. Definition \ref{d-semicone}). 

\begin{thm}[\cite{Li17, LL16, LX16, LX17}]\label{p-m=k}
Let $ (X, D,\xi)$ be a  log Fano cone singularity. Then it is K-semistable if and only if the valuation $\wt_{\xi}$ induced by $\xi$ is a minimizer of $\hvol_{(X,D,x)}$ on $ \Val_{X,x}$.
\end{thm}
%\begin{proof}For the quasi-regular case, this is proved in \cite{Li17, LL16, LX16}. For the general case, including irregular one, this is proved in \cite[Theorem 3.5]{LX17}. 
%\end{proof}
We will also use the following notion frequently:
\begin{defn}\label{def-Kollar}
Let $(X, D, x)$ be a klt singularity. A prime divisor $S$ over $(X,D,x)$ is called a Koll\'{a}r component over $(X, D, x)$, if there exists a projective birational morphism $\mu: Y\rightarrow X$ such that
(i) $\mu$ is an isomorphism over $X\setminus \{x\}$ and the exceptional divisor $S=\mu^{-1}(x)$ is irreducible and $\bQ$-Cartier; (ii) $-S$ is $\mu$-ample; (iii) $(Y, \mu_*^{-1}D+S)$ is plt.
\end{defn}
For any Koll\'{a}r component, we define its {\it different} $\Delta_S$ by the following equality:
$$K_S+\Delta_S=(K_Y+S+\mu_*^{-1}D)|_S.$$ It is easy to see that $(S, \Delta_S)$ is a log Fano pair.

The relevance of Koll\'ar components to the minimization of normalized volume is contained in the following result:
\begin{thm}[{\cite[Theorem 1.2, 1.3]{LX16}}]
Let $(X, D, x)$ be a klt singularity. Assume that $v_0\in \Val_{X,x}$ is a minimizer of $\hvol_{(X,D,x)}$. Then we can find a sequence of Koll\'ar components $S_k$ and constants $c_k>0$, such that 
\[
\mbox{$c_k\cdot \ord_{S_k}\rightarrow v_0$ \ \ and \ \ $\hvol(\ord_{S_k})\rightarrow \hvol(v)$ as $k\rightarrow+\infty$.}
\]
Moreover, if $v_0$ is divisorial, then $v_0=c \cdot \ord_S$ for some $c>0$ and a Koll\'{a}r component $S$ satisfying the condition that the log Fano pair $(S, \Delta_S)$ is K-semistable.
 \end{thm}
 %\begin{proof}See {\cite[Theorem 1.2 and Theorem 1.3]{LX16}}.
 %\end{proof}
In the above theorem, when $v_0$ is a divisorial minimizer, then \cite{Blu16} also shows it yields a Koll\'ar component.
In the case of K-semistable log Fano cones, the approximation stated in the above theorem can be realized concretely by perturbing the Reeb vector field to rational ones.

\subsection{K-stability of log Fano cones}

In this section, we recall the definition of a log Fano cone singularity and its K-stability, by essentially following \cite{CS12, CS15} and \cite[Section 2.5]{LX17}.
Denote by $T$ a complex torus which is isomorphic to $(\bC^*)^r$. %which is a complexification of $(S^1)^r$. 
\subsubsection*{Test configurations}
\begin{defn}
Let $X$ be an $n$-dimensional reduced affine variety which is not necessarily irreducible. We say that a $T$-action on $X$ is good if it is effective and there is a closed $T$-fixed point $x\in X$ (called the vertex) that is in the closure of any $T$-orbit.
By a $T$-singularity in this paper, we always mean an affine variety $X$ with a good $T$-action. If  $D$ is a $T$-invariant $\bR$-divisor on $X$ %with coefficients of contained in $ D$ in $[0,1]$, 
we say that $(X, D)$ is a pair with a good $T$-action.
\end{defn}
Let $N=\Hom(\bC^*, T)$ be the co-weight lattice and $M=N^*$ the weight lattice. If $X={\rm Spec}_{\bC}(R)$ is a $T$-variety, then there is a weight space decomposition:
\begin{equation}\label{eq-Rweight}
R=\bigoplus_{\alpha\in\Gamma} R_\alpha \text{ where } \Gamma=\{\alpha\in M\; |\; R_\alpha\neq 0\}\subset M.
\end{equation}
The action being good implies $R_0=\bC$. We will call any element $\xi$ in the Lie algebra $N_\bR:=N\otimes \mathbb{R}$ {\it a coweight vector} (or abbreviated as {\it a vector}). We will denote by $\la \xi\ra$ the subtorus  of $T$ generated by $\xi$, i.e. the subtorus corresponding to the minimal linear $\bQ$-linear subspace $V\subset N\otimes \mathbb{Q}$ such that $ V\otimes \bR$ contains $\xi$. 

If $T$ acts on a smooth variety $X$, then $\xi$ will give a vector field on $X$. For example, if we consider the multiplication of $\bC^*$ on $\bC$, then the coweight vector $1\in \bZ$ yields the vector field $t\partial_t$. 
\begin{defn}
The Reeb cone of $X$ with respect to a good $T$-action is the following set:
\begin{equation}
{N_\bR^+}:=\left\{\xi\in N_\bR\; |\; \langle \alpha, \xi\rangle >0 \text{ for any } \alpha\in \Gamma\backslash\{0\}\right\}.
\end{equation}
Any vector $\xi\in {N_\bR^+}$ will be called a Reeb vector on the $T$-variety $X$.
\end{defn}
\begin{defn}\label{def-volxi}
For any $\xi\in {N_\bR^+}$, we define its volume as:
$$\vol_X(\xi):=\vol_{X, x}(\xi)=\lim_{k\to \infty} \frac{\sum_{\langle \xi, \alpha \rangle\le k}{\rm \Blue{\dim_\bC}}(R_{\alpha}) }{ k^n/n! }.$$ 
\end{defn}

The limit in the above definition was known to exist by using the multivariable Hilbert series as in \cite[Proof of Theorem 4.10]{CS12} (see also \cite{MSY08}). 
If $(X,x)$ is a normal affine $T$-variety, then each $\xi\in N_\bR^+$ corresponds to a valuation $\wt_\xi\in \Val_{X,x}$ which is defined as:
\begin{equation}\label{eq-wtxi}
\wt_\xi(f)=\min\left\{ \langle \alpha, \xi\rangle; f=\sum_{\alpha}f_\alpha \text{ with } f_\alpha\neq 0\in R_\alpha \right\}.
\end{equation}
Note that in this case $\vol_X(\xi)$ in Definition \ref{def-volxi} is just the volume of the valuation $\wt_\xi$.

One key property of the volume function is the following. %The proof is essentially contained in \cite[Section 3.2]{LX17}. Here we give a sketch of the argument for the sake of completeness. 

\begin{lem}[{see {\cite{CS12, MSY08, LX17}}}]
The function 
$\xi\mapsto \vol_{X,x}(\xi)$ is smooth and strictly convex on $ {N_\bR^+}$.
\end{lem}
\begin{proof}
The smoothness was proved in \cite[Theorem 4.10]{CS12} where $\vol_X(\xi)$ was interpreted as the leading coefficient of the expansion of the so-called index character, (which also appeared in the earlier work of Martelli-Sparks-Yau (see \cite[(6.10)]{MSY08}).

The strict convexity of $\vol_{X,x}$ follows from \cite[Section 3.2]{LX17}.  In fact, if we let $Y\to X$ be the normalization of $X$,  the $T$-action can be lifted to $Y$.  Denote the preimage of $x$ to be $\{y_i\}_i$, then we know  $y_i$ are on {pairwise} distinct components $Y_i$ of $Y$, and the $T$-action on each $Y_i$ is good. We claim that the following identity holds true:
$$\vol_{X, x}(\xi)=\sum_i \vol_{Y, y_i}(\xi).$$
Indeed, since $\xi\mapsto \vol_{X,x}(\xi)$ and $\xi\mapsto \vol_{Y, y_i}(\xi)$ are continuous, we just need to verify the identity when $\xi$ is rational. Any rational $\xi$ generates a $\bC^*$-action and $\vol(\xi)$ reduces to the degree of an ample orbifold line bundle (see \cite[Proposition 4.3]{CS12} and \cite[5.3]{MSY08}). The identity follows from the fact that the degree of the orbifold line bundle is the sum of its degrees on irreducible components. 

Thus we may assume $X$ to be normal.  Then \cite[Proposition  3.10]{LX17}, which generalizes the convexity result from \cite{MSY08}, says that $\xi\mapsto \vol_{X,x}(\xi)$ is a strictly convex function for $\xi\in {N_\bR^+}$.

%Following the construction in \cite[Section 3.2.2]{LX17}, we choose a general point $z$ on the Chow quotient $Z$ of $X$ by $T$ and then we fix a local coordinate system $z_1,..., z_{n-r}$ around $z\in Z$a quasi-monomial valuation of $K(Z)$ centered over $z$ with rational rank $n-r$. Given a lexicographical order on $M$, 

%by using techniques from the study of Newton-Okounkov bodies, we showed in \cite[Section 3]{LX17} that it is equal to the volume of some bounded convex body cut out by a hyperplane in a strictly convex cone in $\bR^n$. The explicit formula in \cite[Proof of Lemma 3.8]{LX17} for its volume shows that it is a smooth and strictly convex function in $\xi\in \blue{N_\bR^+}$.

\end{proof}

\begin{defn}[Log Fano cone singularity]\label{d-logfanocone}
Let $(X, D)$ be an affine pair with a good $T$ action. Assume $(X,D)$ is normal with  klt singularities. Then for any $\xi\in {N_\bR^+}$,  we call the triple $(X,D,\xi)$ a {\it log Fano cone} structure that is polarized by $\xi$.  If $\la \xi\ra\cong \bC^*$ which is equivalent to saying that $\xi$ is a multiple of a vector in ${N_\bQ^+}$, then we call $(X, D, \xi)$ quasi-regular. Otherwise, we call it irregular.
\end{defn}

\begin{defn}[Quotient in the quasi-regular case]\label{d-quotient}
In the quasi-regular case, we can take the quotient $(S,B)$ of $(X\setminus\{x\},D\setminus\{x\})$ by the $\mathbb{C^*}$-group $\la \xi \ra$ generated by $\xi$ in the sense of a Seifert $\bC^*$-bundle, and we will denote by $(X,D)/\langle \xi \rangle$. More precisely, assume $\xi \in \frac{1}{l}N$, and we write 
\begin{eqnarray*}
R=\bigoplus_{k=0} \left(\bigoplus_{\la \xi, \alpha\ra=k/l} R_{\alpha}\right):=\bigoplus_{k=0} R^{\xi}_{k}.
\end{eqnarray*}
Then we take $S={\rm Proj}(\bigoplus_{k=0} R^{\xi}_{k})$. By \cite[Section 4]{Kol04}, $\pi\colon X\setminus \{x\}\rightarrow S$ is a Seifert $\bC^*$-bundle, with the quotient $X\setminus\{x\}\to (S,B_1)$ where $B_1$ is the branch divisor. Write $D=\sum_i a_iD_i$. Since each $D_i$ is $\bC_m$-invariant, $D_i$ is the pull back of a divisor $E_i$ on $S$ and the multiplicity of $D_i$ along $\pi^*(E_i)$ is denoted by $m_i$. Define $B_2=\sum_i\frac{a_i}{m_i}E_i$. Let $B=B_1+B_2$.  Then $\pi^*(K_S+B)=(K_X+D)|_{X\setminus \{x\}}$ since $\pi^*(K_X+B_1)=K_{X\setminus \{x\}}$ (see \cite[Corollary 41]{Kol04}) and $\pi^*(B_1)=D|_{X\setminus \{x\}}$.

The quotient $(S,B)$ is a log Fano variety, because we assume that $(X,D)$ is klt at $x$ (see \cite[42]{Kol04} or \cite[Lemma 3.1]{Kol13}). 
\end{defn}
\begin{defn}[Test configuration]\label{defn-QGTC}
Let $(X, D, \xi_0)$ be a log Fano cone singularity with a torus group $T$ action (see Definition \ref{d-logfanocone}). A $T$-equivariant test configuration (or simply called a test configuration) of $(X, D, \xi_0)$ is a quadruple $(\cX, \cD, \xi_0; \eta)$ with a map $\pi: (\cX, \cD)\rightarrow\bC$ satisfying the following conditions:
\begin{enumerate}
\item[(1)]
$\cX$ is an affine variety and 
$\pi\colon \cX\to \bC$ is a flat family. 
% the fibers away from $0$ are isomorphic to $(X, D)$ 
% (thus $\cR$ is a finitely generated flat $\bC[t]$-algebra), 
 $\cD$ is a divisor on $\cX$ with ${\rm Supp}(\cD)$ not containing any component of {a} fiber of $\pi$. 
\item[(2)] $\eta$ is a holomorphic vector field that generates a $\bC^*$-action on $(\cX, \cD)$ such that $\pi$ is $\bC^*$-equivariant and $\pi_*\eta=-t\partial_t$. As a consequence, there is an isomorphism $\phi\colon (\cX, \cD)\times_{\bC}\bC^*\cong (X, D)\times \bC^*$.
\item[(3)]
The torus $T$ acts on $(\cX, \cD)$ fiberwise and commutes with the $\bC^*$-action generated by $\eta$, and coincides with the action on the first factor when restricted to $(\cX, \cD)\times_{\bC}\bC^*\stackrel{\phi}{\cong} (X, D)\times \bC^*$. % , and there is an isomorphism $\phi\colon  (\cX, \cD)\times_{\bC}\bC^*\cong (X, D)\times \bC^*$.
 %The torus $T$ acts on $(\cX,\cD)\to \mathbb{C}$ \blue{fiberwise}, and coincides with the action on the first factor when restricted to 
%\item[(2)]
 %A $\bC^*$-action on $(\cX, \cD)$ such that $\pi$ is $\bC^*$-equivariant where $\bC^*$ acts on the base $\bC$ by the multiplication  and $\phi: (\cX, \cD)\times_{\bC}\bC^*\cong
%(X, D)\times \bC^*$ is $\bC^*$-equivariant. We denote the coweight vector of $\bC^*$ by $\eta$.
%\item[(4)]
%The holomorphic vector field $\xi_0$ on $\cX\times_{\bC}\bC^*$ (via the isomorphism $\phi$) extends to a holomorphic vector field 
%The torus $T$-action commutes with the $\mathbb{C}^*$-action. %We assume $(X_0, D_0)$ is also a $T$-variety with a good $T$-action generated by $\xi_0$, where $\xi_0$ is in the Reeb cone of the central fiber $X_0$. Moreover there is an injective homomorphism $\frak{v}: \bC\rightarrow \cX$ mapping $t\in \bC$ to the vertex point $o_t\in X_t=\pi^{-1}(\{t\})$.
\end{enumerate}

A test configuration $(\cX, \cD, \xi_0; \eta)$ is called a product one if there is a $T$-equivariant isomorphism $(\cX, \cD)\cong (X, D)\times \bC$ and $\eta=\eta_0-t\partial_t$ where $\eta_0$ is a coweight vector of $T$ and $t\partial_t$ is the canonical lifting of $t\partial_t$ on $\bC$ through the second projection. In this case, we will denote $(\cX, \cD, \xi_0; \eta)$ by $$(X\times\bC, D\times \bC, \xi_0; \eta)=:(X_\bC, D_\bC, \xi_0; \eta).$$

A normal test configuration $(\cX, \cD, \xi_0; \eta)$ is called $\bQ$-Gorenstein if $K_{\cX}+\cD$ is $\bQ$-Cartier.
\end{defn}
According to the above definition, a test configuration $(\cX, \cD, \xi_0; \eta)$ of the log Fano cone $(X, D, \xi_0)$ is a  $\tilde{T}:=T\times\bC^*$-equivariant degeneration of $(X, D)$ where the $\bC^*$-action is generated by $-\eta$. If it is not a product test configuration, then its central fibre $X_0$ admits an effective $\tilde{T}$-action, whose Lie algebra is generated by the Lie algebra of $T$ and $\eta$.

Moreover if we assume $\cX={\rm Spec}(\cR)$ and decompse $\mathcal{R}=\bigoplus_{\alpha}\mathcal{R}_{\alpha}$ into weight spaces with respect to the fiberwise $T$-action, condition (1) in the above definition implies that each weight piece $\cR_{\alpha}$ is a flat $\bC[t]$-module. As a consequence $X$ and $X_0$ have the same weight cone and Reeb cone with respect to the fiberwise $T$-action. In particular, $\xi_0$ is contained in the Reeb cone of $X_0$ under the $T$-action. 
\begin{rem}
Any test configuration can be $T\times\bC^*$-equivariantly embedded into $\bC^N\times\bC$ (for $N\gg 1$) and our definition is the same as the definition given in \cite[Definition 5.1]{CS12}.
%The identity $\pi_*(\eta)=-t\partial_t$ in Definition \ref{defn-QGTC} means that under the $\bC^*$-action generated by $\eta$, the limit cycle of $X_t=\pi^{-1}(\{t\})$ in $\bC^N$ is $X_0$. 
The choice of sign in the  identity $\pi_*(\eta)=-t\partial_t$ in Definition \ref{defn-QGTC} is  compatible with our later arguments and calculations. 
\end{rem}

%\begin{rem}

%Our Definition \ref{defn-QGTC}, by the argument in {\em \cite[Section 6]{LX16}}, implies that there exists an embedding $$(x\in X)\subset (0\in \bC^N),$$ such that $\mathcal{X}$ is obtained by
%a one parameter group of the ambient space $0\in \bC^N$ which also corresponds to a weighted blow up of $\mathbb{C}^N$. 
%The latter was used in the definition in {\em \cite{CS12}}. So our definition of test configurations is indeed equivalent to the one in {\em \cite{CS12}}. %, but we forget the ambient $\bC^N$. 
%\end{rem}
%Any $\xi\in \ft_{\bR}^{+}$ corresponds to a quasi-monomial valuation $\wt_{\xi}$ on $\cX$. So we can consider the log discrepancy $A_{(\cX, \cD)}(\wt_\xi)=:A(\xi)$. 
%For $|\epsilon|\ll 1$, if we denote $\xi_\epsilon=\xi_0+\epsilon \eta$, then $\xi_\epsilon$ corresponds to a quasi-monomial valuation $\wt_{\xi_\epsilon}$ on $\cX$. So we can consider the log discrepancy $A_{(\cX, \cD)}(\wt_{\xi_\epsilon})=:A(\xi_\epsilon)$.  
%We claim that $A(\wt_{\xi_\epsilon})$ is linear in $\epsilon$. 
%In fact, 
\Blue{
Because $K_{\cX}+\cD$ is $\bQ$-Cartier, by the structure theory of $T$-varieties, there exists a $T\times\bC^*$-equivariant nowhere-vanishing section $s\in |m(K_\cX+\cD)|$ (see 
%\cite[Theorem 3.21]{PS08} and 
\cite[Proposition 4.4]{LS13}, and also \cite[2.7]{MSY08}). For any $b\in \bR$ and $\xi+b\eta\in %\ft_{\bR}\oplus \bR \eta\cong 
N_{\bR}\oplus \bR$, define:
$$
A(\xi+b\eta):=\frac{1}{m}\frac{\cL_{\xi+b\eta}s}{s},
$$
where $\cL_{\xi+b\eta}$ is the Lie derivative of $s$ with respect to the vector field associated to $\xi+b\eta$. Note that this is a linear function.
}
%Then by \cite[Lemma 2.18]{LX17} we have 
%$$A_{(\cX, \cD)}(\wt_{\xi_\epsilon})=\frac{1}{m}\frac{\cL_{\xi_\epsilon}s}{s},$$ which shows that $A(\xi_\epsilon)$ is linear in $\epsilon$. 

If $(\cX, \cD, \xi_0; \eta)$ is any $\bQ$-Gorenstein test configuration of an $n$-dimensional log Fano cone $(X, D, \xi_0)$, we will denote:
\begin{equation}\label{eq-normeta}
T_{\xi_0}(\eta)=\frac{A(\xi_0)\eta-A(\eta)\xi_0}{n}.
\end{equation}

\begin{defn}[Weakly special and special test configurations]\label{d-wstc}
In the notations we used before, we define a weakly special test configuration (resp. special test configuration) of $(X, D, \xi_0)$ to be a $\bQ$-Gorenstein test configuration $(\cX, \cD, \xi_0; \eta)$ with central fiber $(X_0, D_0)$ satisfying that:
\begin{enumerate}
\item[(4)] $(\cX,\cD+X_0)$ has log canonical singularities (resp. $(X_0, D_0)$ has klt singularities). 
\end{enumerate} 
In this case, we say that $(X_0, D_0)$ is a weakly special degeneration (resp. special degeneration) of $(X, D)$.
%\item 
%A special test configuration of $(X, D, \xi_0)$ is a $\bQ$-Gorenstein test configuration $(\cX, \cD, \xi_0; \eta)$ with central fiber $(X_0, D_0)$ satisfying that:
%\begin{enumerate}
%\item[(4')] $(X_0, D_0)$ has klt singularities. % and $(X_0, D_0, \xi_0|_{X_0})$ is a log Fano cone singularity.
%\end{enumerate} 
%In this case, we say that $(X_0, D_0)$ is a special degeneration of $(X, D)$.
\end{defn}
Note that by inversion of adjunction, being special implies being weakly special.

%We need also consider a larger class of test configurations than special ones.
%\begin{defn}[Weakly special test configuration]\label{d-wstc}
%Notation as above. 
 %A weakly special test configuration of $(X, D, \xi_0)$ is a $\bQ$-Gorenstein test configuration $(\cX, \cD, \xi_0; \eta)$ with central fiber $(X_0, D_0)$ satisfying moreover that:
%\begin{enumerate}
%\item[(5)] $(\cX,\cD+X_0)$ has log canonical singularities. % and $(X_0, D_0, \xi_0|_{X_0})$ is a log Fano cone singularity.
%\end{enumerate} 
%In this case, we say that $(X_0, D_0)$ is a weakly special degeneration of $(X, D)$.
%\end{defn}

For simplicity, we will just say that $(\cX, \cD)$ is a $\bQ$-Gorenstein (or weakly special, special) test configuration if $\xi_0$ and $\eta$ are clear. We also say that $(X,D,\xi_0)$ degenerates to $(X_0,D_0,\xi_0)$ (or simply to $(X_0, D_0)$).

\subsubsection*{Test configuration and filtration}

In \cite[Section 2.5]{BHJ15}, a filtration viewpoint for test configurations is developed. Here we will mainly work with data over the vertex of the cone which brings more flexibility when applying the minimal model program.  In this section, we will discuss these ideas and modify them to fit into our context. 
\begin{lem}\label{l-filtration} Given a normal $T$-equivariant  test configuration $(\cX,\cD,\xi_0;\eta)$ of $(X,D)$, we can find a $\mathbb{Z}$-graded sequence of ideals $\{\fa_{\bullet}\}$ of $R$ {(see \eqref{eq-Rweight})} such that 
\begin{enumerate}
\item $\fa_k=R$ for $k\le 0$;
\item $\fa_k$ is a homogeneous ideal for any $k\in \bZ$: $\fa_k =\bigoplus_{\alpha} \fa_k\cap R_{\alpha} $ for any $k \in \bZ$;
\item the extended Rees algebra
$ \cR ees:=\bigoplus_{k\in \mathbb{Z}} t^{-k} \fa_k $ satisfies ${\rm Spec}(\cR ees)=\cX$. 
\end{enumerate}
Moreover, if $\eta$ is in the Reeb cone of $X_0$ with respect to $\tilde{T}=T\times\bC^*$, then $\fa_k$ is primary for $k>0$.
\end{lem}
\begin{proof}Recall by the definition of the test configuration, $\cX={\rm Spec}(\cR)$ where $\cR=\bigoplus \cR_{\alpha}$ and each $\cR_{\alpha}$ is a flat $\bC[t]$-module.
%Under the $\bC^*$-action generated by $\eta$, there is an isomorphism $\cX\times_{\bC}\bC^*\cong X\times\bC^*$. 
For any $f\in \cO_X$, we could denote by $\bar{f}$ its pull back from the first factor of $X\times\bC^*$. Since  $\cX\times_{\bC}\bC^*\cong X\times\bC^*$,  we could  mimic the construction in \cite[Section 2.5]{BHJ15} by defining $\fa_k=\bigoplus_{\alpha}\left\{f\in R_{\alpha} \ | t^{-k} \bar{f} \in \cR_{\alpha}\right\}$, and then we form the extended Rees algebra $\mathcal{R}ees=\bigoplus_{k\in \mathbb{Z}}\fa_kt^{-k}$. We claim $\cR ees$ is finitely genrated.
In fact, by construction, we have the injective morphism $\cR ees\rightarrow \mathcal{R}$. Conversely using the weight decomposition with respect to the $\bC^*$-action, any $F\in \cR_\alpha$ is of the form $\sum_k t^{-k}\bar{f}_k$ for some $\{f_k\}\subset R_\alpha$.  So we have $\cR\cong \cR ees$, which in particular % (see  \cite[Section 1.2]{BHJ15}).
implies that %$\{\fa_{k}\}_{k\in \mathbb{Z}_{\ge 0}}$ 
$\cR ees$ is finitely generated.

Since $\cR$ is a flat $\bC[t]$-algebra, that means $\fa_0=R$ which implies that $\fa_{k}=R$ for $k\le 0$. This is the first property. 
The second property follows from that the $\bC^*$-action generated by $\eta$ commutes with $T$. 

Finally, if $\eta$ is in the Reeb cone, then $\langle \eta, \alpha\rangle >0$ for any $\alpha\in \Gamma\backslash\{0\}$ (see \eqref{eq-Rweight}). Thus for any $\alpha\neq 0$ and $f\in R_{\alpha}$, the order of $f$ vanishing along $(t=0)$ is $\langle \eta, \alpha\rangle >0$, which implies for any $k$, $f^m\in \fa_k$ for $m\gg 0$.%Let $\cX={\rm Spec}(\cR)$.Since the test configuration is $T$-equivariant, we know $\cR=$ 
%This follows from \cite[Section 5]{LX17}. There we showed that since $R^*:=\gr_{v_{\xi_0}}(R)$ is finitely generated, we can find  $(x\in X)\subset \mathbb (0\in \mathbb{C}^N)$ with a rational weight on $\mathbb{C}^N$ whose restriction to $X$ is a rational perturbation $\xi$ of $\xi_0$ such that $R^*=\gr_{v_{\xi}}(R)$. Then the weighted blow up given by $\xi$ determines $\fa_n$.
\end{proof}

\begin{rem}\label{r-primary}
Since the Reeb cone with respect to $T$ is open, for any given test configuration, one can always perturb $\xi_0$ to be a rational Reeb vector $\xi'_0\in N_\bQ^+$. Our choice of the sign for $\eta$ with $\pi_*(\eta)=-t\partial_t$ means that  the weight $\eta$ on the function $t$ has weight $1$. For $m\gg 1$ sufficiently divisible, $m \xi'_0+\eta$ is an integral vector in the Reeb cone with respect to $\tilde{T}=T\times\bC^*$. 
 See e.g. Example \ref{exmp-Ad}.
\end{rem}

We give a way of obtaining test configurations using models. \Blue{It generalizes the construction of special test configurations via Koll\'{a}r components as discussed in \cite[2.2.1]{LX17}.}

\begin{exmp}\label{exmp-Ad}
First, we give an example from \cite[Example 7.1.2]{LX16} which will illustrate the construction.
%the above correspondence between special test configurations and Koll\'{a}r components. 
We refer to \cite[2.2.1]{LX17} for more general constructions. \\
Consider the $3$-dimensional $A_d$ singularity $$X=\{z_1^2+z_2^2+z_3^2+z_4^{d+1}=0\}\subset \bC^4  \quad \text{ with }\quad d>3.$$ Set $D=\emptyset$, $\xi_0=\sum_{i=1}^3 (d+1)z_i\partial_{z_i}+2 z_4\partial_{z_4}$ which generates the natural $\bC^*$-action on $X$, and set the vertex to be $x=(0,0,0,0)$. Then $(X,\xi_0)$ is a log Fano cone. 

Consider the filtered blow $\mu: Y\rightarrow X$ which is given by the strict transform of $X$ under the weighted blowup of $\bC^4$ with weights $(2,2,2,1)$. The exceptional divisor is given by $E=\{Z_1^2+Z_2^2+Z_3^2=0\}\subset \bP(2,2,2,1)\cong \bP(1,1,1,1)=\bP^3$. We see that $E\cong \bP(1,1,2)$ with the different $\Delta_E=\frac{1}{2}D_\infty$ where $D_\infty=\{Z_4=0\}\cap E=\bP^2\cap E$. Note that $E$ is a Koll\'{a}r component.

We can also consider the special test configuration given by $$\cX=\{(t; z_1, z_2, z_3, z_4); z_1^2+ z_2^2+ z_3^2+t^{d-3} z_4^{d+1}=0\}\subset \bC^4\times\bC,$$ and 
$\eta=-t\partial_t+2\sum_{i=1}^3 z_i\partial_{z_i}+z_4\partial_{z_4}$ generates the $\bC^*$-action $$(\tau, (t; z_1, z_2, z_3, z_4))\mapsto (\tau^{-1} t, \tau^2 z_1, \tau^2 z_2, \tau^2 z_3, \tau z_4).$$ Since we assume $d>3$, the central fibre $X_0=\{z_1^2+z_2^2+z_3^2=0\}\subset \bC^4$ is isomorphic to $(\bC^2/\bZ_2)\times\bC$ and admits a $(\bC^*)^2$-action generated by $\xi_0$ and $\eta$.
Note that $\eta|_{X_0}$ is in the Reeb cone of $X_0$ with respect to $(\bC^*)^2$ and 
$$(E, \Delta_E)=(X_0\setminus\{0\})/ (\bC^*=\la \eta\ra).$$
 %The ideal $\mathfrak{a}_k=\mu_*(\cO_Y(-k E))$ is then given by
%$${\rm Span}\left\{\left.z_1^{m_1}z_2^{m_2}z_3^{m_3}z_4^{m_4}\right|_X; 2(m_1+m_2+m_3)+m_4\ge k\right\}.$$
 Moreover in this example the log Fano pair $(E, \Delta_E)$ admits an orbifold K\"{a}hler-Einstein metric and is hence K-polystable. So by \cite{LX16, LX17, XZ20}, $\ord_E$ is the unique minimizer of $\hvol_{X,x}$.
\end{exmp}

\begin{defn}\label{d-model}Let $(X,D,\xi_0)$ be a log Fano cone singularity. 
Let $\mu\colon Y\to X$ be a $T$-equivariant proper birational morphism from a normal model $Y$, which is an isomorphism outside $X\setminus \{x\}$ with a $T$-equivariant integral Weil divisor $E$ supported on ${\rm Ex}(\mu)$ such that $-E$ is ample. Denote by $\cR:= \bigoplus_{k\in \mathbb{Z}} t^{-k} \fb_k$, where $\fb_k=\mu_*(\mathcal{O}_Y(-kE))$. 

Then $(\cX,\cD,\xi_0;\eta)$ is a test configuration associated to the model $\mu: Y\rightarrow X$, where $\cX:={\rm Spec}(\cR)$ and $\cD$ is the cycle (with $\mathbb{Q}$-coefficients) degeneration of $D$. More precisely, if we write $D=\sum a_iD_i$, where $D_i$ are prime divisors with the corresponding ideal $I_{D_i}$, then we can define $\cD_i$ on $\cX$ to be the divisor corresponding to the ideal $\mathcal{I}_{\cD_i}:=\bigoplus_{k\in \mathbb{Z}}(\fb_k\cap I_{D_i})t^{-k}\subset \mathcal{R}$, and let $\cD=\sum a_i\cD_i$.
\end{defn}

Conversely, starting with a normal test configuration $(\cX,\cD, \xi_0; \eta)$ and assuming $\eta$ is in the Reeb cone of $X_0$, we take the primary ideals $\fa_k$ as in Lemma \ref{l-filtration}, and then take the normalized filtered blow up (see \cite[Chapter 1]{TW89} for the definition) $\mu\colon Y\to X$ induced by $\fa_{\bullet}=\{\fa_{k}\}_{k\in \mathbb{Z}}$ with an exceptional divisor $E$.  %Then {\Blue{${\rm Proj}\bigoplus_{k=0}^{+\infty} \fa_{k}/\fa_{k+1}$ } is isomorphic to the exceptional divisor $E$ on $Y$.}

\begin{lem}\label{l-tc}The above two constructions give equivalence between normal test configurations  $(\cX,\cD, \xi_0; \eta)$ with $\eta$ in the Reeb cone  and models $\mu\colon Y\to X$ satisfying the conditions in Definition \ref{d-model}. 
%In the above notation, we constructed $\{\fb_k\}$ as in Definition \ref{d-model} for such $\mu$, where $E\subset Y$ is the pull back of $\mathcal{O}(1)$ on the filtered blow up given by ${\rm Proj}\bigoplus_{k=0}(\fa_k/\fa_{k+1})$. 
%if the test configuration $(\cX,\cD,\xi_0; \eta)$ is weakly special, 
%we have $\fa_k=\fb_k$ for any $k\in \mathbb{Z}$. 
Moreover, 
\begin{enumerate}
\item  $(\cX,\cD, \xi_0; \eta)$ is a special test configuration if and only if $\mu\colon Y\to X$ yields a Koll\'ar component; and
\item $(\cX,\cD, \xi_0; \eta)$ is weakly special if and only if $(Y,E+\mu_*^{-1}D)$ is log canonical.
\end{enumerate}
\end{lem}
\begin{proof}If we start with a normal test configuration $(\cX,\cD, \xi_0; \eta)$, then we get a graded sequence of primary ideals $\{\fa_{\bullet}\}$ by Lemma \ref{l-filtration}. 
If we take the filtered blow up of $\{\fa_{\bullet}\}$ and get $E$ as above, then we claim it is normal and the algebra $\{\fb_k=\mu_*(\mathcal{O}_Y(-kE))\}$ is the same as the algebra $\bigoplus_{k=0}\fa_k$.

In fact, $\bigoplus_{k=0}\fa_k\subset\bigoplus_{k=0}\fb_k$ is a subalgebra, but the latter is integral over the former.
Thus it suffices to verify that the $R$-algebra  $\bigoplus_{k=0}\fa_k$ is integrally closed. %This follows from that $\mathcal{R}=\bigoplus_{k\in \bZ}\fa_k$ is normal. 
Similar to  the proof of \cite[9.6.6]{Laz04}, this follows from the fact that  to check whether a function $f$ is contained in $\fa_k$ suffices to only check it at the divisorial valuation along the the special fiber $X_0$. More precisely, let the special fiber $X_0=\sum m_iE_i$ where $E_i$ are the prime divisors, then
\begin{eqnarray*}
& &\mbox{a homogeneous element  } f\in \mbox{ the normal closure } \overline{\bigoplus_{k=0}\fa_k} \\
&\Leftrightarrow& \mbox{$f$ satisfies an equation } f^m+a_1f^{m-1}+\cdots+a_m=0 \mbox{ with }a_i\in \fa_{ik},
\end{eqnarray*}
which implies the vanishing order of $f$ along $E_i$ is at least $km_i$ as the element in $\fa_j$ have vanishing order along $E_i$ at least $jm_i$ by the definition. Then we conclude $f\in \fa_k$. 

If we start with a normal model $\mu\colon Y\to X$ and $E$ as in Definition \ref{d-model}, then $\bigoplus_{k=0}\fb_k$ is a normal algebra where $\fb_k=\mu_*(\mathcal{O}_Y(-kE))$, then we can easily show the Rees algebra $\bigoplus_{k\in \bZ}t^{-k}\fb_k$ is normal, thus the induced test configuration $(\cX,\cD, \xi_0; \eta)$ is normal. If we take the filtered blow up then $Y\cong {\rm Proj}(\bigoplus_{k=0}\fb_k)$ as $-E$ is ample, and the divisor ${\rm Proj}(\bigoplus_{k=0}\fb_k/\fb_{k+1})\subset Y$ yields $E$.

\medskip 

To prove the second part of the statement, let $v\colon \mathbb{A}^1_{\bC}\subset \cX$ corresponds to the section of vertices. Consider the $\bC^*$-action given by $\eta$ in the data of the test configuration, then 
$(\cS={\rm Proj}_{\bC[t]}\cR ees,\mathcal{B})$ is the base of the $\bC^*$-quotient of $(\cX\setminus v(\mathbb{A}^1_{\bC}),\cD\setminus v(\mathbb{A}^1_{\bC}))$ as a Seifert bundle (see \cite{Kol04}), i.e., we remember the codimension one orbifold structure and put it into $\mathcal{B}$. Over the special fiber, we have 
$$S_0\cong {\rm Proj}\bigoplus_{k=0}\fa_{k}/\fa_{k+1}\cong {\rm Proj}\bigoplus_{k=0} \fb_{k}/\fb_{k+1}\cong E.$$
If $(\cX,\cD+\cX_0)$ is log canonical, then $X_0$ is reduced and $(X_0,D_0)$ is semi-log-canonical. Thus $E=S_0$ is reduced and $(S_0, B_0:=\mathcal{B}|_{S_0})$ is semi-log-canonical. Moreover, if we write $K_E+D_E=(K_Y+E+\mu_*^{-1}D)|_E$, then $D_E$ is sent to $B_0$ under the isomorphism between $E$ and $S_0$.
 Thus by inversion of adjunction, $(Y,E+\mu_*^{-1}D)$ is log canonical.

For the converse, assume $(E,D_E)$ is log canonical, it suffices to show that $\fb_k/\fb_{k+1}=H^0(E,\mathcal{O}_E(-kE))$ for any positive integer $k$, as this implies that $(X_0,D_0)$ is the orbifold cone over $(E,D_E)$ induced by the ample ($\bQ$-Cartier) integral Weil divisor $-E$. First, since  the test configuration is $\mathbb{Q}$-Gorenstein,  $-K_{E}-D_E\sim_{\bQ} \lambda E|_E$ for some $\lambda>0$ . Therefore, $E$ and $K_Y+E+\mu_*^{-1}D$ are anti-ample over $X$,
%$(Y,E+\mu_*^{-1}D)$ is log canonical and $K_Y+E+\mu_*^{-1}D\sim_{X,\mathbb{Q}} \sum a_iE_i$ where $E_i$ are components of $E$ and some $a_i\ge 0$, 
$$-(k+1)E=K_Y+E+\mu_*^{-1}D-(k+1)E-{(}K_Y+E+\mu_*^{-1}D), $$ we conclude that $R^1\mu_*\cO_Y(-(k+1)E)=0$ by the Kawamata-Viehweg Vanishing Theorem, then we can apply $\mu_*$ to the following exact sequence
\[
0\to \mathcal{O}_Y(-(k+1)|_E)\to \mathcal{O}_Y(-kE)\to \mathcal{O}_E((-kE)|_E)\to 0.
\]
  If we specialize the argument to the plt case, we obtain that $E$ is indeed a Koll\'ar component. 
\end{proof}

%\begin{rem}\label{r-tc}
%By the above lemma, we know that after the perturbation, $(\cX,\cD,\xi;\eta)$ can be realized as the filtered blow up of $\fa_{\bullet}$ and taking the corresponding cone over the exceptional locus (see by \cite[Section 6]{LX16}). a test configuration which degenerates $X$ to an orbifold cone over $E$ (with possible embedded points supported on the vertex),

%Conversely, by \cite[Section 2]{LX16},
%\end{rem}

%\begin{rem}
%\begin{enumerate}
%\item
%Similar to the observation in \cite{Li18}, $A_{(X_0, D_0)}(\xi)$ can be described as the weight of a holomorphic logarithmic $n$-form under the holomorphic vector field associated to $\xi$.  
%\item
%If $(\cX, \cD, \xi_0; \eta)$ is a special test configuration of $(X, D, \xi_0)$. Then $(\cX, \cD, \xi_0; \eta+a \xi_0)$ is also a special test configuration of $(X, D, \xi_0)$ for any $a\in \bZ$.  
%Notice that we have the invariance: for any $a\in \bQ$: 
%\begin{eqnarray*}
%\Fut(X_0, D_0, \xi_0; a \xi_0+\eta)%&=&n A(\xi)^{n-1} \lambda A(\eta) a_0(\xi)+A(\xi)^n\cdot \lambda (D_{\eta}a_0)(\xi)\\
%&=& \Fut(X_0, D_0, \xi_0; \eta).
%\end{eqnarray*}
%This follows from the invariance of the normalized volume $\hvol_{(X_0,D_0)}(a \xi_0)=\hvol_{(X_0,D_0)}(\xi_0)$ for any $a>0\in \bR$.
%\end{enumerate}
%\end{rem}

\subsubsection*{Generalized Futaki invariants and K-stability}

We define the generalized Futaki invariant for $\mathbb{Q}$-Gorenstein test configuration using the volume function. One can easily show this definition is the same as the one in \cite{CS12}. However, the formula in Definition \ref{d-Fut} is more convenient to use for the current paper. 

\begin{defn}[Generalized Futaki invariant]\label{d-Fut}
For any $\mathbb{Q}$-Gorenstein test configuration $(\cX, \cD, \xi_0; \eta)$ of $(X,D,\xi_0)$ {with the central fibre denoted by} $(X_0, D_0, \xi_0)$, its generalized Futaki invariant is defined as
{\begin{eqnarray*}
\Fut(\cX, \cD, \xi_0; \eta):=\frac{D_{T_{\xi_0}(\eta)}\vol_{X_0}(\xi_0)}{\vol_{X_0}(\xi_0)},
\end{eqnarray*}
where we used $T_{\xi_0}(\eta)$ in \eqref{eq-normeta} and the directional derivative
$$D_{T_{\xi_0}(\eta)}\vol_{X_0}(\xi_0):=\left.\frac{d}{d\epsilon}\right|_{\epsilon=0}\vol_{X_0}(\xi_0+\epsilon T_{\xi_0}(\eta)).$$
%The second identity follows from the assumption that $(\cX, \cD+X_0)$ is log canonical. 
Since generalized Futaki invariant defined above only depends on the data on the central fiber, we will also denote it by $\Fut(X_0, D_0, \xi_0; \eta)$.}
\end{defn}

Next, we will introduce the notions of K-stability. We note that in the definition, we only look at special test configurations, in the spirit of \cite{Tia97}. 

\begin{defn}[K-stability]\label{d-semicone}
We say that $(X, D, \xi_0)$ is K-semistable, if for any special test configuration $(\cX, \cD, \xi_0; \eta)$, we have $\Fut(\cX, \cD, \xi_0; \eta)$ is nonnegative.

We say that $(X, D, \xi_0)$ is K-polystable, if it is K-semistable, and any special test configuration $(\cX, \cD, \xi_0; \eta)$ with $\Fut(\cX, \cD, \xi_0; \eta)=0$ is a product test configuration. 
\end{defn}

If $(\cX, \cD, \xi_0; \eta)$ is a special test configuration, we know $A(\xi_0)=A_{(X_0,D_0)}(\wt_{\xi_0})>0$. Then we see the following identity holds:
\begin{equation}\label{eq-dirhvol}
D_{T_{\xi_0}(\eta)}\vol_{X_0}(\xi_0)=\left.\frac{d}{d\epsilon}\right|_{\epsilon=0}\hvol_{X_0}(\wt_{\xi_0+\epsilon\eta})\cdot \frac{1}{n A(\xi_0)^{n-1}},
\end{equation}
where we use the rescaling invariance of the normalized volume and $A(\xi_0)=A(\xi_0+t\cdot T_{\xi_0}(\eta))$ for $t\ll1$  (see \eqref{eq-normeta})
As a consequence, we can rewrite the Futaki invariant of a special test configuration in the following way:
\begin{equation}\label{eq-Futhvol}
\Fut(\cX, \cD, \xi_0; \eta):=D_{\eta}\hvol_{X_0}(\wt_{\xi_0})\cdot \frac{1}{n A(\xi_0)^{n-1} \cdot \vol_{X_0}(\xi_0)}.
\end{equation}
This shows that it differs from the one in \cite[Definition 2.26]{LX17} by a positive constant. It also differs from Collins-Sz\'{e}kelyhidi's definition by a constant.

\begin{rem}Obviously to define the K-stability notions, we can also consider more general test configurations than the special ones. 
In {\em \cite{LX14}} we proved that for the K-stability of log Fano varieties, to test on all test configurations is equivalent to only test on special test configurations.

 For log Fano cone singularities, results like {\em \cite{LX14}} %to give the equivalence of definitions for stability among different classes of test configurations 
are not completely known.  Nevertheless, later in this paper, we have to deal with weakly special test configurations, as they will naturally appear in our argument. Thus we need to prove a statement (see Proposition \ref{p-specialcone}) similar to  \cite[Theorem 4]{LX14}, which says  that for log Fano cone singularities, our definition of K-semistability is also equivalent to test on all weakly special test configurations. 
%which a priori could be more restrictive. 
 
 Compared to the other literatures, all  test configurations are considered in \cite{CS12}, whereas in \cite{CS15, LX17} K-stability notions are only tested on special test configurations. 
\end{rem}

We will need the following simple fact, which follows from the definition of the generalized Futaki invariant applied to product test configurations:
\begin{lem}\label{lem-Futlinear}
Assume that the log Fano cone $(X, D, \xi_0)$ admits a torus action by $T'\cong (\bC^*)^{r'}$ that commutes with $\la \xi_0 \ra$. Let $N'$ be the coweight lattice of $T'$.
Then the function 
$$\eta\mapsto \Fut(X_\bC, D_\bC, \xi_0; \eta)$$ 
is linear with respect to $\eta\in N'_{\mathbb {R}}$. 
\end{lem}

\subsubsection*{Log Fano varieties}

In the below, we will specialize previous definitions to the case of quasi-regular log Fano cones, which correspond to Fano projective varieties. 

% We recall the following definition:
\begin{defn}
Assume $(S, B)$ is a log Fano variety.
A test configuration of $(S, B, -(K_S+B))$ is a quadruple $(\cS, \mathcal{B}, \cL; \eta)$ with a map $\pi: (\cS, \mathcal{B})\rightarrow \bC$ that satisfies the following conditions: 
\begin{enumerate}
\item[(1)] $\cL$ is a $\pi$-ample $\bQ$-line bundle and $\pi\colon \cS\to \bC$ is a flat family and ${\rm Supp}(B)$ does not contain any component of the fiber. We denote the central fiber by $(S_0, B_0, L_0)$.
\item[(2)] There is a $\bC^*$-action (with coweight $\eta$) on $(\cS, \mathcal{B})$ such that $\pi$ is $\bC^*$-equivariant where $\bC^*$ acts on the base $\bC$ by multiplication and there is a $\bC^*$-equivariant isomorphism $\phi: (\cS, \mathcal{B}, \cL)\times_{\bC}\bC^*\cong (S, B, -(K_S+B))\times\bC^*$, where $\bC^*$-trivially acts on the first factor of $ (S, B, -(K_S+B))\times\bC^*$.
\end{enumerate}
Most of the time, as in the literature, we omit $\eta$ in the quadruple and simply denote the test configuration by $(\cS, \mathcal{B}, \cL)$.

Such a test configuration is called $\bQ$-Gorenstein if $\cS$ is normal, 
$$K_\cS+\mathcal{B} \mbox {  is $\bQ$-Cartier \ \ \ and \ \ \ }  \cL\sim_{\mathbb{Q}}-(K_{\cS}+\mathcal{B}).$$ In this case, we usually just write the test configuration as
$(\cS, \mathcal{B}; \eta)$ or simply as $(\cS, \mathcal{B})$.

A $\bQ$-Gorenstein test configuration is  called special if $(S_0, B_0)$ is a log Fano pair with klt singularities. In this case, we say that $(S_0, B_0)$ is a special degeneration of $(S, B)$.

A test configuration $(\cS, \mathcal{B}, \cL; \eta)$ is called a product one  if there is an isomorphism 
\[
(\cS, \mathcal{B}, \cL)\cong (S, B, -(K_S+B))\times \bC \qquad\mbox{ such that }\eta=\eta_0+t\partial_t
\] 
where $\eta_0$ is a coweight vector on some torus group $T$ acting on $(S,B)$ and $t\partial_t$ is the coweight corresponding to the $\mathbb{C}^*$ factor. In this case, we will denote $(\cS, \mathcal{B}, \cL; \eta)$ simply by $(S_\bC, B_\bC; \eta).$

\end{defn}

For a test configuration of a log Fano variety, by trivially adding a copy over $\{\infty\}$,  we can take the intersection formula (see \cite{Wan12, Oda13}) of the generalized Futaki invariant as the definition.  
 More precisely, for any test configuration of $(\mathcal{S},\mathcal{B})$, we can glue it with a trivial family of $(S,B)\times \bP^1\setminus\{0\}$ along $(S,B)\times \bC^*$ to get $(\bar{\mathcal{S}},\bar{B})$ over $\bP^1$ and denote by $\bar{\cL}\sim_{\mathbb{Q}}-(K_{\bar{\cS}/\bP^1}+\mathcal{B})$.

\begin{defn}[Generalized Futaki invariants]\label{d-Fut2}
For any $\mathbb{Q}$-Gorenstein test configuration $(\cS, \mathcal{B}, \cL; \eta)$ of $(S, B)$, %with $(\cS,\cS_0+\mathcal{B})$ being log canonical, 
we define the generalized Futaki invariant 
\begin{eqnarray*}
\Fut(\cS, \mathcal{B}; \eta):=-\frac{\bar{\cL}^{\cdot n}}{n (-(K_S+B))^{\cdot n-1}} \ \ \ \ \mbox{where $n=\dim S+1$.}
\end{eqnarray*}
\end{defn}

By the intersection formula  (see \cite{Wan12, Oda13}), the above definition of the generalized Futaki invariants coincides with the one in \cite{Don02}. 
\begin{defn}[{K-stability, see \cite{Tia97, Don02, LX14}}]\label{d-ksemiSE}
We say that $(S, B)$ is K-semistable, if %with $\cL=-(K_{\cS}+B)+\pi^*K_{\bP^1}$, its generalized Futaki invariant, denoted by $\Fut(\cS, \mathcal{B}, \cL; \eta)$ 
the generalized Futaki invariant $\Fut(\cS, \mathcal{B}; \eta)$ is nonnegative for any special test configurations. 
We say that $(S, B)$ is K-polystable, if it is K-semistable, and any special test configuration $(\cS, \mathcal{B}, \cL; \eta)$ with $\Fut(\cS, \mathcal{B}, \cL; \eta)=0$ is a product test configuration. 
\end{defn}
\begin{rem}
We choose to work specifically on $\mathbb{Q}$-Gorenstein test configurations $(\cS, B, \cL)$, since it fits into our study on log Fano cones. By \cite{LX14}, we know for a log Fano variety, working on this intermediate generality of test configurations yields the same stability notions as working either only on special test configurations or on all test configurations. 
\end{rem}

Given a $\bQ$-Gorenstein test configuration $(\cS, \mathcal{B}, \cL; \eta)$, by choosing $\lambda$ such that $\lambda(K_{\cS}+\mathcal{B})$ is Cartier, % and restricting over $\bC$, 
we can get a $\bQ$-Gorenstein test configuration $(\cX, \cD, \xi_0; \eta)$ of $(X,D):=C(S, B, -\lambda(K_S+B))$ by letting 
$(\cX, \cD)=C(\cS, \mathcal{B}; -\lambda\cL)$, $\xi_0=u\partial_u$ the canonical rescaling vector on $\cX$ where $u$ is an affine coordinate on the line bundle $\lambda\cL$, and letting $\eta$ also denote its canonical lifting from $\cS$ to $\cX$ that corresponds to the pull back of pluri-log-canonical forms (see \cite[Page 3186-3187]{Li17}. %Note that with this choice of $\eta$, the test configuration $(\cX, \cD, \xi_0; \eta)$  satisfies $A(\xi_0)=\lambda^{-1} $.
%$$A(\xi_0)=\lambda^{-1} \mbox{\ \ \  and \ \ \ }A(\eta)=0.$$ 
\Blue{
\begin{lem}[{see \cite[Theorem 4]{CS12} and \cite[Lemma 6.20]{Li17}}]
Notations as above. 
If $(\cS,\mathcal{B}; \eta)$ is a $\mathbb{Q}$-Gorenstein test {configuration}, %, 
then $$\Fut(\cS, \mathcal{B}; \eta)=\Fut(\cX, \cD, \xi_0; \eta).$$
\end{lem}
\begin{proof}
With the above choice of $\xi_0$, $A(\xi_0)=\lambda^{-1}$. 
Since $\eta$ is the canonical lifting, we have $A(\eta)=0$ so that $T_{\xi_0}(\eta)=\frac{A(\xi_0)}{n}\eta=\frac{\lambda^{-1}}{n}\eta$ (see \eqref{eq-normeta}). So we get:
\[
D_{T_{\xi_0}(\eta)}\vol_{X_0}(\xi_0)=\frac{\lambda^{-1}}{n}
\left.\frac{d\vol_{X_0}(\xi_0+t\eta)}{dt}\right|_{t=0}=\lambda^{-1} \lim_{m\rightarrow +\infty} \frac{w_m}{m^n/(n-1)!}=-\frac{\lambda^{n-1} \bar{\cL}^{\cdot n}}{n},
\]
where $w_m$ is the weight of the $\la\eta\ra$ action on $H^0(S_0, -m \lambda (K_{S_0}+B_0))$. %, \blue{and \\$N_m=\dim_\bC H^0(S_0, -m\lambda(K_{S_0}+B_0))$}. 
The second identity follows from \cite[Theorem 4]{CS12} (see also the calculation in \cite[Proof of Lemma 6.20]{Li17}). For the last identity, see
\cite[Theorem 5.3]{BHJ15}. Dividing both sides by $\vol_{X_0}(\xi_0)=\lambda^{n-1} (-(K_{S}+B))^{\cdot n-1}$ we get the identity.
%This follows from the argument in \cite[Lemma 6.20]{Li17}.
%In fact, in \cite[Page 3192]{Li17}, it was calculated that for any test configuration
%(We note up to getting this formula, the argument in \cite[Lemma 6.20]{Li17} does not use the assumption that $(\cS,\mathcal{B})$ is special and the calculation works for any test configurations.)
%Let $\mu \colon Y\to X$ be the blow up of the vertex so that we obtain the canonical valuation $S_{\infty}\cong S$.
%Then it suffices to notice that 
%\[
%A(\xi_0)^{n-1}\vol_{X_0}(\xi_0)=(-(K_Y-S_{\infty}-\mu_*^{-1}D)|_{S_{\infty}})^{\cdot n-1}=(-K_S-B)^{\cdot n-1},
%\]
%and thus we conclude $\Fut(\cS, \mathcal{B}; \eta)=\Fut(\cX, \cD, \xi_0; \eta)$ by the definitions. 
%By using \eqref{eq-Futhvol} and similar calculations as in \cite[Lemma 6.20]{Li17}, we can verify that
%$$D^\NA(\cS, \mathcal{B}; \eta)=D^\NA(\cX, \cD, \xi_0; \eta).$$ 
%In particular, if $(\cS, \mathcal{B}, \cL; \eta)$ is a special test configuration, then 
\end{proof}
The above lemma says that the definition \ref{d-Fut2} is compatible with the generalized Futaki invariants for log Fano cones in Definition \ref{d-Fut}. Thus Definition \ref{d-semicone} specializes to Definition \ref{d-ksemiSE}.  
}
It is well known that if we have a product test configuration induced by a vector field coming from a $\mathbb{C}^*$-action on $(S, B)$, then the generalized Futaki invariant defined above becomes the classical Futaki invariant. It also follows from Lemma \ref{lem-Futlinear} that
\begin{lem}\label{lem-Futlinear2}
Assume a log Fano variety $(S, B)$ admits a torus action by $T\cong (\bC^*)^{r}$. Let $N$ be the coweight lattice of $T$. Then the Futaki invariant $\eta\mapsto \Fut(S\times \bC, B\times \bC; \eta)$ is linear with respect to $\eta\in N_{\bR}$.
\end{lem}

\Blue{\section{Case of log Fano pairs}\label{sec-logpair}
In this section, we will focus on the stability of log Fano pairs.  More concretely we will construct a common degeneration of two K-semistable degenerations of a log Fano variety, as well as investigate the equivariant K-stability for a torus action. Even in this case of log Fano varieties, we find it more flexible to work on the associated log Fano cones in order to use a combination of techniques from the minimal model program and results on normalized volumes.} The study will be generalized to log Fano cones later. However, we believe that treating the case of log Fano pairs first will help the reader to more easily get the main idea.  

%Other works on this topic also have appeared in \cite{LX16, LX17, Blu16}.

\subsection{K-semistable degeneration of K-semistable log Fano pair}\label{s-Qfano}
\def\Fut{\mathrm{Fut}}
\def\hilb{\mathrm{Hilb}}
\def\SL{\mathrm{SL}}
\def\GG{\mathbb{G}}
\def\PP{\mathbb{P}}
\def\la{\langle}
\def\ra{\rangle}

We will need the following lemma which allows us to reduce a two-step equivariant degeneration to a single equivariant degeneration. The idea of its proof  is similar to the one used in \cite[Section 6]{LX16}. In fact, the proof is a mimic of the argument in the classical GIT situation, but replacing Kempf's instability theorem \cite[Corollary 4.5]{Kempf78} by \cite[Theorem 1.4]{LX16}.
\begin{lem}\label{l-semistable}
Let $(\mathcal{S},\mathcal{B}, \eta)$ be a special test configuration of a K-semistable log Fano variety $(S, B)$ with central fiber $(S_0, B_0)$. Suppose  that ${\rm Fut}(\cS,\mathcal{B})=0$. Then $(S_0, B_0)$ is a K-semistable log Fano variety.
\end{lem}
\begin{proof}
By \cite[Theorem E]{LX16}, if a log Fano pair admits a torus action, then to test its K-semistability it suffices to consider torus equivariant test configurations. Therefore, if $(S_0, B_0)$ is not K-semistable, then there is an equivariant  special test configuration $(\cS', \mathcal{B}'):=(\cS', \mathcal{B}', \eta')$ with respect to the $\bC^*$-action corresponding to $\eta$ such that 
$$\Fut(\cS', \mathcal{B}', \eta')<0.$$ We denote by $(S'_0, B'_0)$ the central fiber of $(\cS', \mathcal{B}', \eta')$. 

We can assume $(\cS, \cB)$ (resp. $(\cS', \cB')$) is $\bC^*$-equivariantly embedded into $\bP^N\times\bC\times\{0\}$ (resp. $\bP^N\times \{0\}\times\bC$). By abuse of notations, we denote
$\eta: \bC^*\to \SL(N+1)$ (resp. $\eta': \bC^*\to \SL(N+1)$) to be the 1-parameter subgroup (1-PS) generated by $\eta$ (resp. $\eta'$). Then $\eta$ commutes with $\eta'$, or equivalently
$[\eta, \eta']=0$. %Let $\hilb(S,B)\in \hilb(\PP^N, P(k))$ where $P(k)$ is the Hilbert polynomial. 
Let $\Theta=mB$ for some sufficiently divisible positive integer $m$ such that $\Theta$ is integral. 
\begin{equation}
\hilb(S,\Theta)\in \HH^{N,P,p}:=\{(\hilb(S),\hilb(\Theta))\in \hilb(\PP^N, P)\times \hilb(\PP^N, p) \mid \Theta\subset S\subset\PP^N\}
\end{equation}
where $P(k)=h^0(\cO_S(k))$ and $p(k)=h^0(\cO_\Theta(k))$ for $k\gg1$ are the Hilbert polynomials for $(S,\Theta)\subset \PP^N\times\PP^N$. The $\SL(N+1)$-action on $\PP^N$ induces an action on $\HH^{N,P,p}$. 
We then have the following convergence:
\begin{align*}
&\hilb(S_0, \Theta^*_0)=\displaystyle \lim_{t\to 0}\eta(t)\cdot \hilb(S,\Theta).
\end{align*}
We remark $\Theta_0:=mB_0$ is not necessarily the same as the scheme $\Theta_0^*$ due to the possible appearance of embedded points on $\Theta_0^*$. However, we have the inclusion of the ideal sheaves $I_{\Theta^*_0}\subset I_{\Theta_0}$, with the support of the cokernel being of codimension at least two on $S_0$. We can similarly define $p'(k)=h^0(\cO_{\Theta_0}(k))$
and have the following convergence:
\begin{align*}
&\hilb(S'_0, \Theta'_0)=\displaystyle \lim_{t\to 0}\eta'(t)\cdot \hilb(S_0, \Theta_0) \subset \HH^{N,P,p'}\\
&\hilb(S'_0, \Theta'^*_0)=\displaystyle \lim_{t\to 0}\eta'(t)\cdot \hilb(S_0, \Theta^*_0)\subset \HH^{N,P,p}.
\end{align*}
Therefore, we have the inclusion of the ideal sheaves $I_{\Theta'^*_0}\subset I_{\Theta'_0}$, and the codimension of the  support of the cokernel  is at least two on $S'_0$.

Our goal now is to construct a {new} test configuration $(\cS'',\mathcal{B}'')$ of $(S, B)$ with a special fiber $(S'_0, B'_0)$ such that $\Fut(\cS'',\mathcal{B}'')<0$, contradicting to our assumption that $(S, B)$ is K-semistable.

Notice that  the action of $\bC^\ast\times\bC^\ast\cong\la\eta\ra\times\la\eta'\ra<\SL(N+1)^2$ on $\HH^{N,P,p}$ induces a $\bC^\ast\times \bC^\ast$-equivariant map  
\begin{equation}\label{tau}
\begin{array}{cccc}
\phi:&\bC^*\times \bC^* & \overset{}{ \xrightarrow{\hspace*{1.0cm}}} & \HH^{N,P,p}\\
& (\eta,\eta') &\longmapsto & (\eta,\eta')\cdot \hilb(S,\Theta) \ 
\end{array}
\end{equation}
which may be regarded as a rational map  $\phi: \PP^1\times\PP^1\dashrightarrow \HH^{N,P,p}$. By extending $\phi$ to the closure  the its graph: 
$$(\PP^1\times\PP^1)\times \HH^{N,P,p}\supset\mathcal{G}:=\mathrm{graph}(\phi) \overset{\hat\phi}{ \xrightarrow{\hspace*{1.8cm}}} \HH^{N,P,p},$$
we obtain the following diagram
\begin{equation}\label{X-D-E}
\xymatrix{ & & \cG  \ar@{>}[d]^f  \ar@{>}[dr]^{\hat\phi} &  \\
\bC^* \ar@{>}[r]^{\!\!\!\!\tau}\ar@{>}@[][urr]^{\hat \tau}	&\bC^*\times \bC^* \ar@{^{(}->}[r] \ar@/_1.8pc/@[][rr]_{(\eta,\eta')\cdot\hilb(S,\Theta)}   &\PP^1\times\PP^1 \ar@{-->}[r]^\phi & \HH^{N,P,p},}
\end{equation}
where  $f$ is a $(\bC^*\times \bC^*)$-equivariant blow-up, $\phi(0,0)=\hilb(S'_0, \Theta'^*_0)$ and 
%Now we introduce the 1-PS
\begin{equation}\label{tau}
\begin{array}{cccc}
\tau:&\bC^* & \overset{}{ \xrightarrow{\hspace*{1.cm}}} & \bC^*\times \bC^*\\
& t &\longmapsto & (t^k,t) \
\end{array}
\end{equation}
is a 1-PS with $k\gg1$ and $\hat\tau$ its lift. Then $\hat\tau$ satisfies
$$\hat\phi\circ\hat\tau(0)=\lim_{t\rightarrow 0} \hat{\phi}\circ\hat{\tau}(t)=\phi(0,0)=\hilb(S'_0, \Theta'^*_0).$$

Let $(\cS'',\til{\Theta}'')$ be the flat family obtained by pulling pack the universal family $(\cS^\hilb,\mathcal{B}^\hilb)\to {\HH^{N,P,p}}$ via $\hat\tau$, and let $\mathcal{B}'':=\frac{1}{m}\til{\Theta}''$.
Then $(\cS'',\mathcal{B}'')$ is a special test configuration and we have
\begin{eqnarray*}
\Fut(\cS'', \mathcal{B}'')&=&\Fut(S'_0, B'_0; k\eta+\eta')\\
&=&\Fut(S'_0, B'_0; k \eta)+\Fut(S'_0, B'_0; \eta')\\
&=&\Fut(S_0, B_0; k \eta)+\Fut(S'_0, B'_0; \eta')\\
&=&k\cdot \Fut(\cS,\mathcal{B})+\Fut(\cS' ,\mathcal{B}')\\
&=&0+\Fut(\cS' ,\mathcal{B}')<0,
\end{eqnarray*}
where we used the linearity of the Futaki invariant (cf. Lemma \ref{lem-Futlinear2}) in the second identity.
Hence $(\cS'',\mathcal{B}'')$ is the test configuration we are looking for and our proof is  completed.
\end{proof}

\subsection{Common degenerations of log Fano pairs}

\subsubsection{A common degeneration result and outline of proof}\label{sec-common1}
The main technical theorem of this section is the following.

\begin{thm}\label{t-doubledeg}
Let $(S,B)$ be an $(n-1)$-dimensional K-semistable log Fano variety. If there are special test configurations $(\mathcal{S}^{(i)},\mathcal{B}^{(i)})$ ($i=1,2$) of $(S, B)$ with central fibers $(S^{(1)}_0,B^{(1)}_0)$ and $(S^{(2)}_0,B^{(2)}_0)$ such that ${\rm Fut}(\mathcal{S}^{(i)},\mathcal{B}^{(i)})=0$, then there are two special test configurations $(\cS'^{(i)},\mathcal{B}'^{(i)})$ of $(S^{(i)}_0, B^{(i)}_0)$ with isomorphic central log Fano fibers $(S'_0, B'_0)$ such that
${\rm Fut}(\mathcal{S}'^{(i)},\mathcal{B}'^{(i)})=0$.
\end{thm}
We remark that Theorem \ref{t-doubledeg} should be  regarded as  an analogy of a corresponding statement in classical geometric invariant theory (GIT) {(see e.g. \cite[Theorem 3.5]{New78})}.
As an immediate consequence  we have the following:
\begin{cor}
In the above notion, if we assume further that $(S^{(1)}_0,B^{(1)}_0)$ is K-polystable, then there is a special test configuration of $(S^{(2)}_0,B^{(2)}_0)$ with generalized Futaki invariant 0 and central fiber isomorphic to $(S^{(1)}_0,B^{(1)}_0)$.
\end{cor}

To make our proof of Theorem \ref{t-doubledeg} in \ref{sec-pfdouble} more accessible, we first give an outline of the argument.
%with a focus on how we combine the MMP and an analysis of the normalized volume function. 
We will use notations introduced in previous sections.

Motivated by the works in \cite{Li17, LL16, LX16}, we consider the normalized volume function $\hvol_{(X,D,x)}(\cdot)$ defined on the valuation space $\Val_{X,x}$ over the vertex $x$ of the cone $(X,D)=C(S,B;-\lambda(K_S+B))$ for a sufficiently divisible $\lambda>0$. Then $(\cS^{(1)}, \mathcal{B}^{(1)})$ determines a  ``ray" of valuations, temporarily denoted by $\{w_\epsilon\}_{0\le \epsilon \ll 1}$, emanating from the canonical valuation $w_0=\ord_S$ ($S$ also denotes the divisor obtained by blowing up the vertex). Indeed, by Lemma \ref{l-n1} below, when $k\gg 1$, $w_{1/k}=a_k\cdot \ord_{E_k}$, where $a_k>0$ and $E_k$ is a Koll\'ar component over $(X,D, x)$ (see Lemma \ref{l-n1}). % (see Definition \ref{def-Kollar}). 
Moreover we know that the generalized Futaki invariant ${\rm Fut}(\mathcal{S}^{(1)},\mathcal{B}^{(1)})$ is the derivative of the normalized volume at $w_0$ along this ray. 

%Rcall that the latter means that there exists a birational model $\mu: Y_k\rightarrow X$ such that (i) the exceptional divisor $E_k$ is irreducible and $\bQ$-Cartier; (ii) $-E_k$ is $\mu$-ample and (iii) $(Y_k, \mu_*^{-1}D+E_k)$ is plt.

By taking cones similar as before, $\{(\cS^{(i)}, \mathcal{B}^{(i)})\}_{i=1,2}$ induce special degenerations of $(X, D)$, which will be denoted by $\{(\cX^{(i)}, \cD^{(i)})\}_{i=1,2}$. Note that $E_k\times\bC^*$ determines a divisorial valuation over $X\times\bC^*$ and hence over $(\cX^{(2)}, \cD^{(2)})$. As we mentioned in the introduction, our main goal is to construct a model $\cY_k^{(2)}\rightarrow \cX^{(2)}$ with a (prime) exceptional divisor $\cE^{(2)}_k$ given by $E_k\times \bC^*$, such that it satisfies $(\cY_k^{(2)}, \cE^{(2)}_k)\times_{\bC}\bC^*\cong (Y_k, E_k)\times \bC^*$, where the isomorphism is compatible with the equivariant isomorphism of the {\it second} special test configuration. 
%So the goal is to show that this divisorial valuation can be extracted as the only exceptional divisor over $\cX^{(2)}$. 
Based on the results from  the minimal model program (MMP) (see \cite{BCHM10}), this would be true if we could find a graded sequence of ideals $\frak{A}_\bullet$ and a positive real number $c'_k$ such that the following two conditions are simultaneously satisfied: 
%\begin{eqnarray}\label{e-extract}
\[
(\cX^{(2)}, \cD^{(2)}+c'_k \frak{A}_\bullet)\mbox{ is klt \ \ \ and\ \ \ } A(E_k\times\bC; \cX^{(2)}, \cD^{(2)}+ c'_k \frak{A}_\bullet)<1, \label{e-extract} \tag{\ding{95}}
\]
%\end{eqnarray}
where $A(E_k\times\bC; \cX^{(2)}, \cD^{(2)}+c'_k \frak{A}_\bullet)$ is the log discrepancy of (the birational transform of) $E_k\times\bC$ with respect to the triple $(\cX^{(2)}, \cD^{(2)}+ c'_k \frak{A}_\bullet)$. Note that this way of applying MMP is also a major ingredient in the study of some related problems in  \cite{Blu16, LX16, LX17} . 

To construct such a graded sequence $\frak{A}_\bullet$ of ideals, we look at the graded sequence of {valuation ideals} $\{\fa_{\bullet}\}$ of $\ord_{E_k}$ and its equivariant degeneration along the second special test configuration $(\cX^{(2)}, \cD^{(2)})$. The resulting graded sequence of ideals over $\cX^{(2)}$ will be denoted by $\frak{A}_{\bullet}$. We claim  for $k\gg 1$, $\frak{A}_{\bullet}$ is exactly what we are looking for. Indeed, as we will show (see Claim \ref{c'}), the assumptions that $(S,B)$ is K-semistable and $\Fut(\cS^{(1)}, \mathcal{B}^{(1)})=0$ guarantee the existence of $c'_k$ satisfying the two conditions in \eqref{e-extract}. This is possible thanks to the interaction between K-semistability and minimization of normalized volumes/normalized multiplicities. %, developed in \cite{Li17, Liu16, LL16, LX16, LX17}. 

%This can be used to shown that there exists a model $Y_k\to X$ extracting only $E_k$. Furthermore, this model degenerates well along the degeneration $(\mathcal{X}^{(2)},\mathcal{D}^{(2)})$ and we get a model $\cY^{(2)}_k\rightarrow \cX^{(2)}$ with the exceptional divisor $\cE^{(2)}_k$ such that $(\cY^{(2)}_k, \cE^{(k)})\times_{\bC}\bC^*=(Y_k,E_k)\times\bC^*$.
%over the special fiber $(V_2,B_2)$ of $(\mathcal{X}_2,\mathcal{D}_2)$,  we obtain $\mathcal{Y}_n\to (\mathcal{X}_2,\mathcal{D}_2)$ which is a degeneration of the model $W_n$. 

Applying the relative Rees algebra construction to $\cE^{(2)}_k\subset \mathcal{Y}^{(2)}_k/\mathbb{C}$, and then taking a quotient by the natural rescaling $\mathbb{C}^*$-action, one can obtain a family over $\mathbb{C}^2$, whose restriction to $\mathbb{C}\times \{t\}$ for $t\neq 0$ is the same as $(\cS^{(1)},\mathcal{B}^{(1)})$  and it gives a degeneration of $(S^{(1)}_0, B^{(1)}_0)$  when restricted to  $\bC\times \{0\}$. On the other hand, over $\{0\}\times \mathbb{C}$, one get a degeneration of $(S^{(2)}_0, B^{(2)}_0)$. Therefore, we obtain that the two log Fano varieties $(S^{(i)}_0,B^{(i)}_0)$ $(i=1,2)$, which are special fibers of the two special test configurations $(\cS^{(i)},\mathcal{B}^{(i)})\ (i=1,2)$ with ${\rm Fut}(\mathcal{S}^{(i)},\mathcal{B}^{(i)})=0\ (i=1,2)$, indeed admit degenerations with {\em isomorphic} special fibers (see Theorem \ref{t-doubledeg}).

\subsubsection{Proof of Theorem \ref{t-doubledeg}}\label{sec-pfdouble}

%\begin{proof}[Proof of Theorem \ref{t-doubledeg}]
As $(\mathcal{S}^{(i)},\mathcal{B}^{(i)})$ ($i=1,2$) are special test configurations,  $(S^{(1)}_0, B^{(1)}_0)$ and $(S^{(2)}_0, B^{(2)}_0)$ are log Fano varieties. Consider the  cone $(X,D)=C(S,B;-\lambda(K_S+B))$ over $S$ and similarly $(X^{(i)}_0, D^{(i)}_0)=C(S^{(i)}_0,B^{(i)}_0; -\lambda(K_{S^{(i)}_0}+B^{(i)}_0)) (i=1,2)$ for some sufficiently divisble $\lambda$. 
Denote the corresponding degeneration of $X$ to $X^{(i)}_0$ over $\mathbb{C}$ to be $\mathcal{X}^{(i)}$, then we get  special test configurations $(\cX^{(i)}, \cD^{(i)}, \xi_0; \eta^{(i)})$  of $(X,D,\xi_0)$, where $\cD^{(i)}$ is the cone over $\mathcal{B}^{(i)}$ and $\xi_0$ is from the natural $\mathbb{C}^*$-action on the cone. 
%Denote by $B_i$ the divisor $C(D_i)+(1-\frac{1}{nr})X^{\infty}_i$ on $V_i$ where $X^{\infty}_i$ is the infinite section of the cone.

\medskip

From \cite[Definition 4.4]{BHJ15}, we know that the central fiber $S^{(1)}_0$ of the special degeneration $\mathcal{S}^{(1)}$ induces a valuation $w':=q \cdot \ord_{F}$ for some divisor $F$ over $S$. Let $\ord_{S}$ denote the canonical divisorial valuation associated to the exceptional divisor, which is isomorphic to $S$, obtained by blowing up the vertex $x$. 
Assume $\mu: \til{S}\rightarrow S$ is a birational morphism such that the divisor $F$ is on $\tilde{S}$ and $(\til{S}, F)$ is log smooth. Let $\til{X}\to X$ be the resolution given by the total space of the line bundle of $\mu^*(\lambda(-K_X-B))$ over $\tilde{S}$.
Then following \cite[Page 3181-3182]{Li17}, we denote by $a_1=-\lambda (A_{(S, B)\times \bC}(S^{(1)}_0)-1)$ and let $w_{\epsilon}$ be the quasi-monomial valuation on the model $(\til{X}, \til{S}+\til{F})$ with weight $(1+ \epsilon a_1, \epsilon q)$ with respect to $\til{S}$ and the pull back $\til{F}$ of $F$ by $\tilde{X}\to \til{S}$ (see \cite[Definition 6.12]{Li17}). %where $C\in \mathbb{N}$ such that $C a_1$ and $Cq$ are positive integers. 
We choose $\epsilon^*$ such that $1+\epsilon a_1>0$ for any $\epsilon \in [0, \epsilon^*)$. Then $w_{\epsilon}$ is centered at the vertex $x$ of $X$. 
By \cite[Proposition 6.16]{Li17}, we have the identity: 
\begin{equation}\label{eq-AXDvk}
A_{(X,D)}(v_k)=k A_{(X,D)}(w_{1/k})=k\cdot A_{(X,D)}(\ord_{S})=k\cdot \lambda^{-1}.
\end{equation}

For $\bN \ni k\gg 1$, let $v_k=k \cdot w_{1/k}$. Then $v_k=d \cdot \ord_{E_k}$ is a multiple of a divisorial valuation $\ord_{E_k}$ for some $d\in \bZ_{>0}$. As a valuation, we can describe $v_k$ explicitly as follows (see \cite[(57)]{Li17}). For any $f\in H^0(S, -m \lambda(K_S+B))$, 
\begin{equation}
v_k(f)=k m +\ord_{S^{(1)}_0}(\bar{f}).
\end{equation}
where $\bar{f}$ is the meromorphic section of $m\cL\rightarrow \cS$ obtained by pulling back $f$ via the map $$(\cS\backslash S^{(1)}_0, m\cL) \cong (S\times\bC^*, m \cdot p_1^* L)\stackrel{p_1}{\longrightarrow} (S, mL). $$
Note that $v_{k}$ can also be defined as the restriction of $\wt_{k\xi_0- \eta}$ from $\bC(\cX^{(1)})$ to $\bC(X)$ by using the $\bC^*$-equivariant isomorphism $\cX^{(1)}\backslash X^{(1)}_0\cong X\times \bC^*$ (see \cite[page 3184-3185]{Li17}). 

\begin{lem}\label{l-n1}
Notations as above, 
for $k\gg 1$, the divisor $E_k$ corresponding to $v_{k}$ is a Koll\'ar component with an associated model $Y_k\rightarrow (X,D)$. Moreover, the special test configuration $(\cX^{(1)}, \cD^{(1)})$ is given by the special test configuration associated to $E_k$ (in the sense of Definition \ref{d-model}) up to a base change. In particular $(\cS^{(1)}, \mathcal{B}^{(1)})$ can be recovered by the model $E_k\rightarrow Y_k\rightarrow (X, D)$.
\end{lem}
\begin{proof} 
%Since $(X,D)$ is K-semistable, and $\Fut(\cX_1,\cD_1)=0$, by \cite{Fuj17} we know that $S$ is the exceptional divisor of a plt blow up $Y_1\to (X,D)$.
For simplicity, we denote $L=-\lambda(K_S+B)$. By \cite[Proposition 2.15]{BHJ15} (see also Lemma \ref{l-filtration}), we know that $\mathcal{X}^{(1)}$ is given by
$${\rm Spec}_{\bC[t]}\left( \bigoplus_{m\in \mathbb{N}}\left(\bigoplus_{j \in \mathbb{Z}}t^{-j}\cF^{j}H^0(S, mL)\right)\right)=:{\rm Spec}_{\bC[t]}(\cR^{(1)})$$
where $\cF^{j}H^0(S, mL)$ is given by:
$$
\cF^j H^0(S, m L)=\left\{s\in H^0(S, mL)\;|\; t^{-j}\bar{s}\in H^0(\cS, m\cL)\right\}.
$$
Therefore $X^{(1)}_0$ is isomorphic to
$${\rm Spec}\left( \bigoplus_{j\in \mathbb{N}}\left(\bigoplus_{m \in \mathbb{N}}\cF^j H^0(S, mL)/ \cF^{j+1}H^0(S, mL) \right)\right),$$
and the $(\mathbb{C}^*)^2$-action on $X^{(1)}_0$ is induced by the two gradings.

On the other hand, 
$f\in \cF^{j}H^0(S, m\lambda(-K_S-B))$ if and only if $\ord_{S^{(1)}_0}(\bar{f})\ge j$ which by \cite[(57)]{Li17} is equivalent to 
\begin{eqnarray*}
v_k(f)&=&mk+\ord_{S^{(1)}_0}(\bar{f})\ge mk+j \ . 
\end{eqnarray*}
In other words, the {valuation ideal} $\fa_p(v_k)$ of $v_k=d\cdot \ord_{E_k}$ is determined by:
$$f\in H^0(S,m\lambda(-K_S-B))\cap \fa_p(v_k) \mbox{ if and only if } f\in \cF^{p-mk}H^0(S, m\lambda(-K_S-B)).$$
 
Since $v_k\in \Val_{X,x}$ is $\mathbb{C}^*$-invariant, we have the identity:
$${\rm gr}_{v_k} R=\bigoplus_{p\in d\bZ} \bigoplus_{m} \cF^{p-mk}H^0(S, mL)/\cF^{p+1-mk}H^0(S, mL).$$
Let $\xi_{\frac{1}{k}}:=\xi_0-\frac{1}{k}\eta$. For an element  
$$\bar{f}\in \cF^j H^0(S, m L)/\cF^{j+1} H^0(S, m L),$$ 
its weight vector is $\alpha=(m,-j)$ and $
\la \xi_{\frac{1}{k}}, \alpha\ra=m+\frac{j}{k}
$.
Thus 
$${\rm Proj}({\rm gr}_{\ord_{E_k}}R)\cong {\rm Proj}({\rm gr}_{v_k}R)$$ is the quotient of $X^{(1)}_0$ by the $\mathbb{C}^*$-action generated by $\xi_{\frac{1}{k}}$ (see Definition \ref{d-quotient}). So we have:
$(E_k, B_k):=\big((X^{(1)}_0, D^{(1)}_0)\setminus\{x^{(1)}\}\big)/\langle \xi_{\frac{1}{k}}\rangle$ (where $x^{(1)}$ is the vertex and $B_k$ includes the orbifold locus) and $E_k$ can be extracted over $X$. 
Since $(X^{(1)}_0, D^{(1)}_0)$ has klt singularities, $(E_k, B_k)$ is a log Fano variety which has klt singularities and hence is a Koll\'ar component over $X$ by the inversion of adjunction. 

To see that last statement, note that we can rewrite $\cR^{(1)}$ as:
\begin{eqnarray*}
\cR^{(1)}&=&\bigoplus_{j\in \bZ}\bigoplus_{m\in \bN} t^{-j} \fa_{mk+j}(v_k)\cap  H^0(S, mL),\\
&=&\bigoplus_{p\in\bZ}\bigoplus_{m\in \bN}t^{-p+mk}\fa_p(v_k)\cap H^0(S,mL),
\end{eqnarray*}
which is isomorphic to the extended Rees algebra of $\fa_\bullet(v_k)$:
\begin{equation}
\bigoplus_{p\in \bZ}\bigoplus_{m\in \bN} t^{-p}\fa_p(v_k)\cap H^0(S, mL)=\bigoplus_{p\in \bZ}t^{-p}\fa_p(v_k).
\end{equation}
Indeed, it is easy to verify that the map $t^{-p+mk}f\mapsto t^{-p}f$ for any $f\in \fa_p\cap H^0(S, mL)$ is an isomorphism of the two algebras. On the other hand, the extended Rees algebra of $\ord_{E_k}$ is given by:
\begin{eqnarray*}
\bigoplus_{q\in \bZ} u^{-q}\fa_q(\ord_{E_k})=\bigoplus_{p\in d\bZ}u^{-\frac{p}{d}}\fa_{p}(v_k).
\end{eqnarray*}
From this we see that $\cX^{(1)}=\cY\times_{\bC, t\mapsto t^d}\bC$ where $\cY$ is the test configuration associated to $\ord_{E_k}$ in the sense of Definition \ref{d-model}.
\end{proof}
In the proof of Lemma \ref{l-n1}, there is a rank 2 torus $(\mathbb{C}^*)^2$ \Red{acting} on $X^{(1)}_0$, such that if we let $\xi_0$ be the coweight vector $(1,0)$, then $X^{(1)}_0/ \langle\xi_0\rangle \cong S^{(1)}_0$, and the action by the coweight $(0,1)$ is induced by the action on $S^{(1)}_0$ from the test configurational $\mathcal{S}^{(1)}$ structure.  We construct a ray $\xi_\epsilon=\xi_0-\epsilon\eta$, where $\eta$ corresponds the action with coweight $(0,1)$. Then any $\xi_\epsilon$ gives a quasi-monomial valuation $\wt_\epsilon$ on $X^{(1)}_0$ (see \eqref{eq-wtxi}). Moreover, for $\epsilon\in[0,\epsilon^*)$, it also induces a sequence of quasi-monomial valuations $w_{\epsilon}$ in $X$ which is contained in $\Val_{X,x}$ (see \cite[Proof of Theorem 3.5]{LX17}). Our proof in Lemma \ref{l-n1} just gives a verification of the divisorial valuation, which can be easily extended to the general case.

Furthermore, as proved in \cite[Lemma 6.20]{Li17} or \cite[Section 2.4]{LX17}, if we define $f(\epsilon):=\hvol(w_\epsilon)$, then it  is a smooth convex function on $[0,\epsilon^*)$ with $0<\epsilon^*\ll 1$ such that $f(0)=\hvol(v)$ and
\begin{eqnarray*}
 f'(0)&= &\left.\frac{d}{d\epsilon}\hvol_X(w_\epsilon)\right|_{\epsilon=0}=\left.\frac{d}{d\epsilon}\hvol_{X_0}(\wt_\epsilon)\right|_{\epsilon=0}\\
  & =&\left.\frac{d}{d\epsilon}\right|_{\epsilon=0}\hvol_{X_0}(\xi_0-\epsilon\eta)\\
  &=& C\cdot \Fut (\cX^{(1)},\cD^{(1)}, \xi_0; \eta^{(1)}),
\end{eqnarray*}
where the last identity follows from \eqref{eq-Futhvol} and the constant 
$$C=n \cdot A_{X_0}(\wt_{\xi_0})^{n-1}\cdot \vol(\xi_0)>0.$$

\begin{lem}\label{l-model}
For $k$ sufficiently large, the model $Y_k\to X$ extracting $E_k$ can be degenerated along $\cX^{(2)}$ to obtain a model $\mu: \cY^{(2)}_k \to \cX^{(2)}$ over ${\bC}$ with an exceptional divisor $\cE_k$ such that the following properties hold true:
\begin{enumerate}
\item There is the following isomorphism which is equivariant with respect to the $\bC^*$-action generated by $\eta^{(2)}$:
$$(\cY^{(2)}_k, \cE^{(2)}_k)\times_{\bC}\bC^*\cong (Y_k, E_k)\times\bC^*.$$ 
\item
$(\cY^{(2)}_k, \mu_*^{-1}\cD^{(2)}+\cE^{(2)}_k+(\cY^{(2)}_k)_t)$ is log canonical for any $t\in \bC$ (i.e., in the terminology of \cite[Definition 2.2]{Kol18}, $(\cY^{(2)}_k, \mu_*^{-1}\cD^{(2)}+\cE^{(2)}_k)$ a locally stable family over $\bA^1_{\mathbb{C}}$).
% that is, 
\end{enumerate}
\end{lem}
\begin{proof}
For a fixed sufficiently large $k$, denote by $I_k$ the $\fm$-primary ideal over $x\in X$ induced by $E_k$ which is the push forward of $\mathcal{O}(-mE_k)$ for a fixed sufficiently divisible $m$.
 Let $$c_k=\lct(I_k; X, D)=:\lct(I_k)$$ 
 be its log canonical threshold. Then because $E_k$ is a Koll\'ar component, we have:
$$f\left(\frac{1}{k}\right)=\hvol(\ord_{E_k})=\mult(I_k) \cdot c_k^n.$$
Note that because of the rescaling invariance of the normalized multiplicities
$\mult(I_k)\cdot\lct^n(I_k),$
we can replace $I_k$ by its powers and the normalized multiplicities do not change, so we do not specifically denote $m$.

\bigskip

Since  $f'(0)=C\cdot {\rm Fut}(\mathcal{X}^{(1)},\cD^{(1)}, \xi_0; \eta)=0$, we have
 $$f\left(\frac{1}{k}\right)=f(0)+O\left(\frac{1}{k^2}\right).$$
 
 Fix $k$, for each $l\ge 1$, as in \cite[Lemma 4.1]{LX16}, we can construct a graded sequence of ideals $\frak{A}_\bullet=\{\frak{A}_l\}$ on $\cX^{(2)}$ such that 
 $$\frak{A}_l\otimes_{\bC[t]}\bC[t, t^{-1}]\cong I_k^{l}[t,t^{-1}] \mbox{\ \  \ and\  \ \ }\frak{A}_l\otimes_{\bC[t]}(\bC[t]/(t))\cong  \bin(I_k^l),$$ 
 where $\{\bin(I_k^l)\}$ is the graded sequence of ideals consisting of initial ideals of the sequence $\{I_k^l\}_l$ for the $\bC^*$-degeneration of $X$ to $X^{(2)}_0$.
To simplify the notations, we just denote 
$$\fb_{k, \bullet}=\{\fb_{k,l}\}_l=\{\bin(I_k^l)\}.$$ 

\medskip

%\noindent {\bf Claim}:
\begin{claim}\label{c'}
For any $\epsilon>0$, we can find $k$ sufficiently large and $\delta$ sufficiently small satisfying:
\begin{equation}\label{e-doublecontrol}
A(E_k; X, D+ c_k'I_k)<\epsilon/2\qquad \mbox{and} \qquad c_k'< {\rm lct}( \fb_{k,\bullet}; X^{(2)}_0, D^{(2)}_0)
\end{equation}
with $c_k':=c_k(1-\delta)$.
\end{claim}

\begin{proof}[Proof of Claim \ref{c'}]
To prove the claim, we first note that, by using $A(E_k, X, D+c_k I_k)=0$ and identity \eqref{eq-AXDvk}:
\begin{eqnarray*}
A(E_k; X, D+(1-\delta)c_k I_k)=\delta \cdot A_{(X,D)}(E_k)=\delta\cdot k \cdot A_{(X,D)}(\ord_S).%\\
%& =&\delta\left(k\cdot A_{(X,D)}(\ord_X)+A_{(X,D)}(v_0)\right).
\end{eqnarray*}
On the other hand, since $(X^{(2)}_0,D^{(2)}_0)$ is K-semistable by Lemma \ref{l-semistable}, we know that $f(0)=\vol(x^{(2)}, X^{(2)}_0, D^{(2)}_0)$ (see Theorem \ref{p-m=k}), where $x^{(2)}$ is the vertex. 
Therefore, 
\begin{eqnarray*}
f(0)&\le & \lct(\fb_{k,\bullet}; X^{(2)}_0, D^{(2)}_0)^n\cdot\mult(\fb_{k,\bullet})\\
 &\le & c_k^n\cdot \mult(I_k)\\
 &=& f\left(\frac 1 k\right)=f(0)+O\left(\frac{1}{k^2}\right),
 \end{eqnarray*}
 where we have used Proposition \ref{p-liu} for the first inequality, and the non-increasing of log canonical thresholds under specialization as well as $\mult(\fb_{k,{\bullet}})=\mult(I_k)$ for the second inequality.
 
We get the inequality:
$$\frac{ \lct(\fb_{k,\bullet}; X^{(2)}_0, D^{(2)}_0)}{c_k}\ge \left(\frac{f(0)}{f(1/k)}\right)^{1/n}.$$
Since $(1+O(\frac{1}{k^2}))^{\frac{1}{n}}$ is also of the order $(1+O(\frac{1}{k^2}))$, for any fixed $\epsilon$, there exists $K_0\gg 0$ such that for any $k\ge K_0$,
$$\left(\frac{f(0)}{f(1/k)}\right)^{1/n}\ge 1-\frac{\epsilon}{4 k\cdot A_{X, D}(\ord_S)}.$$
Now if we choose $\delta$ to be:
$$\delta=\frac{\epsilon}{2k\cdot A_{(X,D)}(\ord_S)},$$
then $c_k'=(1-\delta )\cdot c_k <  \lct(\fb_{k,\bullet}; X^{(2)}_0, D^{(2)}_0 )$ and 
\begin{eqnarray}\label{e-discrepancy}
A(E_k, X, D+(1-\delta )c_k I_k)=\epsilon/2.
\end{eqnarray}
\end{proof}
We may assume $\epsilon$ is less than 1.
It follows from 
{Claim \ref{c'} that}
\begin{equation}
A(E_k\times\bC; \cX^{(2)}, \cD^{(2)}+c'_k \frak{A}_\bullet)<\epsilon/2\qquad \text{ and } \qquad c'_k< {\rm lct}(\frak{A}_\bullet; \cX^{(2)}, \cD^{(2)}+X_0^{(2)}),
\end{equation}
where we used  the inversion of adjunction for the second inequality. 
We can then apply \cite[Corollary 1.4.3]{BCHM10} to precisely extract an irreducible divisor $\cE^{(2)}_k$ to obtain a birational morphism $\mu\colon \cY^{(2)}_k \to \mathcal{X}^{(2)}$ whose restriction over $X\times \mathbb{C}^*$  is the divisor $E_k\times \mathbb{C}^*$ and $-\cE^{(2)}_k$ is ample over $\mathcal{X}^{(2)}$.

%Let $\mu\colon \cY^{(2)}_k\to \mathcal{X}^{(2)}$ denote the family obtained \Red{above} with an irreducible divisor $\mathcal{E}^{(2)}_k$ for which we may assume $-K_{\cY^{(2)}_k}-\mu_*^{-1}\cD^{(2)}-\mathcal{E}^{(2)}_k$ is ample over $\mathcal{X}^{(2)}$. 
Moreover, as  $({\cY^{(2)}_k}, \mu_*^{-1}\cD^{(2)}+(1-\epsilon )\mathcal{E}^{(2)}_k+Y_0^{(2)})$ is log canonical, by ACC of log canonical thresholds (\cite[Theorem 1.1]{HMX14}), we may choose $\epsilon$ to be  sufficiently small and independent of $k$ such that  $({\cY^{(2)}_k}, \mu_*^{-1}\cD^{(2)}+\mathcal{E}^{(2)}_k+Y_0^{(2)})$ is log canonical. 
\end{proof}

There is a $\mathbb{C}^*\times \mathbb{C}^*=\la \xi_0\ra\times \la \eta^{(2)}\ra$-action on $X^{(2)}_0$. Note that $[\xi_0, \eta^{(2)}]=0$.
The ideals $\{\fb_{k,\bullet}\}$ is $(\mathbb{C}^*)^2$-equivariant. In fact, by definition it is clearly equivariant with respect to $\la\eta^{(2)}\ra$. It is also equivariant with respect to the first factor because $E_k$ is $\la \xi_0\ra$-invariant and $\cX^{(2)}$ is $\la\xi_0\ra$-equivariant.

\begin{equation}\label{eq-CD2}
\xymatrix@1 @R=1.2pc @C=1pc
{
(X^{(2)}_0, D^{(2)}_0) \ar@{~>}_{(\cX'^{(2)}, \cD'^{(2)})}[dddd]
 \ar@{-->}[dr] &   & &  &(X, D) \ar_{ (\cX^{(2)}, \cD^{(2)})\longleftarrow \cY^{(2)}_k\longleftarrow\cE^{(2)}_k}@{~>}[llll] \ar^{(\cX^{(1)}, \cD^{(1)})\leftarrow \cY_k\leftarrow \cE_k=E_k\times\bA^1}@{~>}[dddd]  \ar@{-->}[ld] & \ar[l] Y_k\leftarrow E_k
 \\
& (S^{(2)}_0,B^{(2)}_0) \ar@{~>}_{(\cS'^{(2)},\mathcal{B}'^{(2)})}[dd] &  &  (S, B) \ar_<<<<<<<{(\cS^{(2)}, \mathcal{B}^{(2)})}@{~>}[ll] \ar^{(\cS^{(1)}, \mathcal{B}^{(1)})}@{~>}[dd] & & \\
&   &   &\\
&                              (S'_0, B'_0)     &               &(S^{(1)}_0, B^{(1)}_0) \ar@{~>}^{\stackrel{}{(\cS'^{(1)},\mathcal{B}'^{(1)})}}[ll]& & \\
(X'_{0}, D'_{0}) \ar@{-->}[ru] & & & & (X^{(1)}_0, D^{(1)}_0)\ar@{~>}^{(\cX'^{(1)}, \cD'^{(1)})}[llll] \ar@{-->}[lu] & \ar[l] Y_{k,0}\leftarrow E_k .
}
\end{equation}

%\begin{equation}
%\xymatrix@1 @R=3pc @C=3 pc
%{
%(X^{(2)}_0, D^{(2)}_0) \ar[ddd]
% \ar[dr] &    &  &(X, D) \ar_{(\cX^{(2)}, \cD^{(2)})}[lll] \ar^{(\cX^{(1)}, \cD^{(1)})}[ddd]  \ar[ld]\\
%& (S^{(2)}_0,B^{(2)}_0) \ar[d]   &  (S, B) \ar_{(\cS^{(2)}, \mathcal{B}^{(2)})}[l] \ar^{(\cS^{(1)}, \mathcal{B}^{(1)})}[d] &\\
%&                              (S'_0, B'_0)                    &(S^{(1)}_0, B^{(1)}_0) \ar[l]& \\
%(X_{(0,0)}, D_{(0,0)}) \ar[ru] &  & & (X^{(1)}_0, D^{(1)}_0)\ar^{}[lll] \ar[lu].
%}
%\end{equation}

\medskip

Now we apply the family version of the construction first introduced in \cite[Section 2.4]{LX16}, to conclude that the model $\cY^{(2)}_k\to \mathcal{X}^{(2)}$ with relative anti-ample $\cE^{(2)}_k$ over $\mathcal{X}^{(2)}$ yields a degeneration of $\mathcal{X}^{(2)}$ which gives a  family $(\mathfrak{X}, \frak{D})$ over $\mathbb{C}^2$, whose restriction over $ (\mathbb{C}^*)^2\subset \mathbb{C}^2$ is isomorphic to $(X,D)\times (\mathbb{C}^*)^2$.
More precisely, if we assume $\cX^{(2)}={\rm Spec}_{\bC[s]}(\cR^{(2)})$ and define the extended Rees algebra: 
\begin{equation}
\frak{R}=\bigoplus_{m\in\bZ} \fa_m(\ord_{\cE_k}) s^{-m} \subset \cR^{(2)}[s, s^{-1}],
\end{equation}
where as before $\fa_m(\ord_{\cE_k})=\{f\in \cR^{(2)},\; \ord_{\cE_k}(f)\ge m\}$. Then $\frak{X}={\rm Spec}_{\bC[t,s]}(\frak{R})$ and $\frak{D}$ is the divisor on $\frak{X}$ induced by $\cD^{(2)}$. Using the fact that  $({\cY^{(2)}_k}, \mu_*^{-1}\cD^{(2)}+\mathcal{E}^{(2)}_k+({\cY^{(2)}_k})_t)$ is log canonical for any $t\in \bC$ (see Lemma \ref{l-model}.2), %we know these two are indeed weakly special test configurations, 
we know that $(\frak{X}, \frak{D}+\frak{X}\times_{\bC^2}(\{t\} \times \bC))\times_{\bC^2}(\bC\times \{0\})$ is log canonical (see Lemma \ref{l-tc}). %that is $K_{\frak{X}}+\frak{D}$ is $\mathbb{Q}$-Cartier and the restricting pair on each fiber is semi-log-canonical. 

\medskip

Using the basic property of the Rees algebra (see e.g. \cite[Section 4.1]{LX16}), we see that
\begin{eqnarray*}
(\frak{X}, \frak{D})\times_{\bC^2} (\bC\times\{1\}) \cong (\cX^{(2)}, \cD^{(2)}).
\end{eqnarray*}
Moreover, we claim that:
\begin{eqnarray*}
(\frak{X}, \frak{D})\times_{\bC^2} (\{1\}\times\bC) \cong (\cX^{(1)}, \cD^{(1)}).
\end{eqnarray*}
This holds true if the morphism $\fa_m(\ord_{\cE_k})=\mu_*\cO(-m \cE_k)\rightarrow \fa_m(\ord_{E_k})=(\left.\mu\right|_{Y_k})_*\cO(-mE_k)$ is surjective which follows from the vanishing $R^1\mu_*(-m\cE_k)=0$.

The restrictions of $(\frak{X}, \frak{D})$ over the two axes $\mathbb{C}\times \{0\}$ and $\{0 \}  \times \mathbb{C}$ respectively give test configurations $(X^{(1)}_0, D^{(1)}_0)$ and $(X_0^{(2)}, D^{(2)}_0)$ with the same central fiber $(X'_0, D'_0)$.  We know these two test configurations are indeed weakly special because $(\frak{X}, \frak{D}+\frak{X}\times_{\bC^2}(\{0\} \times \bC))\times_{\bC^2}(\bC\times \{0\})$.

The $\la\xi_0\ra$-action on $(X, D)$ extends naturally to $(\mathfrak{X}, \frak{D})$ over $\mathbb{C}^2$. Moreover, $K_{\frak{X}}+\frak{D}$ is $\bQ$-Cartier and admits a $(\bC^*)^2$-equivariant nowhere-vanishing section $\frak{s}\in |m(K_{\frak{X}}+\frak{D})|$. Then we can take the quotient of the action $(\frak{X}, \frak{D})$ by the $\la \xi_0\ra$-action to get a pair $(\frak{S}, \widetilde{\underline{B}})$. Its restrictions over the two axes $\mathbb{C}\times \{0\}$ and $\{0 \}  \times \mathbb{C}$ respectively give  test configurations $(S^{(1)}_0, B^{(1)}_0)$ and $(S_0^{(2)},B^{(2)}_0)$ with the same central fiber $(S'_0, B'_0)$. Because the generalized Futaki invariants are defined by the intersection numbers, we know the generalized Futaki invariant of the test configuration $(\frak{S}, \widetilde{\underline{B}})\times_{\mathbb{C}^2}(\bC\times\{0\})$ degenerating  $(S^{(1)}_0, B^{(1)}_0)$ to $(S'_0, B'_0)$ is 0 since the nearby fibers $(\frak{S}, \widetilde{\underline{B}})\times_{\mathbb{C}^2}(\bC\times\{t\})$ $(t\neq 0)$ all have generalized Futaki invariants 0, and the same is true for the test configuration $(\frak{S}, \widetilde{\underline{B}})\times_{\mathbb{C}^2}(\{0 \}  \times \mathbb{C})$ degenerating $(S_0^{(2)},B^{(2)}_0)$ to  $(S'_0, B'_0)$.

Then the central fiber $(S'_0, B'_0)$ will automatically be a log Fano variety since otherwise it follows from \cite[Theorem 7]{LX14} that we can construct a special test configuration of $(S^{(1)}_0, B^{(1)}_0)$ with a strictly negative Futaki invariant, which contradicts to the K-semistability of $(S^{(1)}_0, B^{(1)}_0)$ by Lemma \ref{l-semistable}.

Thus this completes the proof of  Theorem \ref{t-doubledeg}.
%\end{proof}

{\subsubsection{Application to torus-equivariant K-polystability}}
\begin{lem}\label{l-rank}
Assume $(S,B)$ is an $(n-1)$-dimensional K-semistable log Fano pair. Let $(\cS, \mathcal{B})$ be a special test configuration {with central fibre} $(S_0,B_0)$ such that $\Fut(\mathcal{S},\mathcal{B})=0$. If $S$ admits a torus $T_S\cong (\mathbb{C}^*)^{d-1}$-action, then $(S,B)$ and $(S_0,B_0)$ admits a common K-semistable degeneration $(S_1,B_1)$ with a $T_S\times \mathbb{C}^*\cong (\mathbb{C}^*)^{d}$-action, which extends the $T_S$-action on $(S,B)$ as well the $\bC^*$-action on $(S_0,B_0)$.
\end{lem}
\begin{proof}
By Lemma \ref{l-semistable}, we know $(S_0,B_0)$ is K-semistable. Fix a sufficiently divisible $\lambda$. By Lemma \ref{l-n1}, for $k\gg 1$, the special degeneration induces a Koll\'ar component $E_k$ over the cone 
$$E_k\rightarrow Y_k\rightarrow (X, D)=C(S, B; -\lambda(K_S+B)). $$ 
The cone $(X,D)$ admits a $T\cong T_S\times \bC^*$-action, where the first factor $T_S$-action is induced from the $T_S$-action on $(S,B)$ and the second factor $\bC^*$-action comes from the natural rescaling  on the cone $(X,D)$.

%Pick an arbitrary integral coweight  vector $\eta'$, which generates a subgroup $\la \eta'\ra\cong \bC^*\subset \tilde{T}$. 
Consider the {valuation ideal} $I_k=\fa_m(\ord_{E_k})$ for $m\gg 1$,
and its equivariant  degeneration $\{\fb_{k,l}\}_l$ of $ \{I_k^l\}$ on the fiber of $X_{\bC^d}:=X\times \mathbb{C}^d$ over $0\in \bC^d$ with respect the $T$-action given by $\widetilde{\fb_{k,l}}$. That is to say, $\widetilde{\fb_{k,l}}$ is a $T$-equivariant ideal on $X_{\bC^d}$ which is flat over $\bC^d$, whose restriction over in $\underline{1}=(1,...,1)\in (\bC^*)^d\subset \mathbb{C}^d$ is $I_k^l$, and over $0\in \bC^d$ is $\fb_{k,l}$.

Then as before, we know there is a smooth function $f$ on $[0,\epsilon^*)$ with $0<\epsilon^*\ll 1$ such that
$$f\left(\frac{1}{k}\right)=\mult(I_k)\cdot \lct(I_k)^n \qquad \mbox{ and \ \  }f'(0)=0.$$
Then by exactly the same argument  as in Claim \ref{c'}, we know that  for any $\epsilon>0$, we can pick $k\gg 1$ and $c'_k$ such that
$$A(E_k, X, D+ c_k'I_k)<\epsilon\qquad \mbox{and} \qquad c_k'< {\rm lct}(\fb_{k,\bullet}; X_0, D_0).$$

Therefore, we can construct a $T$-equivariant birational model $\mu_d\colon \cY_k\to X_{\bC^d}$ which over $\underline{1}$ yields $Y_k\to X$, and ${\rm Ex}(\mu_d)=\cE$ is an irreducible divisor which is anti-ample over $X_{\bC^d}$.  For any point $t\in \bC^d$, we can find the coordinate hyperplanes $H_1,...,H_d$ passing through $t$. By inversion of adjunction, $(X_{\bC^d}, D_{\bC^d}+\sum^d_{i=1}H_i+c_k'\cdot\widetilde{\fb_{k,l}})$ is log canonical, which implies that $(\cY_k, \cE+\mu_*^{-1}D_{\bC^d}+\mu^*\sum^d_{i=1}H_i)$ is log canonical by ACC of log canonical thresholds (see \cite{HMX14}) as before. This implies that for any $t\in \bC^d$, any component of the fiber $(\cY_k)_t$ is of dimension $\dim(X)$ by \cite[Proposition 39]{dFKX17} as it is contained in $\mu^*\sum^d_{i=1}H_i$. %Since $\cY_k$ is Cohen-Macaulay, $\cY_k\to \bC^d$ is flat. 
Moreover, $(\cY_{k})_t\to X$ is birational, since if $(\cY_k)_t$ has another component, then it is contained in $E$ and $\mu^*\sum^d_{i=1}H_i$, thus violates  \cite[Proposition 39]{dFKX17} again. 

Denote by  $\fa_i={\mu_d}_*\cO_{\cY_k}(-i\cE)$. For any $j\in \mathbb N$, we have an exact sequence 
$$0\to \cO_{\cY_k}(-(i+1)\cE)\to \cO_{\cY_k}(-i\cE)\to Q_i\to 0$$ for a sheaf $Q_i$ supported on $\cE$. Since $\cE$ is $\mathbb{Q}$-Cartier, we know $Q_i$ is Cohen-Macaulay (see \cite[Proof of Proposition 5.26]{KM98}), thus it is flat over $\bC^d$ as it is equi-dimensional. By the vanishing theorem, we can pushforward the above exact sequence by $\mu_d$ to get
$$0\to \fa_{i+1}\to \fa_i \to {\mu_d}_*(Q_i)\to 0.$$ 
Since for any $t\in \bC^*$, $R^j{(\mu_d)_t}_*(Q_i)_t=0$ for any $j>0$, by base change theorem (see e.g. \cite[Theorem III.12.11]{Har77}), this implies  
$${\mu_d}_*(Q_i)\otimes k(t)={\mu_d}_*(Q_i\otimes k(t))= {(\mu_d)_t}_*(-i\cE_t)/{(\mu_d)_t}_*(-(i+1)\cE_t)$$ and we inductively get $\fa_i\otimes k(t)={(\mu_d)_t}_*\cO_{(\cY_k)_t}(-i\cE_t)$ for all $i\in \bN$. Then we can take the relative extended $k[t_1,...,t_d]$-Rees algebra $\bigoplus_{i\in \mathbb{Z}} \fa_i\cdot t^{-i}$, which yields is $\til{T}$-equivariant family $\frak X_{d+1}$ over $\bC^{d+1}$.
The above discussion says the construction commutes with any base change, i.e. for a any $t\in \bC^d$, if we consider the base change $\frak X_{d+1}\times_{\bC^{d+1}}(\{t\}\times \bC)$, we get exactly the degeneration induced by $(\cY_k)_t\to X$. In particular, we can choose $(X_1,D_1)$ to be the fiber over $0\in \bC^{d+1}$ and $(S_1,B_1)$ is its $\bC^*$-quotient. And $(S_1,B_1)$ is K-semistable, since it is a special degeneration of $(S,B)$ with generalized Futaki invariant 0.
\end{proof}

\subsection{Proof of main results for log Fano pairs}
\begin{proof}[Proof of Theorem \ref{t-uniqueness}]
Given  a K-semistable log Fano pair $(S^{(0)},B^{(0)}):=(S,B)$. If it is not K-polystable, then by \cite{LX14} we know it has special degenerations to  log Fano pairs which are not isomorphic to $(S,B)$, with  generalized Futaki invariant 0. Let $(S^p,B^p)$ be a log Fano pair, which among all possible special degenerations of $(S,B)$, admits a faithful torus action of the maximal dimension. We claim $(S^p,B^p)$ is K-polystable.  If this is not true, by Lemma \ref{l-rank}, $(S^p,B^p)$ has a K-semistable degeneration $(S^{( 1)},B^{(1)})$ which admits a faithful torus action of a larger dimension. However, by the proof of Lemma \ref{l-semistable}, we can indeed degenerate $(S,B)$ to $(S^{( 1)},B^{(1)})$, which is a contradiction. 

%Furthermore, $(S^{(1)},B^{(1)})$ is also K-semistable by Lemma \ref{l-semistable}. It follows from Lemma \ref{l-doubledeg} that any special degeneration $(S^{\langle 1\rangle},B^{\langle 1\rangle })$ of $(S^{(1)},B^{(1)})$ with the generalized Futaki invariant 0  will yield a common special degeneration $(S^{(2)},B^{(2)})$ of $(S^{\langle 1\rangle},B^{\langle 1\rangle })$ and $(S,B)$ with the generalized Futaki invariant 0. Moreover, if there is $T$-action on $(S^{(1)},B^{(1)})$, then it also acts on $(S^{(2)},B^{(2)})$ such that the degeneration is $T$-equivariant. 

%By Lemma \ref{l-rank}, if $S^{(i)}$ and $S^{(i+1)}$ are not isomorphic, then the dimension of the maximal torus effectively acting on $S^{(i+1)}$ is strictly larger than that for $S^{(i)}$.   Thus this degeneration process has to terminate after $r\le \dim S$ steps. Then the end product is K-polystable and is also a special degeneration of $ (S,B)$.

The uniqueness directly follows from  Theorem \ref{t-doubledeg}, as any test configuration $(\cS,\mathcal{B})$ which degenerates $(S,B)$ to a K-polystable log Fano pair $(S_0,B_0)$ automatically satisfies ${\rm Fut}(\cS,\mathcal{B})=0$.
\end{proof}

%\begin{proof}[Proof of Theorem \ref{t-uniqueness}]
%Given  a K-semistable log Fano pair $(S^{(0)},B^{(0)}):=(S,B)$. If it is not K-polystable, then by \cite{LX14} we know it has a special degeneration to a log Fano pair $(S^{(1)},B^{(1)})$ which is not isomorphic to $(S,B)$, with the generalized Futaki invariant being 0. Furthermore, $(S^{(1)},B^{(1)})$ is also K-semistable by Lemma \ref{l-semistable}. It follows from Lemma \ref{l-doubledeg} that any special degeneration $(S^{\langle 1\rangle},B^{\langle 1\rangle })$ of $(S^{(1)},B^{(1)})$ with the generalized Futaki invariant 0  will yield a common special degeneration $(S^{(2)},B^{(2)})$ of $(S^{\langle 1\rangle},B^{\langle 1\rangle })$ and $(S,B)$ with the generalized Futaki invariant 0. Moreover, if there is $T$-action on $(S^{(1)},B^{(1)})$, then it also acts on $(S^{(2)},B^{(2)})$ such that the degeneration is $T$-equivariant. 
%\end{proof}

\begin{proof}[Proof of Theorem \ref{t-equivariant}]It is known from \cite{LX16} that to check K-semistablity, we only need to check the $T$-equivariant special test configurations. Then from K-semistability to K-polystability, it follows from Lemma \ref{l-rank}.
\end{proof}
%\Green{
%\begin{rem}Discuss the mistakes in \cite{LX16, LX17},
%\end{rem}}

\section{General case of log Fano cones}
In Section 4.1, we will generalize the techniques in Section \ref{sec-logpair} to the case of log Fano cones. This allows us to get weakly special test configurations with isomorphic central fibres and generalized Futaki invariants 0, under similar assumption as in Theorem \ref{t-doubledeg}. In Section 4.2, we prove that these weakly special test configurations are indeed special. We prove this fact by generalizing the last step in \cite{LX14} to the setting of  log Fano cone singularities, including the irregular case.  We complete the proof of Donaldson-Sun's conjecture (Theorem \ref{t-DSconj}) and Theorem \ref{t-uniquecone} on existence/uniqueness of K-polystable degenerations in Section \ref{ss-completionofproof}.

\subsection{Common degenerations of log Fano cones}\label{ss-fanocone}

%In this section, we generalize the arguments in Section \ref{sec-logpair} to log Fano cone singularities, i.e., including the irregular case as well. %The strategy is similar, but there are two new technical issues: firstly, due to the possible irregular case, we need to deal with the approximating process of a non-rational valuations by rational ones; secondly, we have to develop the last step of \cite{LX14} to compare weak special and special  test configurations. 
%\medskip

Fix a K-semistable log Fano cone $(X, D, \xi_0)$ with a torus action by $T\cong (\mathbb{C}^*)^r$. 
%generated by $\xi_0$. 
Then $\wt_{\xi_0}$ is a minimizer of $\hvol_{X,D}$ by Theorem \ref{p-m=k}. Assume that $(\cX^{(i)}, \cD^{(i)}, \xi_0; \eta^{(i)})$  $(i=1,2)$ are two special degenerations of $(X, D, \xi_0)$ to $(X^{(i)}_0, D^{(i)}_0, \xi_0), (i=1,2)$ respectively.
%More precisely, there is a torus $(\mathbb{C}^*)^{r+1}=(\mathbb{C}^*)^{r}\times \mathbb{C}^*$ acting on  $\cV_i$ whose restriction on the first factor is compatible with the $(\mathbb{C}^*)^{r}$-action on $X$, and restriction on the second factor gives the special test configuration structure.
Recall that $\xi_0$ on $\cX^{(i)}$ is just given by the natural extension of $\xi_0$ on $X\times \mathbb{C}^*$.  By assumption $\eta^{(i)}$ has an integral coweight which can be written as the form $(\cdot,1)$ with respect to the decomposition of $\tilde{T}:=T\times \mathbb{C}^*\cong(\mathbb{C}^*)^{r+1}$. Note that the central fibers $(X^{(i)}_0, D^{(i)}_0)$ $(i=1,2)$ admit $\tilde{T}$-actions generated by $T$ and $\la\eta^{(i)}\ra$.

\medskip

% Then as the proof of Theorem \ref{t-uniqueness},   Theorem \ref{t-uniquecone} is implied by the following theorem.

\begin{thm}\label{t-doublecone}
Let $(X, D, \xi_0)$ be a K-semstable log Fano cone. With the notations in the above paragraph, assume ${\rm Fut}(\mathcal{X}^{(1)},\cD^{(1)}, \xi_0; \eta^{(1)})=0$ and ${\rm Fut}(\mathcal{X}^{(2)},\cD^{(2)},\xi_0; \eta^{(2)})=0$. Then there are weakly special test configurations $(\cX'^{(i)},\cD'^{(i)},\xi_0; \eta'^{(i)})$ of $\left(X^{(i)}_0, \cD^{(i)}_0, \xi_0 \right)$ $(i=1,2)$ with isomorphic central fibers such that
${\rm Fut}(\cX'^{(i)},\cD'^{(i)},\xi_0 ;\eta'^{(i)})=0$ for $i=1,2$.
% being log Fano cone singularities  %Furthermore, the central fiber $(Y,B_Y, \xi_2)$ is a log Fano cone singularity.
\end{thm}

We follow a similar strategy as in Section \ref{s-Qfano}. % the proof naturally consists of two parts: First we will generalize the proof of Theorem \ref{t-doubledeg} to show that there is equivariant degeneration over $\mathbb{C}^2$. 

\begin{proof}  We first claim  that $(X_0^{(1)},D^{(1)}_0, \xi_0)$ is K-semistable. 
If not, then there is a special test configuration $(\cX''^{(1)}_0,\cD''^{(1)}_0,\xi_0;\eta''^{(1)})$ with $$\Fut(\cX''^{(1)}_0,\cD''^{(1)}_0,\xi_0;\eta''^{(1)})<0,$$ which degenerates 
$(X_0^{(1)},D^{(1)}_0, \xi_0)$ to  $(X_0''^{(1)},D''^{(1)}_0, \xi_0)$. 
Then we claim there is a test configuration $(\til{\cX}''^{(1)}_0,\til{\cD}''^{(1)}_0,{\xi}_0; k\eta^{(1)}+ \eta''^{(1)})$ for some $k\gg 0$ degenerating 
$(X,D, \xi_0)$ to  $(X_0''^{(1)},D''^{(1)}_0, \xi_0)$ with the generalized Futaki invariant 
\begin{eqnarray*}
& &\Fut(\til{\cX}''^{(1)}_0,\til{\cD}''^{(1)}_0,{\xi}_0; k\eta^{(1)}+ \eta''^{(1)})\\
& =&k\cdot \Fut(\cX^{(1)}, \cD^{(1)}, \xi_0; \eta^{(1)})+\Fut(\cX''^{(1)}_0,\cD''^{(1)}_0,\xi_0;\eta''^{(1)})\\
& <&0,
\end{eqnarray*}
which is contradictory to our assumption $(X,D,\xi_0)$ is K-semistable. 
Here we used the linearity of the generalized Futaki invariant from Lemma \ref{lem-Futlinear} as in the log Fano cone case.
Denote by $\sigma: \bA^1\to\til{ \cX}_0^{(1)}$ the section of vertices and similarly $\sigma'': \bA^1\to\til{ {\cX}}_0''^{(1)}$. To see the existence of such test a configuration we fix a rational vector $\xi_0'\in {N_\bQ^+}$,  and take the quotient, we get $(\cS^{(1)}_0, \mathcal{B}^{(1)}_0)$ and $(\cS''^{(1)}_0, \mathcal{B}''^{(1)}_0)$ which give special test configurations of  the log Fano pairs obtained as the $\la \xi'_0 \ra$-quotients of  $(X,D)$ and $(X^{(1)}_0, D^{(1)}_0)$.
%$(X\setminus\{x\},D\setminus\{x\})/\la \xi'_0 \ra$ and $(X^{(1)}_0, D^{(1)}_0)\setminus\{x^{(1)}\}/\la \xi'_0 \ra$. 
%\[
%(\cS^{(1)}_0, \mathcal{B}^{(1)}_0):=(\til{\cX}^{(1)}_0,\til{\cD}^{(1)}_0)/\la \xi'_0 \ra\  \mbox{ and  }\  (\cS''^{(1)}_0, \mathcal{B}''^{(1)}_0):=(\til{\cX}''^{(1)}_0,\til{\cD}''^{(1)}_0)/\la \xi'_0 \ra \ .
%\]
Since $[\eta^{(1)},\eta''^{(1)}]=0$, the proof of Lemma \ref{l-semistable} shows that there is a test configuration  %$(\til{\cS}''^{(1)}_0,\overline{B}''^{(1)}_0)$ 
that degenerates the $\la \xi'_0 \ra$-quotient of $(X,D)$ to that of  
$(X_0''^{(1)},D''^{(1)}_0)$. Then we can take the cone back to get $(\til{\cX}''^{(1)}_0,\til{\cD}''^{(1)}_0,{\xi}_0)$. (Also see \cite[Section 4.2]{LX17} for a direct construction.)

\medskip

Applying the diophantine approximation (cf. \cite[Lemma 2.7]{LX17}) of the coordinates of $\xi_0$, we can choose a sequence of  integral vectors $\{\til{\xi}_k\}$ such that $|\til{\xi}_k-k \xi_0|\le A$ for any constant $A>0$ where $k$ is an infinite sequence of increasing positive integers. Consider the Koll\'{a}r component $E_k$ determined by $\til{\xi}_k+\eta^{(1)}$ over $x\in (X,D)$ (it is a Koll\'ar component by Lemma \ref{l-n1}). 
Let $I_k=\fa_m(\ord_{E_k})$ for a sufficiently divisible $m$ depending on $k$. Let $c_k=\lct(I_k; X,D)$ and consider:
$$
f(\frac{1}{k})=\hvol(\ord_{E_k})=\mult(I_k)\cdot c_k^n.
$$
Let $\tilde{T}=\langle \xi_0, \eta^{(1)}\rangle\cong (\bC^*)^{r+1}$ be the torus generated by $\xi_0$ and $\eta^{(1)}$, and $\tilde{N}={\rm Hom}(\bC^*, \tilde{T})$ be the coweight lattice of $\til{T}$. 

Since $(X^{(1)}_0, D^{(1)}_0, \xi_0)$ is K-semistable, 
$$\hvol(\xi):=\hvol_{(X^{(1)}_0, D^{(1)}_0)}(\wt_\xi)$$ 
is a smooth function of $\xi\in \til{N}^{+}_{\bR}$ and obtains the minimum at $\xi_0$ (see Theorem \ref{p-m=k}). By \eqref{eq-Futhvol}, this also implies that for any rational vector $\eta_1\in \tilde{N}_{\bR}$,  
\begin{eqnarray}\label{e-futd}
 \frac{d\ \hvol(\wt_{\xi_0+t\eta_1})}{dt}\Big|_{t=0}=C\cdot \Fut(X^{(1)}_0\times \mathbb{C},D^{(1)}_0\times \mathbb{C},\xi_0; \eta_1)  =0
 \end{eqnarray}

 By Taylor's Remainder Theorem there is a neighborhood $U$ of $\xi_0\in  \tilde{N}_\bR$ and a positive constant $C>0$ (independent of $\xi$) such that, 
for any $\xi\in U$, we have the inequality:
\[
\hvol(\xi_0)\le \hvol(\xi)\le \hvol(\xi_0)+C |\xi-\xi_0|^2.
\]
Note that $f(\frac{1}{k})=\hvol(\frac{1}{k}\til{\xi}_k+\frac{1}{k}\eta^{(1)})$ by the rescaling invariance of the normalized volume. Because
$
\left|\frac{1}{k}\til{\xi}_k+\frac{1}{k}\eta^{(1)}-\xi_0\right|\le C'k^{-1}
$
for $C'>0$ independent of $k$, there exists $K_0\gg 1$ such that for any $k\ge K_0$, $f(\frac{1}{k})=f(0)+O(\frac{1}{k^2})$. 

Then the same argument as in the case of the log Fano varieties using \cite[Corollary 1.4.3]{BCHM10}, shows that we can  find  $\mu^{(2)}\colon \cY^{(2)}_k\to \mathcal{X}^{(2)}$  a  morphism over $\mathbb{C}$ with a divisor $\mathcal{E}^{(2)}_k$ such that $-\mathcal{E}^{(2)}_k$ is ample over $\mathcal{X}^{(2)}$ and $(\cY^{(2)}_k, \cE^{(2)}_k) \times_{\bC}\bC^*=(Y_k, E_k)\times \bC^*$ where the isomorphism is equivariant with respect to the $\bC^*$-action generated by $\eta^{(2)}$.
Moreover, fixed any arbitrarily small $\epsilon$, we can choose $k$ sufficiently large such that the log discrepancy of $E_k$ with respect to $(X , D+(1-\delta)c_k\cdot I_k)$ is less than $\epsilon$, and $(X^{(2)}_0,D_0^{(2)}+(1-\delta)c_k\cdot \bin(I_k))$ is log canonical for a suitable choice of small $\delta$ (see \eqref{e-doublecontrol}). The it follows from the ACC of log canonical thresholds (see \cite{HMX14}) that $(\cY^{(2)}_k, \cE^{(2)}_k+(\mu^{(2)})^{-1}_*\cD^{2}+(\cY^{(2)}_k)_t)$ is log canonical for any $t\in \bC$. 

The relative extended Rees algebra gives a family $(\frak{X}, \frak{D})$ over $\mathbb{C}^2$, such that over $\mathbb{C}\times \{t\}$ (resp. $\{t\}\times \mathbb{C}$) ($t\neq 0$), it gives a family which is isomorphic to $(\cX^{(1)}, \cD^{(1)})$ (resp. $(\cX^{(2)}, \cD^{(2)})$). The family $(\frak{X}, \frak{D})$ admits a $(\mathbb{C}^*)^{r+2}$-action.

%Since  $(\mathcal{Y}^{(2)}_k,\mu_*^{-1}\cD^{(2)}+\mathcal{E}^{(k)}_k)$ is locally stable, restricting over $\mathbb{C}\times \{0\}$ (resp. $\{0\}\times \mathbb{C}$), 
By Lemma \ref{l-tc}, we get {\it weakly special} test configurations 
\[
(\cX'^{(i)},\cD'^{(i)},\xi_0; \eta'^{(i)}) \mbox{ of }(X^{(i)}_0, D^{(i)}_0, \xi_0) \ \ (i=1,2)
\] 
with an isomorphic central fiber $(X'_0,D'_0,\xi_0)$. 

We claim that the generalized Futaki invariants $\Fut(\cX'^{(i)},\cD'^{(i)},\xi_0; \eta'^{(i)})$ are 0. Indeed, by the construction, 
$$\left.(\frak{X}, \frak{D},\xi_0; \eta^{(1)})\right|_{\bC\times\{t\}}\cong (\cX^{(1)},\cD^{(1)}, \xi_0; \eta^{(1)} ).$$ 
It follows from our assumption that 
$$\Fut(\cX^{(1)}, \cD^{(1)}, \xi_0; \eta^{(1)})=\Fut(X^{(1)}_0, D^{(1)}_0, \xi_0; \eta^{(1)})=0.$$ 
By the flatness of the weighted piece and $(\bC^*)^2$ equivariance, we get for any $t$, 
$$\vol_{X^{(1)}_0}(\xi_0+t\eta^{(1)})=\vol_{X'_0}(\xi_0+t\eta^{(1)}),$$
which implies that  $\Fut(X'_0, D'_0, \xi_0; \eta^{(1)})=0$ (see \eqref{e-futd}). Similarly, we have $\Fut(X'_0, D'_0, \xi_0; \eta^{(2)})=0$.

\end{proof}

By the above result, we obtain two {\it weakly special} test configurations $(\cX'^{(i)}, \cD'^{(i)}, \xi_0; \eta'^{(i)})$ with isomorphic central fibres 
 $(X'^{(1)}_0, D'^{(1)}_0, \xi_0)\cong (X'^{(2)}_0, D'^{(2)}_0, \xi_0)$ and zero generalized Futaki invariants. In the next subsection, we are going to show that $(\cX'^{(i)}, \cD'^{(i)}, \xi_0; \eta'^{(i)})$ are indeed special test configurations. 

\subsection{Vanishing Futaki invariants and special degenerations}\label{ss-special}
We will prove Proposition \ref{p-specialcone}, which says to test K-(semi, poly)stability of a log Fano cone, although in our definition we only require to test on all special degenerations,  it is indeed the same to test on all weakly special test configurations. A tool we will use is to write the generalized Futaki invariant of a weakly special configuration as the derivative of the leading coefficient of the index character (see \cite{MSY08, CS12, CS15}).

If there are two $T$-equivariant weakly special test configurations 
$$(\cX^{(i)}={\rm Spec}(\cR^{(i)}),\cD^{(i)},\xi_0;\eta) \mbox{ \ \ of a K-semistable log Fano cone $(X,D,\xi_0)$},$$ 
%we can consider a graph $\cX^{g}$ of the birational map $\cX^{(1)}\dasharrow \cX^{(2)}$, and let $\cD^g$ be the birational transform of $D\times \mathbb{C}^*$. Then we know $(\cX^g,\cD^g,\xi_0;\eta)$ yields a new test configuration. B
with $\Fut(\cX^{(i)},\cD^{(i)},\xi_0;\eta)=0$, by Lemma \ref{l-filtration}, we know $\cX^{(i)}$ is associated to a graded sequence of ideals $\fa^{(i)}_{\bullet}$ which we can assume to be primary (see Remark \ref{r-primary}) as
\begin{eqnarray*}
& &\Fut(\cX^{(i)},\cD^{(i)},\xi_0;m\xi_0'+\eta)\\
 &=&\Fut(\cX^{(i)},\cD^{(i)},\xi_0;m\xi_0')+\Fut(\cX^{(i)},\cD^{(i)},\xi_0;\eta)\\ 
 &=&0,
\end{eqnarray*}
where $\Fut(\cX^{(i)},\cD^{(i)},\xi_0;\xi_0')=0$ follows from the K-semistability of $(X,D)$.
Moreover, since the test configuration is weakly special, by  Lemma \ref{l-tc} there is indeed a birational morphism $\mu^{i}\colon Y^i \to X$ $(i=1,2)$ with a reduced exceptional divisor $E^{i}$ such that $(Y^{i},E^{i}+(\mu^{i})^{-1}_*D)$ is log canonical and $\fa^{(i)}_k=\mu^i_*(-kE^i)$.  Therefore, we can take a normalized graph $\mu^g\colon Y^g\to X$ of $Y^1\dasharrow Y^2$ over $X$ with $p_i\colon Y^g\to Y^i$. 
Then for any pair $(a,b)$ such that $\left(-ap_1^*(E^1)-bp_2^*(E^2)\right)$ is integral, by Definition \ref{d-model}, we can consider the test configuration $\cY_{a,b}$ of $(X,D,\xi_0)$ induced by $\left(-ap_1^*(E^1)-bp_2^*(E^2)\right)$.

We apply the $T$-equivariant index character (see \cite[Section 4]{CS12} for more details) 
for any $\xi\in \tilde{N}^+_{\bR}\subset \tilde{N}_{\mathbb{R}}\cong \mathbb{R}^{r+1}$ where $\tilde{N}^+_{\bR}$ is the Reeb cone of the $\tilde{T}=T\times\bC^*$-action and $t\in \bC$ with the real part $\Re(t)>0$, and define:
\begin{equation}
F(a,b;\xi, t)=\sum_{\alpha\in \tilde{\frak{t}}^+_{\bR}}e^{-t\alpha(\xi)}\dim R^{a,b}_\alpha(v),
\end{equation}
where $R^{a,b}$ is the ring of the special fiber of $\cY_{a,b}$.

Now if we fix a prime integral vector $\xi\in t^+_{\bR}\cap N$ such that 
$$\la \xi \ra\mbox{-quotient of }(\cX^{(i)},\cD^{(i)})=(\cS^{(i)}, B^{(i)},\cL^{(i)})  \ \ (i=1,2)$$ give test configurations of the $\la \xi\ra$-quotient $(S,B)$ of $(X,D)$ with polarizations $\cL^i$. Then the quotient of $\cY_{a,b}$ by $\la\xi\ra$ is given by the normalized graph $\cS_{a,b}$ of $\cS^{(1)}\dasharrow \cS^{(2)}$ with morphisms $\phi_i\colon \cS_{a,b} \to \cS^{(i)}$ and the polarization is given by $a\phi_1^*\cL^{(1)}+b\phi_2^*\cL^{(2)}$.

\medskip 
The following  statement essentially follows from \cite[Theorem 4.10]{CS12}.
\begin{prop}\label{p-polynomial}
For a fixed $\xi\in  \tilde{N}^+_{\bR}$ the index character $F(a,b;\xi, t)$ has a meromorphic extension to $\bC$ with poles along the imaginary axis. Near $t=0$ it has a Laurent series expansion:
\begin{equation}
F(a,b;\xi, t)=\frac{a_0(a,b;\xi)n!}{t^{n+1}}+\frac{a_1(a,b;\xi)(n-1)!}{t^n}+\cdots,
\end{equation}
where $a_0(a,b;\xi)$ is a polynomial of $(a,b)$ whose coefficients depends smoothly on $\xi\in  \tilde{N}^+_{\bR}$.
\end{prop}
\begin{proof}
 It follows from \cite[Proposition 4.3]{CS12} that when $\xi$ is rational, then $a_0$ coincides with the leading term of the total weight on the test configuration $\cS_{a,b}$ constructed from the quotient log Fano pair. Since it can be represented by an intersection formula, in particular, it is a polynomial of $a$ and $b$ by \cite{Wan12, Oda13}.

 Denote by $s=r+1$. By the proof of \cite[Theorem 4.10]{CS12}, we know 
$$F(a,b; \xi, t)=\frac{e^{-t(\xi_1\alpha_1+\cdots \xi_s\alpha_s)}\cdot {HN_{a,b}(e^{-t\xi_1},...e^{-t\xi_s})}}{\Pi^N_{j=1}(1-e^{-t(\xi_1w_{1j}+\cdots+ \xi_sw_{sj})})},$$
where $\xi=(\xi_1,...,\xi_s)\in {\tilde{N}^+_{\bR}}$, $(\alpha_1,...,\alpha_s)\in \mathbb{Z}^s$ and $w_{ij}$ ($1\le i \le s$, $1\le j\le N$) are real numbers. 
The leading term of the Laurent expansion is the same as the leading term of
$$\frac{HN_{a,b}(1,...,1)}{\Pi^N_{j=1}(1-e^{-t(\xi_1w_{1j}+\cdots+ \xi_sw_{sj})})}\ \ .$$ Since $a,b$ only appear in the part $HN_{ab}(1,...,1)$  which does not depend on $\xi$, and from the case that $\xi$ is rational, we know that $HN_{a,b}(1,...,1)$ is a polynomial of $(a,b)$, which implies $a_0$ is a polynomial of $(a,b)$.
%Similarly, we can also prove $a_1$ is a polynomial of $(a,b)$.
%We know that for $a,b\gg 0$ $H^0(X,(aL_1+bL_2))$ is a polynomial of $(a,b)$, and so is for the total weight of any $\mathbb{C}^*$-action $H^0(X,(aL_1+bL_2))$
\end{proof}

With all these preparations,  we can prove Proposition \ref{p-specialcone} which is a generalization of \cite[Theorem 4]{LX14} from the quasi-regular case to the general case of  an arbitrary log Fano cone singularity. Although we expect the full results of special degeneration in \cite{LX14} can be extended, here we only need the last step of the argument. %as our test configurations  $(\cX'^{(i)},\cD'^{(i)},\xi_0; \eta'^{(i)})\; (i=1,2)$ are known to be weakly special. 

\begin{prop}\label{p-specialcone}
Let $(\mathcal{X}, \cD,\xi_0; \eta)$ be a weakly special test configuration of a log Fano cone singularity $(X, D, \xi_0)$. Then we can find a special test configuration $(\cX',\cD',\xi_0;\eta')$ and a positive integer $m$ such that
$${\rm Fut}(\mathcal{X}',\cD',\xi_0; \eta')\le m\cdot {\rm Fut}(\cX,\cD,\xi_0;\eta),$$
and the strict inequality holds if  $(\cX,\cD,\xi_0; \eta)$ is not a special test configuration. 
\end{prop}
\begin{proof}
%By Lemma \ref{l-normalize}, with out of loss of generality we can assume $(\cX, \cD, \xi_0; \eta)$ is normalized: $A(\xi_0)=n$ and $A(\eta)=0$.
By Lemma \ref{l-tc}, the weakly special test configuration is induced by a $T$-equivariant morphism $\mu\colon Y\to X$, such that the reduced exceptional divisor $E$ is anti-ample over $X$ and $(Y,E+\mu_*^{-1}D)$ is log canonical.
Suppose $(\cX,\cD,\xi_0;\eta)$ is not special, then $(Y,E+\mu_*^{-1}D)$  is not plt. Therefore, by \cite[Proposition 2.10]{LX16}, we can find a $T$-equivariant  Koll\'ar component $S$ over $x\in (X,D)$ such that its log discrepancy with respect to $(Y,E+\mu_*^{-1}D)$ is 0.  Denote by $\mu'\colon Y'\to X$ the plt blow extracting precisely $S$. So by Lemma \ref{l-tc} again, it gives a special test configuration $(\mathcal{X}',\cD',\xi_0; \eta')$ and the base change factor $m$ (which we omit from now on) corresponds to a multiple such that the coefficient of $S$ in the pull back of $mE$ is integral. 

Let $Y^g$ be the normalized graph $Y\dasharrow Y'$ and $p\colon Y^g \to Y$, $p'\colon Y^g\to Y'$ the natural morphisms. Then for any pair of positive integers $(a,b)$, the divisor $bp^*E+ap'^*S$ are anti-ample, and therefore induces a test configuration $\cX_{a,b}$ by Lemma \ref{l-tc}. We take $a_0(a,b,\xi)$ as in Proposition \ref{p-polynomial}.

Now we claim that
$$D_{T_{\xi_0}(\eta)}a_0(1,0,\xi_0)=\Fut(\cX, \cD,\xi_0; \eta)> \Fut(\cX', \cD', \xi_0; \eta')=D_{T_{\xi_0}(\eta)}a_0(0,1,\xi_0).$$ 

%Let $\xi_0'$ be a  small rational perturbation of $\xi_0$. Then $(\cS,\mathcal{B})=(\cX,\cD)/\la\xi_0' \ra$ and $(\cS',\mathcal{B}')=(\cX',\cD')/\la\xi_0' \ra$ give two test configurations of $(X,D)/\la \xi_0' \ra$. Then take a common model $\hat{\cS}$ of $\cS$ and $\cS'$, and construct the cone given by

%We note that $\cX$ is an open set of $\overline{\cX}$ which is projective over $X\times \mathbb{C}$, where $\overline{\cX}$ is obtained by contracting the closure of $E\times \mathbb{C}^*$ in the total space of the deformation to the normal cone for $E\subset Y$ as a $\mathbb{Q}$-Cartier Weil divisor (see \cite[Section 2.4]{LX16} for the details). And similarly we have $\overline{\cX}'$. Now take an $T$-equivariant common model $\hat{Y}\to X$ which is isomorphic outside $X\setminus\{x\}$ and dominates $Y$ and the model $Y_S$ extracting $S$. dominating model $ $  $p: \hat{\cX}\rightarrow \overline{\cX}$ and $q: \hat{\cX}\rightarrow \overline{\cX}'$ and 
To see this we write: 
$$p^*(K_{Y}+E+\mu_*^{-1}D)=p'^*(K_{Y'}+S+(\mu'_*)^{-1}D)+G,$$
and since the log discrepancy $A_{Y,E+\mu_*^{-1}D}(S)=0$, the negativity lemma {(see \cite[Lemma 3.39]{KM98})} implies that $G\ge 0$.
%where  $\overline{\cD}$ and  $\overline{\cD}'$ are birational transform of $\cD$ in $\overline{\cX}$ and $\overline{\cX}'$. The discrepancy of $X_0'$ with respect to $(\overline{\cX}, \overline{\cD})$ is equal to $A_{Y,E+\mu_*^{-1}D}(S)+1$ which is $0$ by our choice of $S$. Since since $K_{\overline{\cX}}+\overline{\cD}$ is anti-ample over $X\times \bC$,   we can apply the negativity lemma to conclude that  $G\ge 0$. 

For any irreducible component  $E_i$ in ${\rm Supp}(G)$, denote by $c_i$ its coefficient in $G$. In particular, from our assumption that $\cX$ is not a special test configuration,  for some component $E_0$ contained in ${\rm Supp}(E)$, its coefficient $c_0$ is positive. Let $F_i$ be divisor on $X_0$ given by the orbifold cone $C(E_i,-E|_{E_i})$.

We take the previous construction for the two test configurations  $\cX$ and $\cX'$. By Proposition \ref{p-polynomial}, for a fixed $\xi_0$, if we define
$$f(t;\xi_0)=D_{T_{\xi_0}(\eta)}a_0(1-t,t;\xi_0),$$ 
%$$f(t;\xi)=\frac{(b_0a_1-a_0b_1+\frac{1}{2}\sum (a_0b_{i,0}-a_1b_0))(1-t,t,\xi_0)}{a^2_0(\xi_0)}.$$ 
then the difference of the generalized Futaki invariant is of the form 
\begin{eqnarray*}
\Fut(\cX', \cD',\xi_0; \eta')- \Fut(\cX, \cD, \xi_0; \eta)= \int^1_{0}\frac{d}{dt} f(t;\xi _0) \ dt.
\end{eqnarray*}

The integrand is  smooth in $[0,1]$, and the proof of \cite[Proposition 5]{LX14} shows that it is non-positive when $\xi_0$ is rational. Thus it is non-positive. We claim its value at 0 is
\begin{eqnarray}\label{e-compare}
\left.\frac{d}{dt} f(t;\xi_0)\right|_{t=0}= - \frac{1}{ \vol_{X}(\wt_{\xi_0})} \sum_i c_i\cdot \vol_{F_i}(\wt_{\xi_0})<0.
\end{eqnarray}
%where $\hat{\xi}=\frac{\xi}{A(\xi)}$ and $T_{\eta}(\xi):=\frac{A(\xi)\eta-A(\eta)\xi}{A^2(\xi)}$ (see \cite[Equation (11)]{LX17}). 
In fact to see \eqref{e-compare}, when $\xi_0$ is rational, we can compute on the quotient log Fano pair, and this is given in \cite[Page 217]{LX14}. Since both sides are smooth functions on $\xi_0$, we know that they must be equal to each other. 
%Then if we normalize $\eta$, it follows from \cite{CS12,CS15} that $$\Fut(\cX, \cD, \xi_0; \eta)=C \cdot D_{\eta}a_0(1,0,\xi)$$ for some positive constant $C$ which only depends on $(X,D,\xi_0)$ and $\eta$ is normalized. Thus t
%By looking at the self intersection number of the generalized Futaki invariant on the quotient log Fano pair, the calculation in the proof  \cite[Proposition 5]{LX14} shows that for any $\xi\in \ft^{+}_{\mathbb{Q}}$,  $$\frac{d}{dt} f(1-t,t;\xi ) \ge 0.$$ Thus $\frac{d}{dt} f(1-t,t;\xi ) )\ge 0$ for any $\xi_0\in \ft^+_{\bR}$.
%Moreover, at 0 we have $$\frac{d}{dt} f(1-t,t;\xi ) = - \frac{1}{\vol_{X_0}(\wt_{\xi_0})} \sum_i a_i\cdot \vol_{E_i}(\wt_{\xi_0})<0.$$
%where $\hat{\xi}=\frac{\xi}{A(\xi)}$ and $T_{\eta}(\xi):=\frac{A(\xi)\eta-A(\eta)\xi}{A^2(\xi)}$ (see \cite[Equation (11)]{LX17}).  To see the above equality, note that it holds for $\xi\in \ft^+_{\mathbb{Q}}$ since again we can compute the left hand side on the quotient log Fano pair and apply the proof of \cite[Proposition 5]{LX14} to know it is equal to the right hand side. Since both sides are smooth function on $\xi$, we know that they must equal each other for all $\xi\in \ft^+_{\bR}$. 
\end{proof}

An immediate consequence is the following.
\begin{cor}\label{c-ws=s} 
For a K-semistable log Fano cone singularity $(X,D,\xi)$, if it has a weakly special test configuration $(\cX,\cD,\xi_0;\eta)$ with the generalized Futaki invariant being 0, then it  is a special test configuration, i.e., the central fiber is klt. 
\end{cor}

\subsection{Completion of the proofs of main theorems for log Fano cones}\label{ss-completionofproof}

\begin{proof}[Proof of Theorem \ref{t-uniquecone}]
The proof follows the same structure as the proof of Theorem \ref{t-uniqueness}, so we will only outline the steps. 

We first prove the existence of K-polystable degenerations.
 Let $(X^p,D^p,\xi_0)$ be a log Fano cone which admits a maximal dimensional torus $T^p$-action among all K-semistable special degeneration of $(X, D, \xi_0)$. By definition, $T^p$ contains $T$. If $(X^p,D^p,\xi_0)$ is not K-polystable, it admits a $T$-equivariant degeneration to a non-isomorphic log Fano cone $(X^{(1)},D^{(1)},\xi_0)$ under a special test configuration with generalized Futaki invariant 0. Similar to the proof of Lemma \ref{l-rank}, this  indeed implies that there is a $T^p$-equivariant degeneration of $(X^p,D^p,\xi_0)$ to $(X^{(2)},D^{(2)},\xi_0)$ under a special test configuration of generalized Futaki invariant 0. Moreover, $(X^{(2)},D^{(2)},\xi_0)$ admits an action by a torus whose dimension is equal to $\dim(T^p)+1$.  By the proof of Theorem \ref{t-doublecone}, $(X^{(2)},D^{(2)},\xi_0)$ is also a K-semistable special degeneration of $(X,D,\xi_0)$, which is a contradiction.

%We first prove the existence of K-polystable degenerations. In the proof of Theorem \ref{t-doublecone} we have shown that for any special test configuration $(\cX,\cD,\xi_0;\eta)$ of $(X^{(0)},D^{(0)},\xi_0):=(X,D,\xi_0)$ with $\Fut(\cX,\cD,\xi_0;\eta)=0$, the special fiber $(X^{(1)},D^{(1)},\xi_0)$ is K-semistable.  Furthermore, any special degeneration of $(X^{(1)},D^{(1)},\xi_0)$ can be indeed written as a special degeneration of $(X^{(0)},D^{(0)},\xi_0)$. Similar to the proof of Theorem \ref{t-uniqueness}, if the K-semistable degeneration $(X^{(1)},D^{(1)},\xi_0)$ is not isomorphic to $(\cX,\cD,\xi_0;\eta)$, then $(X^{(1)},D^{(1)},\xi_0)$ admits an effective action by a torus $\tilde{T}$ one dimensional larger than $\dim (T)$. Therefore such step has to terminate, namely when we have the central fiber being K-polystable. 

To see the uniqueness, by Corollary \ref{c-ws=s}, we can replace the word ``weakly special" by ``special" in the statement of Theorem \ref{t-doublecone}. Recall that by definition special degenerations of K-polystable log Fano cone with zero generalized Futaki invariants must be product. So the uniqueness of K-polystable degeneration follows.
\end{proof}

\begin{proof}[Proof of Theorem \ref{t-DSconj}] %We have two slightly different proofs. The first one depends on Proposition \ref{p-specialcone}. \medskip

It is shown in \cite{DS17} that there is a special test configuration $(\cW, \xi_0; \eta)$ of the intermediate cone $(W, \xi_0)$ {with central fibre} $(C, \xi_0)$. Because $C$ admits a Ricci-flat K\"ahler cone metric, we know $C$ is K-polystable  (see \cite[Theorem 1.1]{CS15} or Corollary \ref{cor-RF2K}). In particular, ${\rm Fut}(\cW, \xi_0; \eta)=0$. Moreover by \cite[Theorem 1.4]{LX17}, we know that $W$ is K-semistable and is uniquely determined by the algebraic germ $(M_\infty, o)$.

Assume $W$ specially degenerates to another K-polystable Fano cone $C'$ by a special test configuration $(\mathcal{W}',\xi_0; \eta')$ with $\Fut(\mathcal{W}',\xi_0; \eta')=0$. Then Theorem \ref{t-doublecone}  implies that $C$ and $C'$ degenerates to a Fano cone $C''$ by special test configurations with generalized Futaki invariants 0. This implies $C\cong C''\cong C'$ by the polystability of $C$ and $C'$. 

\end{proof}

\appendix
\section{Ding-polystability of Ricci-flat K\"ahler cones}\label{s-analytic}

In the proof of Theorem \ref{t-DSconj} above, we rely on  the result  proved in \cite{CS15} which says that that for a Fano cone singularity with a Ricci-flat K\"ahler cone metric, the generalized Futaki invariant $\Fut(\mathcal{X}, \xi_0;\eta)>0$ for any non-product special test configuration. However, as we have seen, in our argument (see e.g. the proof of Theorem \ref{t-doublecone}), more general test configuration will show up. Therefore in this appendix, we want to discuss the proof of a more general statement, namely  for any non-product $\mathbb{Q}$-Gorenstein test configuration, the corresponding Ding invariant is positive (see Theorem \ref{thm-RF2Ding}). This can be used to slightly modify the proof of Theorem \ref{t-DSconj} (see Remark \ref{r-2nd}). We point out that our proof of Theorem \ref{thm-RF2Ding} follows the general strategy in \cite{Ber15} and is slightly different from \cite{CS15}. {For simplicity of notations, in this appendix we will restrict to the case that the boundary divisor $D=\emptyset$ which suffices for proving the main application of our results in Theorem \ref{t-DSconj}.}

\begin{defn}[Ding-stability]\label{d-ding}
We say that $(X, \xi_0)$ is Ding-semistable, if for any $\bQ$-Gorenstein test configuration $(\cX,  \xi_0; \eta)$ of $(X, \xi_0)$ {with central fibre} $(X_0, \xi_0)$, its Berman-Ding invariant, denoted by $D^\NA(\cX,  \xi_0; \eta)$ is nonnegative, where 
\begin{eqnarray}\label{eq-DNA}
D^\NA(\cX, \xi_0; \eta):=\frac{D_{T_{\xi_0}(\eta)}\vol_{X_0}(\xi_0)}{\vol(\xi_0)}+{\lct(\cX; \cX_0)-1}.
\end{eqnarray}
We say that $(X,  \xi_0)$ is Ding-polystable, if it is Ding-semistable, and any $\bQ$-Gorenstein test configuration $(\cX, \xi_0; \eta)$ with $D^\NA(\cX, \xi_0; \eta)=0$ is a product test configuration. 
\end{defn}

We will show in \eqref{eq-Dslope} that the $D^\NA$-invariant in \eqref{eq-DNA} is equal to the slope of a Ding-type functional (see Definition \ref{def-Dvphi}) along a subgeodesic ray associated to the test configuration. Following the notations of \cite{BHJ15} in the log Fano case, we will use $D^\NA$ to denote such slope functional, since it can be viewed as a functional on the space of non-Archimedean metrics (associated to test configurations).

{\begin{rem}\label{rem-Ding}
We immediately see that $D^\NA(\cX, \xi_0; \eta)=\Fut(\cX, \xi_0; \eta)$ if and only if the test configuration is weakly special, and Ding-semistability (resp. Ding-polystability) implies K-semistability (resp. K-polystability). It has been proved that in the log Fano pair case, they are equivalent \cite{BHJ15, Fuj16}.
Following \cite{Ber15}, it will become clear that the notions of Ding-stability fit better with our calculation.
\end{rem} }

\begin{thm}\label{thm-RF2Ding}
Assume $(X, \xi_0)$ admits a Ricci-flat K\"{a}ler cone metric. Then $(X, \xi_0)$ is Ding-polystable among $\bQ$-Gorenstein test configurations.

\end{thm}

\begin{cor}\label{cor-RF2K}
Assume $(X, \xi_0)$ admits a Ricci-flat K\"{a}hler cone metric. Then $(X, \xi_0)$ is K-polystable among all weakly special test configurations.
\end{cor}

\begin{rem}\label{r-2nd}
Corollary \ref{cor-RF2K} could yield an alternative argument in one step of the proof of Theorem \ref{t-DSconj}.
%The other one uses analytic result from Corollary \ref{cor-RF2K} proved in the next section. 
More precisely, with notations in the proof of Theorem \ref{t-DSconj}, let $C$ and $C'$ be two K-polystable degenerations of $W$. Then the degenerations of $C$ and $C'$ to $C''$ 
obtained via Theorem \ref{t-doublecone} 
are weakly special with zero Futaki invariant. We can skip Proposition \ref{p-specialcone}  but replace \cite[Theorem 1.1]{CS15} by the stronger statement Corollary \ref{cor-RF2K}, which directly implies  there is no non-product weakly special test configurations of $C$ and $C'$ with zero Futaki invariant.  
Then we conclude immediately that $C\cong C''\cong C'$. 
\end{rem}

Let $(X, \xi_0)$ be a Fano cone singularity with the vertex point $o$. Recall that this implies that $X$ is a normal affine variety with at worst klt singularities. Moreover there is a good $T$ action where $T\cong (\bC^*)^r$ and $\xi_0\in {N_\bR^+}$. On $X$ there exists a $T$-equivariant nowhere-vanishing holomorphic $m$-pluricanonical form $s\in |-mK_X|$. Such holomorphic form can be solved uniquely up to a constant as in \cite[2.7]{MSY08}. In the following, we will use the following volume form on $X$ associated to $s$:
\begin{equation}
dV_X=\left(\sqrt{-1}^{mn^2}s\wedge \bar{s}\right)^{1/m}.
\end{equation}

Assume that $(X, \xi_0)$ is equivariantly embedded into $(\bC^N, \xi_0)$ with $\xi_0=\sum_i a_i z_i\frac{\partial}{\partial z_i}$ with $a_i\in \bR_{>0}$.  Fix a reference smooth K\"{a}hler cone metric on $\bC^N$ whose associated Reeb vector field $r\partial_r-{\sqrt{-1}}J(r\partial_r)=2\xi_0$. By its rescaling property such a radius function is $C^0$-comparable to $\sum_{i=1}^N |z_i|^{2/a_i}$.
%\[
%r^2=\sum_{i=1}^N |z_i|^{2/a_i}.
%\]
%The corresponding K\"{a}hler cone metric on $\bC^N$ is equal to:
%\begin{eqnarray*}
%\omega_{\bC^N}%&=&\sddb r^2=\sqrt{-1}\left(\sum_i a_i^{-1} |z_i|^{2(a_i^{-1}-1)} z_i d\bar{z}_i\right)\\
%&=& \sqrt{-1}\sum_i \left(a_i^{-1} (a_i^{-1}-1)|z_i|^{2(a_i^{-1}-2)}\bar{z}_i dz_i \wedge z_i d\bar{z}_i+\sum_i a_i^{-1} |z_i|^{2(a_i^{-1}-1)} dz_i\wedge d\bar{z}_i\right)\\
%&=& \sqrt{-1}\sum_{i=1}^N a_i^{-2}|z_i|^{2(a_i^{-1}-1)}dz_i\wedge d\bar{z}_i.
%\end{eqnarray*}
The restriction $\omega_X:=\omega_{\bC^N}|_{X}$ is a K\"{a}hler cone metric on $X$.
Moreover $2{\rm Im}(\xi_0)=J(r\partial_r)$ is the Reeb vector field of $\omega_{\bC^N}$ and $\omega_X$. 
%The following identities can be verified directly:
%\begin{equation}\label{eq-xi0r}
%\xi_0(r^2)=\bar{\xi}_0(r^2)=r^2, \quad 2{\rm Re}(\xi_0)=r\partial_r;
%\end{equation}
%\begin{equation}\label{eq-xi0u}
%\frac{n \sqrt{-1} 2(\partial v) \wedge (\bar{\partial} r) \wedge (\sddb r^2)^{n-1}}{(\sddb r^2)^n}=\xi_0(v).
%\end{equation}
Since $T$ acts on $X$, $T$ also acts on the set of functions on $X$ by $\tau\circ f(x)=f(\tau^{-1}x)$ for any $\tau\in T$ and $x\in X$. For convenience, we denote $X^\circ=X\setminus \{o\}$ where $o$ is the vertex of $X$ and define: 
\begin{defn}
Denote by $PSH(X, \xi_0)$ the set of bounded real functions $\vphi$ on $X^\circ$ that satisfies:
\begin{enumerate}
\item[(1)] ${\vphi\circ \tau}=\vphi$ for any $\tau\in \langle \xi_0\rangle$;
\item[(2)] $r^2_\vphi:=r^2e^\vphi$ is a proper plurisubharmonic function on $X$. 
\end{enumerate}
\end{defn}
We can think of functions in $PSH(X, \xi_0)$ as transversal K\"{a}hler potentials as in \cite{DS17}. More precisely, because $\partial_r$ generates a $\bR_{+}$-action ($\bR_+=\{a\in \bR; a>0\}$) on $X^\circ$ without fixed points, if the link of $X$ is defined as $Y:=\{r=1\}\cap X$, then $Y=X^\circ/\bR_{+}$ and $X^\circ\cong Y\times\bR_{+}$.
{We set:}
\begin{equation}
\chi=\frac{\sqrt{-1}}{2}(\bar{\partial}-\partial)\log r^2=-\frac{1}{2}Jd\log r^2,
\end{equation}
and define:
\begin{defn}\label{defn-PSHY}
Denote by $PSH(Y, \xi_0)$ the set of bounded real function $\vphi$ on $Y$ that satisfies:
\begin{enumerate}
\item[(1)] ${\vphi\circ \tau}=0$ for $\tau\in \exp(\bR \cdot \im(\xi_0))$.
\item[(2)] $\vphi$ is upper semicontinuous on $Y$ and $\left.(d\chi+\sddb\vphi)\right|_Y\ge 0$, where the positivity is in the sense of currents.
\end{enumerate}
\end{defn}
Here we identify the function on $Y$ with its pull back to $X^\circ\cong Y\times \bR_{+}$ via the projection to the first factor.
There is an isomorphism $PSH(X, \xi_0)\cong PSH(Y, \xi_0)$ by sending $\vphi\mapsto \vphi|_Y$. We will use these two equivalent descriptions in the following discussion.

\begin{defn}
We say that $r^2_\vphi:=r^2e^\vphi$ where $\vphi\in PSH(X, \xi_0)$ is a radius function of a Ricci-flat K\"{a}hler cone metric on $(X, \xi_0)$ if $\vphi$ is smooth on $X^\reg$ and there exists 
a constant $C>0$ such that 
\begin{equation}\label{eq-RFKC}
(\sddb r^2_\vphi)^n=C \cdot dV_X.
\end{equation}
\end{defn}
If we take $\cL_{r\partial_r}$ on both sides, we get:
$
\cL_{r\partial_r} dV_X=2n dV_X,
$
which is also equivalent to $\cL_{\xi_0}s=mn s$. If we write 
\begin{equation}\label{eq-dVtrans}
dV_X=2 r^{2n-1}dr\wedge \Omega_Y,\quad \text{ or equivalently }\quad \Omega_Y:=2^{-1} r^{1-2n}i_{\partial_r}dV_X,
\end{equation}
then $\cL_{\partial_r}\Omega_Y=0$. 
On the other hand, a direct computation shows that:
\begin{equation}\label{eq-transverse}
\sddb r^2_\vphi=r_\vphi^2 (d\chi+\sddb \vphi)+d r^2_\vphi\wedge \left(\chi-\frac{1}{2}Jd\vphi\right),
\end{equation}
Then it is easy to verify that the equation \eqref{eq-RFKC} is equivalent to:
\begin{equation}
(d\chi+\sddb \vphi)^{n-1}\wedge \chi=\frac{C}{n}\cdot e^{-n\vphi} \Omega_Y.
\end{equation}

%Since we assume $\vphi$ is smooth on $X^\reg$, we can re-write the above equation as an equation on the link $Y=\{r=1\}$:

The equation \eqref{eq-RFKC} is the Euler-Lagrange equation for the following Ding-type functional:

\begin{defn}[see \cite{CS15, LX17}]\label{def-Dvphi}
For any function $\vphi\in PSH(X, \xi_0)$, define:
\begin{equation}
D(\vphi)=E(\vphi)-\log\left(\int_X e^{-r^2_\vphi} dV\right)=:E(\vphi)+G(\vphi)
\end{equation}
where $E(\vphi)$ is defined by its variations:
\[
\delta E(\vphi)\cdot \delta\vphi=-\frac{1}{(n-1)!(2\pi)^n\vol_X(\xi_0)}\int_X(\delta\vphi)e^{-r^2_\vphi}(\sddb r^2_\vphi)^n.
\]
\end{defn}
Using the identity \eqref{eq-transverse}, one can verify that:
\begin{equation}
\delta E(\vphi)\cdot \delta\vphi=-\frac{n}{(2\pi)^n\vol(\xi_0)}\int_{Y}(\delta\vphi) (d\chi+\sddb \vphi)^{n-1}\wedge \chi.
\end{equation}
As in the standard K\"{a}hler case, a consequence of this description is the following explicit expression of $E(\vphi)$ (see \cite{DS17}):
\begin{equation}
E(\vphi)=-\frac{1}{(2\pi)^n\vol(\xi_0)}\sum_{i=0}^{n-1}\int_Y \vphi (d\chi+\sddb \vphi)^i\wedge (d\chi)^{n-1-i}\wedge \chi.
\end{equation}
In the similar vein, using \eqref{eq-dVtrans} we have the identity:
\begin{equation}
G(\vphi)=-\log\left(\int_Y e^{-n \vphi}\Omega_Y\right)-\log (n-1)!.
\end{equation}
We will study the asymptotic of $E(\vphi_t)$. In the following we will denote $\bD:=\{z\in \bZ; |z|\le 1\}$, $\bD^*=\bD\setminus \{0\}$ and $S^1=\{z\in \bD; |z|=1\}$. We will always identify the functions on $X$ with functions on $X\times\bD$ or $X\times\bD^*$ by pulling back via the projection to the first factor.

%For simplicity, we introduce the logarithmic coordinate $s=-\log t^2$ on the Riemann surface $\bS:=[0, +\infty)\times S^1$.
\begin{prop}[{see \cite[Lemma 5.10]{LX17}}]\label{prop-hessE}
Let $\vphi(x, t)=\vphi(x, |t|): X\times \bD^*\rightarrow \bR$ be a upper semicontinuous function such that $\vphi_t:=\vphi(\cdot, |t|)\in PSH(X, \xi_0)$ for each $t\in\bD^*$. Assume
$\sddb (r^2e^\vphi)\ge 0$ over $X\times \bD^*$ in the sense of currents. Then the following identity holds:
\begin{eqnarray*}
\sqrt{-1}\frac{\partial^2}{\partial t \partial \bar{t}} E(\vphi_t)dt\wedge d\bar{t}&=&-\frac{1}{(n+1)!(2\pi)^n\vol(\xi_0)}\int_{X\times \bD^*/\bD^*}(\sddb (r^2 e^\vphi))^{n+1} e^{-r^2_\vphi} 
\\&=&-\frac{1}{(2\pi)^n\vol(\xi_0)} \int_{Y\times \bD^*/\bD^*}(d\chi+\sddb \vphi)^n\wedge\chi.
\end{eqnarray*}
In particular, $E(\vphi_t)$ is concave in $-\log|t|^2$.
\end{prop}
\begin{proof}
The proof of the first identity is the same as the proof as in \cite[Lemma 5.10]{LX17}. The second identity follows from the first one and using the following identity on $X\times \bD^*$ to calculate:
\[
\sddb r^2_\vphi=r^2_\vphi(d\chi+\sddb \vphi)+d r^2_\vphi\wedge (\chi-\frac{1}{2}Jd\vphi).
\]
\end{proof}

Now assume that $(\cX, \xi_0;\eta)$ is a $\bQ$-Gorenstein test configuration of $X$. Because $\eta$ commutes with $\xi_0$ and generates a $\bC^*$-action, we can assume
that $\cX$ is embedded into $\bC^N\times\bC$ and the embedding is equivariant with respect to the $T\times \bC^*$-action generated by $\{\xi_0, \eta\}$. If we write
$\eta=\sum_i b_i z_i\frac{\partial}{\partial z_i}$ with $b_i\in \bZ$ and let $\sigma(t): \bC^*\rightarrow GL(N, \bC)$ be the one-parameter subgroup generated by the vector field $\eta$. 
Then $\sigma(t)(z_i)=t^{b_i} z_i$ and we let
\[
r(t)^2:=\sigma(t)^*(r^2)=: r^2 e^{\til{\vphi}(t)}.
\]
%where the function $\til{\vphi}(t)\in PSH(X, \xi_0)$ is given by:
%\begin{equation}\label{eq-tilvphi}
%\til{\vphi}(t)=\log \left(\sum_i |t|^{2b_i/a_i}|z_i|^{2/a_i}\right)- \log \left(\sum_i |z_i|^{2/a_i}\right).
%\end{equation}

The asymptotic of $E(\til{\vphi}_t)$ can be easily calculated:
\begin{prop}[{see \cite[Proposition 5.13]{LX17}}]\label{prop-Easymp}
We have the following identity:
\begin{equation}
\lim_{t\rightarrow 0}\frac{E(\til{\vphi}_t)}{-\log|t|^2}=\frac{D_{\eta}\vol(\xi_0)}{\vol(\xi_0)}.
\end{equation}

\end{prop}
\begin{proof}
We refer to \cite{LX17} for details. Here we just sketch the key ingredients. Let $\xi_\epsilon=\xi+\epsilon \eta=\sum_i (a_i+\epsilon b_i)z_i\frac{\partial}{\partial z_i}$ and $r_\epsilon$ be a radius function for $\xi_\epsilon$. Then we have:
\begin{equation}
\vol(\xi_\epsilon)=\frac{1}{n!(2\pi)^n}\int_{X_0} e^{-r_\epsilon^2}(\sddb r_\epsilon^2)^n.
\end{equation}
Taking derivative with respect to $\epsilon$ in the above volume formula, we can derive:
\begin{eqnarray*}
D_{\eta}\vol(\xi_0)&=&\frac{1}{(2\pi)^n(n-1)!}\int_{X_0}\theta e^{-r^2}(\sddb r^2)^n,
\end{eqnarray*}
where we have denoted $\theta:=\eta(\log r^2)$. 
%Note that, by using the expression for $\til{\vphi}_t$ in \eqref{eq-tilvphi}, we have the identity:
%\[
%-\frac{d}{d \log |t|^2}\til{\vphi}=\frac{\sum \frac{b_i}{a_i}|t|^{2b_i/a_i}|z_i|^{2/a_i}}{\sum_i |t|^{2b_i/a_i}|z_i|^{2/a_i}}=\sigma(t)^*\theta.
%\]
We can then calculate (see \cite[Appendix C]{MSY08} or \cite[Lemma 5.11]{LX17}):
\begin{eqnarray*}
\frac{d}{d(-\log|t|^2)}E(\til{\vphi}_t)&=&\frac{1}{(n-1)!(2\pi)^n\vol(\xi_0)}\int_X \dot{\til{\vphi}}e^{-r(t)^2}(\sddb r(t)^2)^n\\
&=&\frac{1}{(n-1)!(2\pi)^n\vol(\xi_0)}\int_X \sigma(t)^*(\theta)e^{-\sigma^*r^2}\sigma^*(\sddb r^2)^n\\
&=&\frac{1}{(n-1)!(2\pi)^n\vol(\xi_0)}\int_{X_t}\theta e^{-r^2}(\sddb r^2)^n.
\end{eqnarray*}
As explained in \cite[Proof of Proposition 5.12]{LX17}, the last expression converges as $t\rightarrow 0$ to $D_{\eta}\vol(\xi_0)/\vol(\xi_0)$. 

By Proposition \ref{prop-hessE} $E(\til{\vphi}_t)$ is concave in $-\log|t|^2$. So the statement follows from the above discussion and the following identity for concave functions:
\[
\lim_{t\rightarrow 0} \frac{d}{d(-\log|t|^2)}E(\til{\vphi}_t)=\lim_{t\rightarrow 0} \frac{E(\til{\vphi}_t)}{-\log|t|^2}
\]
\end{proof}

We need the following basic result from \cite{DS17} which generalizes Berndtsson's result to the K\"{a}hler cone setting.

\begin{thm}[{\cite{DS17}, see also \cite{BBEGZ,Ber15}}]\label{thm-affine}
Let $\vphi(x,t)=\vphi(x, |t|): X\times \bD^*\rightarrow \bR$ be an upper semicontinuous function such that $\vphi_t:=\vphi(\cdot, t)\in PSH(X, \xi_0)$ for each $t\in \bD^*$. Assume
$\sddb (r^2e^\vphi)\ge 0$ over $X\times \bD^*$ in the sense of currents. Then $G(\vphi_t)$ is convex in $-\log|t|^2$.  Moreover, if $G(\vphi_t)$ is affine in $-\log|t|^2$, then 
there exists a holomorphic vector field $\eta_0$ on $X$ commuting with $\xi$ such that $r_{\vphi_t}=\sigma_t^* r_{\vphi_0}$ where $\sigma_t=\exp( \log |t|\cdot \eta_0)$. 
\end{thm}
Let $(\cX, \xi_0;\eta)$ be a $\bQ$-Gorenstein test configuration of $X$ with the projection map $\pi: \cX\rightarrow\bC$. Let $X_t:=\pi^{-1}(t)$ be the fiber over $\{t\}$ and $o_t$ the vertex point of $X_t$.
Denote $\cX^\circ=\cX\setminus \{o_t; t\in \bC\}$. 
In the following discussion, we denote by $R^2$ the function obtained by restricting $r^2$, considered as a function on $\bC^N\times \bC$, to $\cX$ via a fixed the equivariant embedding $\cX\rightarrow \bC^N\times\bC$:
$
R^2=\left.r^2\right|_{\cX}.%=\left.\sum_i |z_i|^{2/a_i}\right|_{\cX}.
$

\begin{defn}
Denote by $PSH(\left.\cX\right|_\bD, \xi_0)$ the set of bounded real functions $\Phi$ on $\left.\cX^\circ\right|_{\bD}$ that satisfies:
\begin{enumerate}
\item[(1)] $\tau\circ \Phi=\Phi$ for any $\tau\in T$;
\item[(2)] $R^2_\Phi:=R^2e^\Phi$ is a proper plurisubharmonic function on $\cX|_\bD$. 
\end{enumerate}
\end{defn}
As before, we can think of functions in $PSH(\cX|_\bD, \xi_0)$ as transversal K\"{a}hler potentials on $\cX|_\bD$. If we also denote by $\chi$ the restriction of $\chi=\frac{\sqrt{-1}}{2}(\bar{\partial}-\partial)\log R^2=-\frac{1}{2}Jd\log R^2$ to $\cY:=\{R=1\}\cap \cX$, the we can similarly define $PSH(\cY, \xi_0)$ as Definition \ref{defn-PSHY}.

Moreover, using the equivariant isomorphism
$\iota: \cX|_{\bD^*}\cong X\times \bD^*$, we can associate to any $\Phi\in PSH(\left. \cX\right|_\bD)$ plurisubharmonic function $\vphi$ on $X\times\bD^*$ and hence a path $\vphi_t\in PSH(X, \xi_0)$ such that $R^2_\Phi=\iota^*(r^2_\vphi)$. As an example, the path asssociated to
$\Phi=0$ and is given by $\til{\vphi}_t$.

\begin{prop}\label{prop-GLelong}
Assume $\Phi\in PSH(\cX, \xi_0)$ and let $\vphi_t\in PSH(X, \xi_0)$ be the associated path. Then $G(\vphi_t)$ is subharmonic in $t$ and its Lelong number at $t=0$ is given by 
$1-\lct(\cX, \cX_0)$. 

\end{prop}

\begin{proof}

Since $R_\Phi^2=\iota^*(r^2 e^{\vphi})$ is plurisubharmonic over $\cX|_{\bD^*}\cong X\times \bD^*$. Applying Theorem \ref{thm-affine}, we get $G(\vphi_t)$ is subharmonic in $t$. To see that it's subharmonic over $\bD$, we just need to show that $G(\vphi_t)$ is uniformly bounded from above. Because $\Phi$ bounded, we know that
\[
\left|G(\vphi_t)-G(\til{\vphi}_t)\right|\le C.
\]
So we just need to show that $G(\til{\vphi}_t)$ is uniformly bounded from above. 

Because $\eta$ preserves the global section $s\in |m K_{\cX}|$: $\cL_\eta s=0$. As a consequence, 
$dV_{X_t}=\left.\left(\sqrt{-1}^{mn^2} s\wedge \bar{s}\right)^{1/m}\right|_{X_t}$ satisfies $\sigma_t^*dV_{X_t}=dV_{X_1}=dV_{X}$. So we have:
\begin{eqnarray*}
G(\til{\vphi}_t)=-\log\left(\int_X e^{-\sigma(t)^*r^2}(\sigma^* dV_{X_t})\right)
=-\log\left(\int_{X_t} e^{-r^2}dV_{X_t}\right).
\end{eqnarray*}
Because $\cL_{r\partial_r}dV_{X_t}=2n dV_{X_t}$, we can write $dV_{X_t}=2 r^{2n-1}dr\wedge \Omega_{Y_t}$ and calculate:
\begin{eqnarray}
\int_{X_t}e^{-r^2} dV_{X_t}&=&(n-1)!\int_{Y_t}\Omega_Y\nonumber \\
&=&
C_n\cdot \int_{\{r\le 1\}\cap X_t}e^{-r^2}dV_{X_t}\le C_n \int_{\{r\le 1\}\cap X_t}dV_{X_t},\label{eq-Xtrescale}
\end{eqnarray}
where % $C_n$ is a constant equal to:
$
C_n=\frac{(n-1)!}{\int_0^1e^{-r^2}r^{2n-1}dr^2}.%=\frac{\int_0^{+\infty} e^{-r^2}r^{2n-1}dr}{\int_0^1e^{-r^2}r^{2n-1}dr}=\frac{1}{1-e^{-1}\sum_{i=0}^{n-1}1/i!}.
$

Now the upper boundedness of $G(\til{\vphi}_t)$ can be seen in two ways. For one way, one can resolve the singularity of $\{r\le 1\}\cap (\cX|_{\bD})$ and estimate the integral 
using the method as in \cite[Proof of Lemma 3.7]{Li13} or \cite{BJ16}. The other approximation approach is the following. Recall that $r^2$ is the radius function associated to the vector field $\xi_0=\sum_i a_i z_i \frac{\partial}{\partial z_i}$. Now we choose a sequence of 
vector fields $\xi^{(k)}=\sum_i a_i^{(k)}z_i \frac{\partial}{\partial z_i}$ with $a_i^{(k)}\in \bQ$ and $a_i^{(k)}\rightarrow a_i$ as $k\rightarrow+\infty$. Choose a sequence of new radius function $r^{(k)}=r_{\xi^{(k)}}$ such that $r^{(k)}$ is uniformly $C^0$-comparable to the functions $\sum_{i=1}^N|z_i|^{2/(a_i^{(k)})}$. 
Then there exist $C_1, C_2>0$ such that, for any $\epsilon>0$, we have:
$C_1 (r^{(k)})^{1-\epsilon}\le r\le C_2 (r^{(k)})^{1+\epsilon}$ for $k\gg 1$. So we get:
\begin{eqnarray*}
\int_{\{r\le 1\}\cap X_t}dV_{X_t}\le \int_{\left\{r^{(k)}\le C_1^{(1-\epsilon)^{-1}}\right\}\cap X_t}dV_{X_t}.
\end{eqnarray*}
Because $a_i^{(k)}$ is rational, we can taking quotient of $\cX$ by the $\bC^*$-action generated by $\xi^{(k)}=\sum_i a_i^{(k)}\partial_{z^i}$ and reduces to the log Fano case considered in \cite{Ber15} in which case the upper boundedness of $G(\til{\vphi}_t)$ was shown.

Finally we need to calculate the Lelong number of $G(\til{\vphi}_t)$ with respect to $t$. According to \cite{Ber15}, the Lelong number of $G(t)$ is equal to the infimum of $c$ such that 
\begin{eqnarray*}
\int_U e^{-G-(1-c)\log|t|^2}i d\tau\wedge d\bar{\tau}=\int_{\cX|_\bD}e^{-r^2-(1-c)\log|t|^2}dV_{\cX}<+\infty.
\end{eqnarray*}

We have the following identity:
\begin{equation}\label{eq-intcut}
\int_{\cX|_\bD}e^{-r^2-(1-c)\log|\tau|^2}dV_{\cX}=C_n\cdot \int_{\cX|_\bD\cap \{r\le 1\}}e^{-r^2-(1-c)\log|\tau|^2}dV_{\cX}.
\end{equation}
Because $e^{-1}\le e^{-r^2}\le 1$ is a bounded function, the right-hand-side of \eqref{eq-intcut} is integrable if and only if 
$1-c< \lct(\cX\cap \{r\le 1\}, \cX_0\cap \{r\le 1\})$. Using the rescaling symmetry as used in \eqref{eq-Xtrescale}, we see that
$\lct(\cX\cap \{r\le 1\}, \cX_0\cap \{r\le 1\})=\lct(\cX, \cX_0)$. So we are done.

\end{proof}

Assume $r^2 e^{\vphi_\KE}$ with $\vphi_\KE\in PSH(X, \xi_0)$ is a radius function of a Ricci-flat K\"{a}hler cone metric on $(X, \xi_0)$. Let $(\cX, \xi_0; \eta)$ be a test configuration of $(X, \xi_0)$. We construct \Cyan{a} geodesic ray associated to $(\cX, \xi_0; \eta)$ by solving the homogeneous Monge-Amp\`{e}re equation:
\begin{equation}
(\sddb (R^2e^\Phi))^{n+1}=0 \text{ on } \cX|_\bD, \quad \Phi|_{X\times S^1}=\vphi_\KE.
\end{equation}
Using transversal point of view, this equation is equivalent to the following equation:
\begin{equation}
(d\chi+\sddb\Phi)^n\wedge\chi=0 \text{ on } \cY|_{\bD}, \quad \left.\Phi\right|_{Y\times S^1}=\left.\vphi_\KE\right|_Y.
\end{equation}

By considering the envelope (or its equivalent formulation on $\cX|_\bD$) 
\[
\Phi: =\sup\left\{\Psi\in PSH(\cY|_\bD, \xi_0): \Psi \le \left.\vphi_\KE\right|_Y  \text{ on } \partial (\cY|_\bD)=Y\times S^1 \right\},
\]
then the following result can be proved in exactly the same way as in \cite[Proposition 2.7]{Ber15} by using the transversal K\"{a}hler structures of $(\cY, \xi_0)$. Note that this kind of extension has also been used in \cite{DS17} (see also \cite{CS15, HL18}). 
\begin{prop}[{see \cite[Proposition 2.7]{Ber15}}]
$\Phi$ is locally bounded such that $R^2 e^\Phi$ has positive curvature current and satisfies $(\sddb (R^2 e^\Phi))^{n+1}=0$ on $\cX|_\bD$.
\end{prop}

Finally we can give the proof of Theorem \ref{thm-RF2Ding}.
\begin{proof}[Proof of Theorem \ref{thm-RF2Ding}]
Let $\Phi$ be the geodesic ray emanating from $\vphi_\KE$ that is determined by $(\cX, \xi_0)$. Let $\vphi_t$ be the associated path in $PSH(X, \xi_0)$. Then because $(\sddb (R^2e^\Phi))^{n+1}=0$,
$E(\vphi_t)$ is affine in $t$ by Proposition \ref{prop-hessE}. $G(\vphi_t)$ is subharmonic in $t$ by Proposition \ref{prop-GLelong}. So $D(t):=D(\vphi_t)$ is subharmonic over $\bD$. 
Because $D(t)$ depends only $|t|$, $D(t)$ is convex in $-\log|t|^2$. Because $D(\vphi_t)\ge D(\vphi_{\KE})$ for any $t\in \bD$, we see that $D(t)$ is a non-decreasing function in $-\log|t|^2$.

By Proposition \ref{prop-Easymp} and Proposition \ref{prop-GLelong}, we have:
\begin{eqnarray}\label{eq-Dslope}
\lim_{t\rightarrow 0}\frac{D(t)}{-\log|t|^2}=\frac{D_{\eta}\vol(\xi_0)}{\vol(\xi_0)}-(1-\lct(\cX, \cX_0))=D^\NA(\cX, \xi_0; \eta).
\end{eqnarray}
If $D^\NA(\cX, \xi_0; \eta)=0$, then because $D(t)$ is convex and non-decreasing in $-\log|t|^2$, we see that $D(t)$ is affine and hence $G(\vphi_t)$ is affine. So by Theorem \ref{thm-affine}, there exists holomorphic vector field $\eta_0$ such that
$\vphi_t=(\sigma_t)^*\vphi_\KE$ where $\sigma_t=\exp(\log|t| \eta_0)$.
The rest of the argument is the same as \cite[Proposition 3.3]{Ber15} as extended to the Ricci-flat cone setting in \cite{CS15}.

\end{proof}

\vspace{6mm}
\noindent
Chi Li, Purdue University.   \\
li2285@purdue.edu

\medskip
\noindent
Xiaowei Wang, Rutgers University at Newark.\\ %NJ 07102 
xiaowwan@rutgers.edu

\medskip
\noindent
Chenyang Xu, Beijing International Center for Mathematical Research. \\
cyxu@math.pku.edu.cn\\
\noindent
Department of Mathematics, Massachusetts Institute of Technology. \\
cyxu@math.mit.edu\\
\noindent
Current address: Department of Mathematics, Princeton University. \\
chenyang@princeton.edu

\end{document}